\documentclass[reqno]{amsart}
\usepackage{amssymb}
\usepackage{graphicx}

\usepackage[usenames, dvipsnames]{color}
\usepackage{verbatim}
\usepackage{mathrsfs}
\usepackage{bm}
\usepackage{cite}

\numberwithin{equation}{section}

\newtheorem{theorem}{Theorem}[section]
\newtheorem{corollary}[theorem]{Corollary}
\newtheorem{lemma}[theorem]{Lemma}

\theoremstyle{definition}
\newtheorem{remark}[theorem]{Remark}

\theoremstyle{definition}

\theoremstyle{definition}

\makeatletter
\def\dashint{\operatorname%
{\,\,\text{\bf-}\kern-.98em\DOTSI\intop\ilimits@\!\!}}
\makeatother

\def\\det{\text{det}}

\def\.5{\frac{1}{2}}

\newcommand{\RN}[1]{%
  \textup{\uppercase\expandafter{\romannumeral#1}}%
}

\renewcommand{\epsilon}{\varepsilon}

\newcounter{marnote}

\begin{document}
\title[Singular analysis of the stress concentration]{Singular analysis of the stress concentration in the narrow regions between the inclusions and the matrix boundary}

\author[C.X. Miao]{Changxing Miao}
\address[C.X. Miao]{1. Beijing Computational Science Research Center, Beijing 100193, China.}
\address{2. Institute of Applied Physics and Computational Mathematics, P.O. Box 8009, Beijing, 100088, China.}
\email{miao\_changxing@iapcm.ac.cn}

\author[Z.W. Zhao]{Zhiwen Zhao}

\address[Z.W. Zhao]{Beijing Computational Science Research Center, Beijing 100193, China.}
\email{zwzhao365@163.com}


\date{\today} 



\begin{abstract}
We consider the Lam\'{e} system arising from high-contrast composite materials whose inclusions (fibers) are nearly touching the matrix boundary. The stress, which is the gradient of the solution, always concentrates highly in the narrow regions between the inclusions and the external boundary. This paper aims to provide a complete characterization in terms of the singularities of the stress concentration by accurately capturing all the blow-up factor matrices and making clear the dependence on the Lam\'{e} constants and the curvature parameters of geometry. Moreover, the precise asymptotic expansions of the stress concentration are also presented in the presence of a strictly convex inclusion close to touching the external boundary for the convenience of application.
\end{abstract}

\maketitle



\section{Introduction}

In composite structure comprising of a matrix and inclusions (fibers), it is common for inclusions in close proximity to each other or close to touching the matrix boundary. The mathematical problem can be described by the Lam\'{e} system. The principal quantity of interest from the perspective of engineering is the stress, which is the gradient of a solution to the Lam\'{e} system. It is well known that the stress may blow up as the distance $\varepsilon$ between inclusions or between the inclusion and the matrix boundary approaches to zero. Especially the latter exhibits more complex singular behavior due to the interaction from the boundary data.

This work is devoted to the investigation of the blow-up phenomena occurring in high-contrast fiber-reinforced composites, which is stimulated by the great work of Babu\u{s}ka et al. \cite{BASL1999}, where the Lam\'{e} system was utilized and they observed numerically that the size of the strain tensor stays bounded independent of the distance $\varepsilon$ between adjacent inclusions. Subsequently, Bonnetier and Vogelius \cite{BV2000} considered the scalar equation in the conductivity problem and proved that the gradient of a solution to the conductivity equation with piecewise constant coefficients remains bounded for two touching disks. Li and Vogelius \cite{LV2000} then studied the general divergence form second order elliptic equation with piecewise smooth coefficients and obtained the boundness of the gradient with its upper bound independent of $\varepsilon$ for the inclusions of arbitrary smooth shape. In subsequent work \cite{LN2003}, Li and Nirenberg further extended the $\varepsilon$-independent gradient estimates in \cite{LV2000} to general divergence form second order elliptic systems, especially covering systems of elasticity. Their results also demonstrate the numerical observation in \cite{BASL1999}. Dong and Li \cite{DL2019} recently made clear the dependence on the elliptic coefficients and the distance between two disks in the optimal gradient estimates and established more interesting higher-order derivative estimates for the isotropic conductivity problem. This, in particular, answered open problems $(b)$ and $(c)$ proposed by Li and Vogelius \cite{LV2000}. However, for the inclusions of general smooth shape and second order elliptic systems of divergence form, it remains to be solved. We refer to page 894 of \cite{LN2003} for these open problems. In addition, Kim and Lim \cite{KL2019} made use of the single and double layer potentials with image line charges to establish an asymptotic expression for the solution to the conductivity problem in the presence of core-shell structure with circular boundaries. Calo, Efendiev and Galvis \cite{CEG2014} obtained an asymptotic formula of a solution to the second-order elliptic equations of divergence form in the presence of high-conductivity inclusions or low-conductivity inclusions.

The elliptic coefficients considered in aforementioned work are assumed to be away from $0$ and $\infty$. The situation will become very different when the elliptic constants are allowed to deteriorate. For example, the scalar conductivity problem with finite conductivity $k$ turns into the perfect conductivity problem as the contrast $k$ degenerates to be $\infty$. It has been shown in various literature that the electric field, which is the gradient of a solution to the perfect conductivity equation, generally blows up as the distance $\varepsilon$ between inclusions or between the inclusions and the matrix boundary goes to zero. Its blow-up rate has been proved to be $\varepsilon^{-1/2}$ in two dimensions \cite{AKLLL2007,BC1984,BLY2009,AKL2005,Y2007,Y2009,K1993}, $(\varepsilon|\ln\varepsilon|)^{-1}$ in three dimensions \cite{BLY2009,LY2009,BLY2010,L2012}, and $\varepsilon^{-1}$ in dimensions greater than three \cite{BLY2009}, respectively. Further, more precise characterization in terms of the singularities of the electric field concentration have been established by Ammari et al. \cite{ACKLY2013}, Bonnetier and Triki \cite{BT2013}, Kang et al. \cite{KLY2013,KLY2014,KLY2015}, Li et al. \cite{LLY2019,Li2020}. Recently, the study on the singular behavior of the gradient has been generalized to the nonlinear $p$-Laplacian \cite{GN2012,G2015} and the Finsler $p$-Laplacian \cite{CS2019,CS20192}. Their method of barriers is purely nonlinear, which differs from the ones utilized in the linear case.

However, there is significant difficulty in extending the results in the perfect conductivity problem to the full elasticity. For example, the maximum principle does not hold for the systems. A delicate iterate technique was then built in \cite{LLBY2014} to overcome this difficulty, where Li et al. \cite{LLBY2014} obtained $C^{k}$ estimates for a class of elliptic systems. After that, this iterate scheme was applied to the investigation on the blow-up behavior of the gradient of a solution to the Lam\'{e} system with partially infinity coefficients. Bao, Li and Li \cite{BLL2015,BLL2017} firstly established the pointwise upper bounds on the gradient for two strictly convex inclusions. Their results indicated that the gradient blow-up rate under vectorial case is consistent with that under scalar case above. The subsequent work \cite{L2018} provided a lower bound on the gradient, which demonstrates the optimality of the blow-up rate in dimensions two and three. It is worth emphasizing that Kang and Yu \cite{KY2019} recently introduced singular functions constructed by nuclei of strain to give a precise description for the stress concentration in two dimensions. The mathematical approaches in \cite{KY2019} are based on the layer potential techniques and the variational principle, which are different from the iterate scheme above. Besides the aforementioned interior estimates, there is another direction of research to establish the boundary estimates \cite{BJL2017,LZ2020}. For one strictly convex inclusion close to touching the matrix boundary, Bao, Ju and Li \cite{BJL2017} obtained the optimal upper and lower bounds on the gradient. In particular, the lower bound on the gradient in \cite{BJL2017} was constructed by finding a blow-up factor which is a linear functional in relation to the boundary data. Subsequently, Li and Zhao \cite{LZ2020} extended to the general $m$-convex inclusions and found that some special boundary data with $k$-order growth will strengthen the singularities of the stress. This is a novel blow-up phenomenon induced by the boundary data. Although the optimal blow-up rate of the stress is derived in \cite{BJL2017,LZ2020}, the dependence on the Lam\'{e} constants and the curvature parameters of geometry is not explicit. Moreover, they only captured a single blow-up factor in the construction of the lower bound of the gradient but not the whole blow-up factor matrices, which determine essentially whether the blow-up will occur or not. By contrast with the results in \cite{BJL2017,LZ2020}, the novelty of this paper lies in capturing all the blow-up factor matrices and meanwhile showing the explicit dependence on the Lam\'{e} coefficients and the curvature parameters of geometry in the optimal upper and lower bounds on the gradient. Further, these blow-up factor matrices can be applied to establish the corresponding asymptotic expansions of the gradient.

The outline of this paper is as follows. Section \ref{SL01} is to describe the elasticity problem and state the main results. In Section \ref{SL003}, we consider a general boundary value problem \eqref{P2.008} and give the leading term for the gradient of the solution. Section \ref{SL004} is dedicated to the proofs of Theorems \ref{Lthm066} and \ref{thma002}. The asymptotic behavior of the stress concentration for a strictly convex inclusion close to touching the matrix boundary is analyzed in Section \ref{SEC005}.
\section{Problem formulation and Main results}\label{SL01}
\subsection{The elasticity problem}
In this paper, the Lam\'{e} system in linear elasticity is considered. Let $D\subset\mathbb{R}^{d}\,(d\geq2)$ be a bounded domain with $C^{2,\alpha}$ boundary, and $D_{1}^{\ast}$ be a convex open set in $D$ with $C^{2,\alpha}$ boundary, $0<\alpha<1$, which touches the external boundary $\partial D$ at one point. By a translation and rotation of the coordinates, if necessary, let
\begin{align*}
\partial D_{1}^{\ast}\cap\partial D=\{0\}\in\mathbb{R}^{d},\quad D_{1}^{\ast}\subset\{(x',x_{d})\in\mathbb{R}^{d}\,|\,x_{d}>0\}.
\end{align*}
Here and below, we denote ($d-1$)-dimensional variables and domains by adding superscript prime. After moving upward $D_{1}^{\ast}$ along $x_{d}$-axis by a arbitrarily small constant $\varepsilon>0$, we get
\begin{align*}
D_{1}^{\varepsilon}:=D_{1}^{\ast}+(0',\varepsilon).
\end{align*}
For simplicity, let
\begin{align*}
D_{1}:=D_{1}^{\varepsilon},\quad\mathrm{and}\;\Omega:=D\setminus\overline{D}_{1}.
\end{align*}

Suppose that $\Omega$ and $D_{1}$ are filled with two different isotropic and homogeneous elastic materials with different Lam\'{e} constants $(\lambda,\mu)$ and $(\lambda_{1},\mu_{1})$, respectively. The elasticity tensors $\mathbb{C}^0$ and $\mathbb{C}^1$ for the background $\Omega$ and the inclusion $D_{1}$ can be written, respectively, as
$$C_{ijkl}^0=\lambda\delta_{ij}\delta_{kl} +\mu(\delta_{ik}\delta_{jl}+\delta_{il}\delta_{jk}),$$
and
$$C_{ijkl}^1=\lambda_1\delta_{ij}\delta_{kl} +\mu_1(\delta_{ik}\delta_{jl}+\delta_{il}\delta_{jk}),$$
where $i,j,k,l=1,2,...,d$ and $\delta_{ij}$ denotes the kronecker symbol: $\delta_{ij}=0$ for $i\neq j$, $\delta_{ij}=1$ for $i=j$. For a given boundary data $\varphi=(\varphi^{1},\varphi^{2},...,\varphi^{d})^{T}$, we consider the following Dirichlet problem for the Lam$\mathrm{\acute{e}}$ system with piecewise constant coefficients
\begin{align}\label{La.001}
\begin{cases}
\nabla\cdot \left((\chi_{\Omega}\mathbb{C}^0+\chi_{D_{1}}\mathbb{C}^1)e(u)\right)=0,&\hbox{in}~D,\\
u=\varphi, &\hbox{on}~\partial{D},
\end{cases}
\end{align}
where $u=(u^{1},u^{2},...,u^{d})^{T}:D\rightarrow\mathbb{R}^{d}$ represents the elastic displacement field, $e(u)=\frac{1}{2}\left(\nabla u+(\nabla u)^{T}\right)$ denotes the elastic strain, $\chi_{\Omega}$ and $\chi_{D_{1}}$ are the characteristic functions of $\Omega$ and $D_{1}$, respectively. As shown in \cite{BLL2015}, if the strong ellipticity condition
\begin{align*}
\mu>0,\quad d\lambda+2\mu>0,\quad \mu_1>0,\quad d\lambda_1+2\mu_1>0
\end{align*}
holds, then there exists a unique solution $u\in H^{1}(D;\mathbb{R}^{d})$ of the Dirichlet problem (\ref{La.001}) for given $\varphi\in H^{1}( D;\mathbb{R}^{d})$.

Denote the linear space of rigid displacement in $\mathbb{R}^{d}$ by
$$\Psi:=\{\psi\in C^1(\mathbb{R}^{d}; \mathbb{R}^{d})\ |\ \nabla\psi+(\nabla\psi)^T=0\}.$$
A base of $\Psi$ is written as
\begin{align}\label{OPP}
\{\psi_{\alpha}\}_{1}^{\frac{d(d+1)}{2}}:=\left\{\;e_{i},\;x_{k}e_{j}-x_{j}e_{k}\;\big|\;1\leq\,i\leq\,d,\;1\leq\,j<k\leq\,d\;\right\},
\end{align}
where $\{e_{1},e_{2},...,e_{d}\}$ is the standard basis in $\mathbb{R}^{d}$. For the convenience of computations, we adopt the following order with respect to $\psi_{\alpha}$: $\psi_{\alpha}=e_{\alpha}$ if $\alpha=1,2,...,d$; $\psi_{\alpha}=x_{d}e_{\alpha-d}-x_{\alpha-d}e_{d}$ if $\alpha=d+1,...,2d-1$; if $\alpha=2d,...,\frac{d(d+2)}{2}\,(d\geq3)$, there exist two indices $1\leq i_{\alpha}<j_{\alpha}<d$ such that
$\psi_{\alpha}=(0,...,0,x_{j_{\alpha}},0,...,0,-x_{i_{\alpha}},0,...,0)$.

It has been proved in the Appendix of \cite{BLL2015} that for fixed $\lambda$ and $\mu$,
\begin{align*}
u_{\lambda_1,\mu_1}\rightarrow u\quad\hbox{in}\ H^1(D; \mathbb{R}^{d}),\quad \hbox{as}\ \min\{\mu_1, d\lambda_1+2\mu_1\}\rightarrow\infty,
\end{align*}
where $u_{\lambda_{1},\mu_{1}}$ is the solution of \eqref{La.001} and $u\in H^1(D; \mathbb{R}^{d})$ verifies
\begin{align}\label{La.002}
\begin{cases}
\mathcal{L}_{\lambda, \mu}u:=\nabla\cdot(\mathbb{C}^0e(u))=0,\quad&\hbox{in}\ \Omega,\\
u=C^{\alpha}\psi_{\alpha},&\hbox{on}\ \partial{D}_{1},\\
\int_{\partial{D}_{1}}\frac{\partial u}{\partial \nu_0}\Big|_{+}\cdot\psi_{\alpha}=0,&\alpha=1,2,...,\frac{d(d+1)}{2},\\
u=\varphi,&\hbox{on}\ \partial{D}.
\end{cases}
\end{align}
Here the free constants $C^{\alpha}$, $\alpha=1,2,...,\frac{d(d+1)}{2}$ are determined by the third line of \eqref{La.002} and the co-normal derivative is given by
\begin{align*}
\frac{\partial u}{\partial \nu_0}\Big|_{+}&:=(\mathbb{C}^0e(u))\nu=\lambda(\nabla\cdot u)\nu+\mu(\nabla u+(\nabla u)^T)\nu,
\end{align*}
where $\nu$ is the unit outer normal of $\partial D_{1}$ and the subscript $+$ denotes the limit from outside the domain. We here would like to remark that the existence, uniqueness and regularity of weak solutions to problem \eqref{La.002} have also been established in the Appendix of \cite{BLL2015}. Furthermore, the $H^{1}$-solution $u$ of problem (\ref{La.002}) was improved to be of $C^1(\overline{\Omega};\mathbb{R}^{d})\cap C^1(\overline{D}_{1};\mathbb{R}^{d})$ for any $C^{2,\alpha}$-domain.

Suppose that $\partial D_{1}$ and $\partial D$ near the origin are, respectively, the graphs of two $C^{2,\alpha}$ functions $\varepsilon+h_{1}$ and $h$ with respect to $x'$. To be specific, for some positive constant $R$, independent of $\varepsilon$, let $h_{1}$ and $h$ satisfy that
\begin{enumerate}
{\it\item[(\bf{H1})]
$\kappa_{1}|x'|^{m}\leq h_{1}(x')-h(x')\leq\kappa_{2}|x'|^{m},\;\mbox{if}\;\,x'\in B'_{2R},$
\item[(\bf{H2})]
$|\nabla_{x'}^{j}h_{1}(x')|,\,|\nabla_{x'}^{j}h(x')|\leq \kappa_{3}|x'|^{m-j},\;\mbox{if}\;\,x'\in B_{2R}',\;j=1,2,$
\item[(\bf{H3})]
$\|h_{1}\|_{C^{2,\alpha}(B'_{2R})}+\|h\|_{C^{2,\alpha}(B'_{2R})}\leq \kappa_{4},$}
\end{enumerate}
where $m\geq2$ is an integer and $\kappa_{i}$, $i=1,2,3,4$ are all positive constants independent of $\varepsilon$. We additionally assume that $h_{1}-h$ is even with respect to each $x_{i}$ in $B_{R}'$, $i=1,...,d-1$. It is worth emphasizing that the inclusions under conditions ({\bf{H1}})--({\bf{H3}}) actually contain the strictly convex inclusions, which were extensively studied in previous work \cite{BLL2015,BLL2017,BJL2017,KY2019}.

For $z'\in B'_{R}$ and $0<t\leq2R$, define the thin gap
\begin{align*}
\Omega_{t}(z'):=&\left\{x\in \mathbb{R}^{d}~\big|~h(x')<x_{d}<\varepsilon+h_{1}(x'),~|x'-z'|<{t}\right\}.
\end{align*}
For notational simplicity we adopt the abbreviated notation $\Omega_{t}$ to represent $\Omega_{t}(0')$. Its top and bottom boundaries can be, respectively, written by
\begin{align*}
\Gamma^{+}_{t}:=\left\{x\in\mathbb{R}^{d}|\,x_{d}=\varepsilon+h_{1}(x'),\;|x'|<t\right\},~~
\Gamma^{-}_{t}:=\left\{x\in\mathbb{R}^{d}|\,x_{d}=h(x'),\;|x'|<t\right\}.
\end{align*}

Introduce a scalar auxiliary function $\bar{v}\in C^{2}(\mathbb{R}^{d})$ such that $\bar{v}=1$ on $\partial D_{1}$, $\bar{v}=0$ on $\partial D$,
\begin{align}\label{BATL001}
\bar{v}(x',x_{d}):=\frac{x_{d}-h(x')}{\varepsilon+h_{1}(x')-h(x')},\;\,\mathrm{in}\;\Omega_{2R},\quad\mbox{and}~\|\bar{v}\|_{C^{2}(\Omega\setminus\Omega_{R})}\leq C.
\end{align}
Denote
\begin{align}\label{deta}
\delta(x'):=\varepsilon+h_{1}(x')-h(x'),\quad f(\bar{v}):=\frac{1}{2}\left(\bar{v}-\frac{1}{2}\right)^{2}-\frac{1}{8}.
\end{align}
For $x\in\Omega_{2R}$, we define a family of vector-valued auxiliary functions as follows:
\begin{align}\label{CAN01}
\bar{u}_{0}=&\varphi(x',h(x'))(1-\bar{v}(x',x_{d}))+\mathcal{F}_{0},
\end{align}
and
\begin{align}\label{OPQ}
\bar{u}_{\alpha}(x',x_{d})&=\psi_{\alpha}\bar{v}(x',x_{d})+\mathcal{F}_{\alpha},\quad\alpha=1,2,...,\frac{d(d+1)}{2},
\end{align}
where $\psi_{\alpha}$, $\alpha=1,2,...,\frac{d(d+1)}{2}$, are defined in (\ref{OPP}) and
\begin{align}\label{OPQ1}
\mathcal{F}_{0}=&-\frac{\lambda+\mu}{\mu}f(\bar{v})\varphi^{d}(x',h(x'))\sum^{d-1}_{i=1}\partial_{x_{i}}\delta\,e_{i}\notag\\
&-\frac{\lambda+\mu}{\lambda+2\mu}f(\bar{v})\sum^{d-1}_{i=1}\varphi^{i}(x',h(x'))\partial_{x_{i}}\delta\,e_{d},
\end{align}
and, for $\alpha=1,2,...,\frac{d(d+1)}{2}$,
\begin{align}\label{OPAKLN01}
\mathcal{F}_{\alpha}=&\frac{\lambda+\mu}{\mu}f(\bar{v})\psi^{d}_{\alpha}\sum^{d-1}_{i=1}\partial_{x_{i}}\delta\,e_{i}+\frac{\lambda+\mu}{\lambda+2\mu}f(\bar{v})\sum^{d-1}_{i=1}\psi^{i}_{\alpha}\partial_{x_{i}}\delta\,e_{d}.
\end{align}
We here would like to point out that the correction terms $\mathcal{F}_{\alpha}$, $\alpha=1,2,...,d$ were captured in the previous work \cite{LX2020}. In this paper, we further find the correction term for a more general boundary value problem \eqref{P2.008}, see Theorem \ref{thm8698} below.

\subsection{Main results}

To begin with, we first recall the following three types of boundary data introduced in \cite{LZ2020}, which are classified according to the parity. To be specific, let $\varphi=(\varphi^{1},...,\varphi^{d})\neq0$ on $\Gamma^{-}_{R}$ satisfying that for $x\in\Gamma^{-}_{R}$,
\begin{itemize}
{\it
\item[(\bf{A1})] for $i=1,2,...,d$, $j=1,...,d-1$, $\varphi^{i}(x)$ is an even function of each $x_{j}$;
\item[(\bf{A2})] if $d=2$, for $i=1,2$, $\varphi^{i}(x)$ is odd with respect to $x_{1}$; if $d\geq3$, for $i=1,...,d-1$, $\varphi^{i}(x)$ is odd with respect to some $x_{j_{i}}$, $j_{i}\in\{1,...,d-1\}$, and $\varphi^{d}(x)$ is odd with respect to $x_{1}$ and even with respect to each $x_{j}$, $j=2,...,d-1$;
\item[(\bf{A3})] if $d=2$, $\varphi^{1}(x)$ is odd with respect to $x_{1}$, and $\varphi^{2}(x)=0$; if $d\geq3$, for $i=1,...,d-1$, $\varphi^{i}(x)$ is odd with respect to $x_{i}$, and $\varphi^{d}(x)$ is odd with respect to $x_{1}$ and $x_{2}$, respectively.}
\end{itemize}
For the convenience of notations, introduce the blow-up rate indices as follows: for $i=0,2$ and $i=k,k+1$, $k\geq2$,
\begin{align}\label{rate}
\rho_{i}(d,m;\varepsilon):=&
\begin{cases}
\varepsilon^{\frac{d+i-1}{m}-1},&m>d+i-1,\\
|\ln\varepsilon|,&m=d+i-1,\\
1,&m<d+i-1.
\end{cases}
\end{align}
Under these three types of boundary data, Li and Zhao \cite{LZ2020} obtained the following results.

{\bf Theorem A (Corollary 1.6 of \cite{LZ2020}).}
Assume that $D_{1}\subset D\subseteq\mathbb{R}^{d}\,(d\geq2)$ are defined as above, conditions ({\bf{H1}})--({\bf{H3}}) hold. Let $u\in H^{1}(D;\mathbb{R}^{d})\cap C^{1}(\overline{\Omega};\mathbb{R}^{d})$ be the solution of \eqref{La.002}. Assume that one of ({\bf{A1}}), ({\bf{A2}}) and ({\bf{A3}}) holds. If $\varphi\in C^{2}(\partial D;\mathbb{R}^{d})$ satisfies the $k$-order growth condition,
\begin{align}\label{growth}
|\varphi(x)|\leq \eta\,|x|^{k},\quad\;\,\mathrm{on}\;\Gamma^{-}_{R},
\end{align}
for some integer $k>0$ and a positive constant $\eta$. Then for a sufficiently small $\varepsilon>0$,
\begin{align}\label{DALN666}
|\nabla u(x)|\leq&\frac{C}{\varepsilon+|x'|^{m}}\left[\eta\rho_{A}(\varepsilon)+\frac{\|\varphi\|_{C^{2}(\partial D)}}{\rho_{0}(d,m;\varepsilon)}+|x'|\left(\eta\rho_{B}(\varepsilon)+\frac{\|\varphi\|_{C^{2}(\partial D)}}{\rho_{2}(d,m;\varepsilon)}\right)\right]\notag\\
&+\frac{\eta|x'|^{k}}{\varepsilon+|x'|^{m}}+C\|\varphi\|_{C^{2}(\partial D)},\quad\;\,x\in\Omega_{R},
\end{align}
where
\begin{align}
\rho_{A}(\varepsilon)=&
\begin{cases}
\rho_{k}(d,m;\varepsilon)/\rho_{0}(n,m;\varepsilon),&\text{for\;case\;({\bf{A1}})},\\
1/\rho_{0}(d,m;\varepsilon),&\text{otherwise},
\end{cases}\label{LT001}\\
\rho_{B}(\varepsilon)=&\begin{cases}
\rho_{k+1}(d,m;\varepsilon)/\rho_{2}(n,m;\varepsilon),&\text{for\;case\;({\bf{A2}})},\\
1/\rho_{2}(d,m;\varepsilon),&\text{otherwise}.
\end{cases}\label{LT002}
\end{align}

From \eqref{DALN666}, we see that the singular behavior of $|\nabla u|$ is determined by the following three parts: $\rho_{A}(\varepsilon)\varepsilon^{-1}$, $\rho_{B}(\varepsilon)\varepsilon^{1/m-1}$ and $\varepsilon^{k/m-1}\,(m>k)$. We point out that the blow-up rate $\rho_{A}(\varepsilon)\varepsilon^{-1}$ is achieved at the $(d-1)$-dimensional ball $\{|x'|\leq\sqrt[m]{\varepsilon}\}\cap\Omega$, while, the latter two blow-up rates are generated on the cylinder surface $\{|x'|=\sqrt[m]{\varepsilon}\}\cap\Omega$. Furthermore, it follows from \eqref{DALN666} that

$(a)$ if condition ({\bf{A1}}) holds, then for $m\leq d$, $k\geq1$ or $m\geq d+k$, $k\geq1$, $|\nabla u|$ blows up at the rate of $\frac{\rho_{k}(d,m;\varepsilon)}{\varepsilon\rho_{0}(d,m;\varepsilon)}$, while, for $d+k>m>d,\,k>1$, the blow-up rate of $|\nabla u|$ is $\frac{1}{\varepsilon^{1-1/m}\rho_{2}(d,m;\varepsilon)}$;

$(b)$ if condition ({\bf{A2}}) holds, then for $m\leq d$, $k\geq1$, $|\nabla u|$ blows up at the rate of $\frac{1}{\varepsilon\rho_{0}(d,m;\varepsilon)}$, while, for $m>d$, $k\geq1$, its blow-up rate is $\frac{\rho_{k+1}(d,m;\varepsilon)}{\varepsilon^{1-1/m}\rho_{2}(d,m;\varepsilon)}$. In particular, when $m=d+k,\,k=1$ or $m>d+k,\,k\geq1$, we have $\frac{\rho_{k+1}(d,m;\varepsilon)}{\varepsilon^{1-1/m}\rho_{2}(d,m;\varepsilon)}=\varepsilon^{k/m-1}$;

$(c)$ if condition ({\bf{A3}}) holds, then we obtain that for $m\leq d$, $k\geq1$, the blow-up rate of $|\nabla u|$ is  $\frac{1}{\varepsilon\rho_{0}(n,m;\varepsilon)}$; $d+k>m>d,\,k>1$, $|\nabla u|$ blows up at the rate of $\frac{1}{\varepsilon^{1-1/m}\rho_{2}(d,m;\varepsilon)}$; for $m\geq d+k,\,k\geq1$, its blow-up rate is $\frac{1}{\varepsilon^{1-k/m}}$. Especially when $m=d+k,\,k>1$, we get $\frac{1}{\varepsilon^{1-1/m}\rho_{2}(d,m;\varepsilon)}=\frac{1}{\varepsilon^{1-k/m}}$.

To show the optimality of the blow-up rates summarized in $(a)$--$(c)$ above, we select three special examples from conditions ({\bf{A1}})--({\bf{A3}}) to establish the optimal upper and lower bounds on the blow-up rate of the gradient and meanwhile extract the accurate information of the boundary data $\varphi$, the Lam\'{e} constants $\lambda$ and $\mu$, and the curvature parameters $\tau_{1}$ and $\tau_{2}$ from the constant $C$ in \eqref{DALN666}. Specifically, suppose that for $x\in\Gamma^{-}_{R}$,
\begin{itemize}
{\it
\item[({\bf{E1}})] $\varphi^{i}=-\eta|x'|^{k}$, $i=1,2,...,d$,
\item[({\bf{E2}})] $\varphi^{i}=0$, $i=1,...,d-1$, $\varphi^{d}=\eta x_{1}|x_{1}|^{k-1}$,
\item[({\bf{E3}})] $\varphi^{i}=\eta x_{i}|x_{i}|^{k-1}$, $i=1,...,d-1$, $\varphi^{d}=0$,}
\end{itemize}
where $\eta$ is a positive constant and $k$ is a positive integer. To ensure that $\varphi\in C^{2}(\partial D)$, if condition ({\bf{E1}}) holds, we consider $k\geq2$; if condition ({\bf{E2}}) or ({\bf{E3}}) holds, we consider $k\geq1,\,k\neq2$.

Let $\Omega^{\ast}:=D\setminus \overline{D_{1}^{\ast}}$. For $\alpha,\beta=1,2,...,\frac{d(d+1)}{2}$, define
\begin{align}\label{FNCL001}
a_{\alpha\beta}^{\ast}:=\int_{\Omega^{\ast}}(\mathbb{C}^0e(u_{\alpha}^{\ast}), e(u_{\beta}^{\ast}))dx,\quad Q_{\alpha}^{\ast}[\varphi]:=\int_{\partial D_{1}^{\ast}}\frac{\partial u_{0}^{\ast}}{\partial\nu_{0}}\Big|_{+}\cdot\psi_{\alpha},
\end{align}
where $\varphi\in C^{2}(\partial D;\mathbb{R}^{d})$ and $u_{\alpha}^{\ast}\in{C}^{2}(\Omega^{\ast};\mathbb{R}^d)$, $\alpha=0,1,...,\frac{d(d+1)}{2}$, respectively, verify
\begin{equation}\label{l03.001}
\begin{cases}
\mathcal{L}_{\lambda,\mu}u_{0}^{\ast}=0,&\mathrm{in}\;\Omega^{\ast},\\
u_{0}^{\ast}=0,&\mathrm{on}\;\partial D_{1}^{\ast}\setminus\{0\},\\
u_{0}^{\ast}=\varphi(x),&\mathrm{on}\;\partial D,
\end{cases}\quad
\begin{cases}
\mathcal{L}_{\lambda,\mu}u_{\alpha}^{\ast}=0,&\mathrm{in}~\Omega^{\ast},\\
u_{\alpha}^{\ast}=\psi_{\alpha},&\mathrm{on}~\partial{D}_{1}^{\ast}\setminus\{0\},\\
u_{\alpha}^{\ast}=0,&\mathrm{on}~\partial{D}.
\end{cases}
\end{equation}
We would like to remark that the definitions of $a_{\alpha\beta}^{\ast}$ and $Q_{\alpha}^{\ast}[\varphi]$ are valid only in some cases, see Lemmas \ref{KM323} and \ref{lemmabc} below. Suppose that for some $\kappa_{5}>0$,
\begin{align}\label{ACDT001}
\kappa_{5}\leq\mu,d\lambda+2\mu\leq\frac{1}{\kappa_{5}}.
\end{align}
Introduce the Lam\'{e} constants $\mathcal{L}_{d}^{\alpha}$, $\alpha=1,2,...,\frac{d(d+1)}{2}$ as follows:
\begin{align}
&(\mathcal{L}_{2}^{1},\mathcal{L}_{2}^{2},\mathcal{L}_{2}^{3})=(\mu,\lambda+2\mu,\lambda+2\mu),\quad d=2,\label{AZ}\\
&(\mathcal{L}_{d}^{1},...,\mathcal{L}_{d}^{d-1},\mathcal{L}_{d}^{d},...,\mathcal{L}_{d}^{2d-1},\mathcal{L}_{d}^{2d},...,\mathcal{L}_{d}^{\frac{d(d+1)}{2}})\notag\\
&=(\mu,...,\mu,\lambda+2\mu,...,\lambda+2\mu,2\mu,...,2\mu),\quad d\geq3.\label{AZ110}
\end{align}

Before stating our main results, we first introduce the blow-up factor matrices as follows:
\begin{gather*}\mathbb{A}^{\ast}=\begin{pmatrix} a_{11}^{\ast}&\cdots&a_{1d}^{\ast} \\\\ \vdots&\ddots&\vdots\\\\a_{d1}^{\ast}&\cdots&a_{dd}^{\ast}\end{pmatrix}  ,\;\,
\mathbb{B}^{\ast}=\begin{pmatrix}a_{1\,d+1}^{\ast}&\cdots&a_{1\,\frac{d(d+1)}{2}}^{\ast} \\\\ \vdots&\ddots&\vdots\\\\ a_{d\,d+1}^{\ast}&\cdots&a_{d\,\frac{d(d+1)}{2}}^{\ast}\end{pmatrix},
\end{gather*}
\begin{gather}
\mathbb{C}^{\ast}=\begin{pmatrix}a_{d+1\,1}^{\ast}&\cdots&a_{d+1\,d}^{\ast} \\\\ \vdots&\ddots&\vdots\\\\ a_{\frac{d(d+1)}{2}\,1}^{\ast}&\cdots&a_{\frac{d(d+1)}{2}\,\frac{d(d+1)}{2}}^{\ast}\end{pmatrix},\notag\\
\mathbb{D}^{\ast}=\begin{pmatrix} a_{d+1\,d+1}^{\ast}&\cdots&a_{d+1\,\frac{d(d+1)}{2}}^{\ast} \\\\ \vdots&\ddots&\vdots\\\\ a_{\frac{d(d+1)}{2}\,d+1}^{\ast}&\cdots&a_{\frac{d(d+1)}{2}\,\frac{d(d+1)}{2}}^{\ast}\end{pmatrix}.\label{MTABZR001}
\end{gather}
Write
\begin{align}\label{KR001}
\mathbb{F}^{\ast}=\begin{pmatrix} \mathbb{A}^{\ast}&\mathbb{B}^{\ast} \\  \mathbb{C}^{\ast}&\mathbb{D}^{\ast}
\end{pmatrix}.
\end{align}
For $\alpha=1,2,...,d,$ define
\begin{gather}\label{MATA001}
\mathbb{F}_{1}^{\ast\alpha}[\varphi]:=
\begin{pmatrix} Q_{\alpha}^{\ast}[\varphi]&a_{\alpha\,d+1}^{\ast}&\cdots&a_{\alpha\,\frac{d(d+1)}{2}}^{\ast} \\ Q_{d+1}^{\ast}[\varphi]&a^{\ast}_{d+1\,d+1}&\cdots&a^{\ast}_{d+1\,\frac{d(d+1)}{2}} \\ \vdots&\vdots&\ddots&\vdots\\Q^{\ast}_{\frac{d(d+1)}{2}}[\varphi]&a^{\ast}_{\frac{d(d+1)}{2}\,d+1}&\cdots&a^{\ast}_{\frac{d(d+1)}{2}\,\frac{d(d+1)}{2}}
\end{pmatrix}.
\end{gather}
For $\alpha=1,2,...,\frac{d(d+1)}{2}$, by substituting the column vector $\big(Q_{1}^{\ast}[\varphi],Q_{2}^{\ast}[\varphi],...,Q_{\frac{d(d+1)}{2}}^{\ast}[\varphi]\big)^{T}$ for the elements of $\alpha$-th column in the matrix $\mathbb{F}^{\ast}$, we generate the new matrix $\mathbb{F}_{3}^{\ast\alpha}[\varphi]$ as follows:
\begin{gather}\label{LRNMA001}
\mathbb{F}_{3}^{\ast\alpha}[\varphi]=:
\begin{pmatrix}
a_{11}^{\ast}&\cdots&Q_{1}^{\ast}[\varphi]&\cdots&a_{1\,\frac{d(d+1)}{2}}^{\ast} \\\\ \vdots&\ddots&\vdots&\ddots&\vdots\\\\a_{\frac{d(d+1)}{2}\,1}^{\ast}&\cdots&Q_{\frac{d(d+1)}{2}}^{\ast}[\varphi]&\cdots&a_{\frac{d(d+1)}{2}\frac{d(d+1)}{2}}^{\ast}
\end{pmatrix}.
\end{gather}

Here and throughout this paper, $a\lesssim b$ (or $a\gtrsim b$) represents $a\leq Cb$ (or $a\geq \frac{1}{C}b$) for some positive constant $C$, depending only on $d,R,\kappa_{3},\kappa_{4},\kappa_{5}$ and the $C^{2,\alpha}$ norms of $\partial D_{1}$ and $\partial D$, but not on the distance parameter $\varepsilon$, the boundary data $\varphi$, the curvature parameters $\kappa_{1}$, $\kappa_{2}$ and the Lam\'{e} constants $\mathcal{L}_{d}^{\alpha}$, $\alpha=1,2,...,\frac{d(d+1)}{2}$. Without loss of generality, we set $\varphi(0)$=0. Otherwise, we substitute $u-\varphi(0)$ for $u$ throughout this paper. Notice that following the standard elliptic theory (see \cite{ADN1959,ADN1964}), we conclude that
\begin{align*}
\|\nabla u\|_{L^{\infty}(\Omega\setminus\Omega_{R})}\leq C\|\varphi\|_{C^{2}(\partial D)},
\end{align*}
from which it suffices to study the singular behavior of the gradient $\nabla u$ in the thin gap $\Omega_{R}$ in the following.

Our first main result is concerned with the optimal upper and lower bounds on the blow-up rate of the gradient in the shortest segment $\{x'=0'\}\cap\Omega$ between the inclusion and the matrix boundary.
\begin{theorem}\label{Lthm066}
Assume that $D_{1}\subset D\subseteq\mathbb{R}^{d}\,(d\geq2)$ are defined as above, conditions $\mathrm{(}${\bf{H1}}$\mathrm{)}$--$\mathrm{(}${\bf{H3}}$\mathrm{)}$ hold, and $\varphi\in C^{2}(\partial D;\mathbb{R}^{d})$. Let $u\in H^{1}(D;\mathbb{R}^{d})\cap C^{1}(\overline{\Omega};\mathbb{R}^{d})$ be the solution of (\ref{La.002}). Then for a sufficiently small $\varepsilon>0$, $x\in\{x'=0'\}\cap\Omega$,

$(\rm{i})$ if condition $\mathrm{(}${\bf{E1}}$\mathrm{)}$, $\mathrm{(}${\bf{E2}}$\mathrm{)}$ or $\mathrm{(}${\bf{E3}}$\mathrm{)}$ holds for $m\leq d$, then
\begin{align*}
|\nabla u|\lesssim
\begin{cases}
\frac{\max\limits_{1\leq\alpha\leq d}\kappa_{2}^{\frac{d-1}{m}}|\mathcal{L}_{d}^{\alpha}|^{-1}|\det\mathbb{F}_{1}^{\ast\alpha}[\varphi]|}{\det\mathbb{D}^{\ast}}\frac{1}{\varepsilon\rho_{0}(d,m;\varepsilon)},&d-1\leq m\leq d,\\
\frac{\max\limits_{1\leq\alpha\leq d}|\det\mathbb{F}_{3}^{\ast\alpha}[\varphi]|}{\det\mathbb{F}^{\ast}}\frac{1}{\varepsilon},&m<d-1,
\end{cases}
\end{align*}
and, there exists some integer $1\leq\alpha_{0}\leq d$ such that $\det\mathbb{F}_{1}^{\ast\alpha_{0}}[\varphi]\neq0$ and $\det\mathbb{F}_{3}^{\ast\alpha_{0}}[\varphi]\neq0$,
\begin{align*}
|\nabla u|\gtrsim&
\begin{cases}
\frac{\kappa_{1}^{\frac{d-1}{m}}|\det\mathbb{F}_{1}^{\ast\alpha_{0}}[\varphi]|}{|\mathcal{L}_{d}^{\alpha_{0}}|\det\mathbb{D}^{\ast}}\frac{1}{\varepsilon\rho_{0}(d,m;\varepsilon)},&d-1\leq m\leq d,\\
\frac{|\det\mathbb{F}_{3}^{\ast\alpha_{0}}[\varphi]|}{\det\mathbb{F}^{\ast}}\frac{1}{\varepsilon},&m<d-1;
\end{cases}
\end{align*}

$(\rm{ii})$ if condition $\mathrm{(}${\bf{E1}}$\mathrm{)}$ holds for $m\geq d+k$, then
\begin{align*}
\frac{\eta\kappa_{1}^{\frac{d-1}{m}}}{\kappa_{2}^{\frac{d+k-1}{m}}}\frac{\rho_{k}(d,m;\varepsilon)}{\varepsilon\rho_{0}(d,m;\varepsilon)}\lesssim|\nabla u|\lesssim\frac{\eta\kappa_{2}^{\frac{d-1}{m}}}{\kappa_{1}^{\frac{d+k-1}{m}}}\frac{\rho_{k}(d,m;\varepsilon)}{\varepsilon\rho_{0}(d,m;\varepsilon)},
\end{align*}
where $\kappa_{i}$, $i=1,2$ are defined in condition $\mathrm{(}${\bf{H1}}$\mathrm{)}$, $\mathcal{L}_{d}^{\alpha}$, $\alpha=1,2,...,d$ are defined by \eqref{AZ}--\eqref{AZ110}, $\rho_{i}(d,m;\varepsilon)$, $i=0,k$ are defined in \eqref{rate}, the blow-up factor matrices $\mathbb{D}^{\ast}$, $\mathbb{F}^{\ast}$, $\mathbb{A}_{1}^{\ast\alpha}$, $\alpha=1,2,...,d$ and $\mathbb{F}_{2}^{\ast\alpha}[\varphi]$, $\alpha=1,2,...,\frac{d(d+1)}{2}$ are defined by \eqref{MTABZR001}--\eqref{LRNMA001}, respectively.

\end{theorem}
\begin{remark}
Our results in Theorems \ref{Lthm066} and \ref{thma002} not only answer the optimality of the blow-up rate of the gradient in all dimensions, but also improve the results in Theorems 1.10 and 6.1 of \cite{LZ2020} by accurately capturing the blow-up factor matrices and revealing the explicit dependence on the Lam\'{e} constants $\mathcal{L}_{d}^{\alpha}$ and the curvature parameters $\kappa_{1}$ and $\kappa_{2}$.
\end{remark}
\begin{remark}
The assumed condition $\det\mathbb{F}_{1}^{\ast\alpha_{0}}\neq0$ or $\det\mathbb{F}_{3}^{\ast\alpha_{0}}[\varphi]\neq0$ implies that $\varphi\not\equiv0$ on $\partial D$. Otherwise, if $\varphi=0$ on $\partial D$, then it follows from integration by parts that
$Q_{\alpha_{0}}^{\ast}[\varphi]=\int_{\partial D}\frac{\partial u_{\alpha_{0}}^{\ast}}{\partial\nu_{0}}\big|_{+}\cdot\varphi=0.$ This yields that $\det\mathbb{F}_{1}^{\ast\alpha_{0}}[\varphi]=0$ and $\det\mathbb{F}_{3}^{\ast\alpha_{0}}[\varphi]=0$, which provides a contradiction. Although it is difficult to demonstrate the assumed condition $\det\mathbb{F}_{1}^{\ast\alpha_{0}}[\varphi]\neq0$ or $\det\mathbb{F}_{3}^{\ast\alpha_{0}}[\varphi]\neq0$ for any given boundary data $\varphi$, it is an interesting problem to analyze these blow-up factor matrices by numerical computations and simulations.
\end{remark}
\begin{remark}
In Theorem \ref{Lthm066}, for the purpose of constructing the lower bound on the gradient in the case of $m<d-1$, the blow-up factor matrix $\det\mathbb{F}_{3}^{\ast\alpha_{0}}[\varphi]$, as a whole, is required to be non-zero, which means that there exists at least one non-zero element in the column vector $\big(Q_{1}^{\ast}[\varphi],Q_{2}^{\ast}[\varphi],...,Q_{\frac{d(d+1)}{2}}^{\ast}[\varphi]\big)^{T}$. This weakens the assumed condition ($\boldsymbol{\Phi5}$) of Theorem 1.10 in \cite{LZ2020} that for $m<d-1$, there exists some integer $1\leq k_{0}\leq d$ such that $Q_{k_{0}}^{\ast}[\varphi]\neq0$ and $Q_{\beta}^{\ast}[\varphi]=0$ for all $\beta\neq k_{0}$.
\end{remark}

The next theorem aims to establish the optimal gradient estimates on the cylinder surface $\{|x'|=\sqrt[m]{\varepsilon}\}\cap\Omega$. Similarly as before, we first introduce some blow-up factor matrices. For $\alpha=d+1,...,\frac{d(d+1)}{2}$, after replacing the elements of $\alpha$-th column in the matrix $\mathbb{D}^{\ast}$ by column vector $\big(Q_{d+1}^{\ast}[\varphi],...,Q_{\frac{d(d+1)}{2}}^{\ast}[\varphi]\big)^{T}$, we obtain the new matrix $\mathbb{F}_{2}^{\ast\alpha}[\varphi]$ as follows:
\begin{gather}\label{HNMT001}
\mathbb{F}_{2}^{\ast\alpha}[\varphi]=:
\begin{pmatrix}
a_{d+1\,d+1}^{\ast}&\cdots&Q_{d+1}^{\ast}[\varphi]&\cdots&a_{d+1\,\frac{d(d+1)}{2}}^{\ast} \\\\ \vdots&\ddots&\vdots&\ddots&\vdots\\\\a_{\frac{d(d+1)}{2}\,d+1}^{\ast}&\cdots&Q_{\frac{d(d+1)}{2}}^{\ast}[\varphi]&\cdots&a_{\frac{d(d+1)}{2}\frac{d(d+1)}{2}}^{\ast}
\end{pmatrix}.
\end{gather}

Then our second main theorem is stated as follows:
\begin{theorem}\label{thma002}
Assume that $D_{1}\subset D\subseteq\mathbb{R}^{d}\,(d\geq2)$ are defined as above, conditions $\mathrm{(}${\bf{H1}}$\mathrm{)}$--$\mathrm{(}${\bf{H3}}$\mathrm{)}$ hold, and $\varphi\in C^{2}(\partial D;\mathbb{R}^{d})$. Let $u\in H^{1}(D;\mathbb{R}^{d})\cap C^{1}(\overline{\Omega};\mathbb{R}^{d})$ be the solution of (\ref{La.002}). Then for a sufficiently small $\varepsilon>0$, $x\in\{x'=(\sqrt[m]{\varepsilon},0,...,0)'\}\cap\Omega$,

$(\rm{i})$ if condition $\mathrm{(}${\bf{E1}}$\mathrm{)}$, $\mathrm{(}${\bf{E2}}$\mathrm{)}$ or $\mathrm{(}${\bf{E3}}$\mathrm{)}$ holds for $d<m<d+k,\,k>1$, then
\begin{align*}
|\nabla u|\lesssim&
\begin{cases}
\max\limits_{d+1\leq\alpha\leq\frac{d(d+1)}{2}}|\mathcal{L}_{d}^{\alpha}|^{-1}|Q_{\alpha}^{\ast}[\varphi]|\frac{\kappa_{2}^{\frac{d+1}{m}}}{1+\kappa_{1}}\frac{1}{\varepsilon^{1-1/m}\rho_{2}(d,m;\varepsilon)},&d+1\leq m<d+k,\\
\frac{\max\limits_{d+1\leq\alpha\leq\frac{d(d+1)}{2}}|\det\mathbb{F}_{2}^{\ast\alpha}[\varphi]|}{(1+\kappa_{1})\det\mathbb{D}^{\ast}}\frac{1}{\varepsilon^{1-1/m}},&d<m<d+1,
\end{cases}
\end{align*}
and, under the condition of $Q_{d+1}^{\ast}[\varphi]\neq0$ and $\mathbb{F}_{2}^{\ast d+1}[\varphi]\neq0$,
\begin{align*}
|\nabla u|\gtrsim&
\begin{cases}
\frac{\kappa_{2}^{\frac{d+1}{m}}|Q_{d+1}^{\ast}[\varphi]|}{(1+\kappa_{2})|\mathcal{L}_{d}^{d+1}|}\frac{1}{\varepsilon^{1-1/m}\rho_{2}(d,m;\varepsilon)},&d+1\leq m<d+k,\\
\frac{|\det\mathbb{F}_{2}^{\ast d+1}[\varphi]|}{(1+\kappa_{2})\det\mathbb{D}^{\ast}}\frac{1}{\varepsilon^{1-1/m}},&d<m<d+1;
\end{cases}
\end{align*}

$(\rm{ii})$ if condition $\mathrm{(}${\bf{E2}}$\mathrm{)}$ holds, then for $m>d+k,\,k\geq1,\,k\neq2$ or $m=d+k,\,k=1$,
\begin{align*}
\frac{\eta(\kappa_{1}^{\frac{d+1}{m}}\kappa_{2}^{-\frac{d+k}{m}}+1)}{1+\kappa_{2}}\frac{1}{\varepsilon^{1-k/m}}\lesssim|\nabla u|\lesssim&
\frac{\eta(\kappa_{2}^{\frac{d+1}{m}}\kappa_{1}^{-\frac{d+k}{m}}+1)}{1+\kappa_{1}}\frac{1}{\varepsilon^{1-k/m}},
\end{align*}
and, for $m=d+k,\,k>2$,
\begin{align*}
\frac{\eta\kappa_{1}^{\frac{d+1}{m}}}{(1+\kappa_{2})\kappa_{2}^{\frac{d+k}{m}}}\frac{\rho_{k+1}(d,m;\varepsilon)}{\varepsilon^{1-1/m}\rho_{2}(d,m;\varepsilon)}\lesssim|\nabla u|\lesssim\frac{\eta\kappa_{2}^{\frac{d+1}{m}}}{(1+\kappa_{1})\kappa_{1}^{\frac{d+k}{m}}}\frac{\rho_{k+1}(d,m;\varepsilon)}{\varepsilon^{1-1/m}\rho_{2}(d,m;\varepsilon)};
\end{align*}

$(\rm{iii})$ if condition $\mathrm{(}${\bf{E3}}$\mathrm{)}$ holds, then for $m>d+k,\,k\geq1,\,k\neq2$ or $m=d+k$, $k=1$,
\begin{align}\label{ANBM}
\frac{\eta}{1+\kappa_{2}}\frac{1}{\varepsilon^{1-k/m}}\lesssim|\nabla u|\lesssim\frac{\eta}{1+\kappa_{1}}\frac{1}{\varepsilon^{1-k/m}},
\end{align}
and, for $m=d+k,\,k>2$,
\begin{align*}
|\nabla u|\lesssim \frac{\max\limits_{d+1\leq\alpha\leq\frac{d(d+1)}{2}}\kappa_{2}^{\frac{d+1}{m}}|\mathcal{L}_{d}^{\alpha}|^{-1}|Q^{\ast}_{\alpha}[\varphi]|+\eta}{1+\kappa_{1}}\frac{1}{\varepsilon^{1-k/m}},
\end{align*}
and, under the condition of $Q^{\ast}_{d+1}[\varphi]\neq0$,
\begin{align*}
|\nabla u|\gtrsim&\frac{1}{\varepsilon^{1-k/m}}
\frac{\kappa_{1}^{\frac{d+1}{m}}|Q^{\ast}_{d+1}[\varphi]|}{(1+\kappa_{2})\mathcal{L}_{d}^{d+1}},
\end{align*}
where $\kappa_{i}$, $i=1,2$ are defined in condition $\mathrm{(}${\bf{H1}}$\mathrm{)}$, the blow-up factors $Q_{\alpha}^{\ast}[\varphi]$ and the Lam\'{e} constants $\mathcal{L}_{d}^{\alpha}$, $\alpha=d+1,...,\frac{d(d+1)}{2}$ are, respectively, defined by \eqref{FNCL001} and \eqref{AZ}--\eqref{AZ110}, $\rho_{i}(d,m;\varepsilon)$, $i=2,k+1$ are defined in \eqref{rate}, the blow-up factor matrices $\mathbb{D}^{\ast}$ and $\mathbb{F}_{1}^{\ast\alpha}[\varphi]$, $\alpha=d+1,...,\frac{d(d+1)}{2}$ are defined by \eqref{MTABZR001} and \eqref{HNMT001}, respectively.
\end{theorem}
\begin{remark}
The optimal gradient estimates in Theorem \ref{thma002} address the remaining optimality of the blow-up rate in Remark 6.5 of \cite{LZ2020}, which makes complete the optimal lower bounds of the gradient on the cylinder surface $\{|x'|=\sqrt[m]{\varepsilon}\}\cap\Omega$. In addition, it is worth mentioning that as shown in \eqref{ANBM}, if condition ({\bf{E3}}) holds, the leading singularity of $|\nabla u|$ only arises from $|\nabla u_{0}|$ in the case of $m>d+k,\,k\geq1,\,k\neq2$ or $m=d+k$, $k=1$. This is different from the blow-up phenomenon under condition ({\bf{E1}}) or ({\bf{E2}}).
\end{remark}

By applying the proofs of Theorems \ref{Lthm066} and \ref{thma002} with a slight modification, we obtain the following two corollaries. Before stating the first corollary, we first introduce two notations $\mathcal{H}_{A}^{\ast}(m,d,k;\varphi)$ and $\mathcal{H}_{B}^{\ast}(m,d,k;\varphi)$. To be specific, define
\begin{align}\label{MRA001}
&\mathcal{H}_{A}^{\ast}(m,d,k;\varphi):=\notag\\
&\begin{cases}
\eta\kappa_{2}^{\frac{d-1}{m}}\kappa_{1}^{-\frac{d+k-1}{m}},&m\geq d+k-1,\;\text{for\;case\;({\bf{A1}})},\\
\max\limits_{1\leq\alpha\leq d}\kappa_{2}^{\frac{d-1}{m}}(\mathcal{L}_{d}^{\alpha})^{-1}|Q_{\alpha}^{\ast}[\varphi]|,&m\geq d+k-1,\;\text{for\;case\;({\bf{A2}})\;or\;({\bf{A3}})},\\
\max\limits_{1\leq\alpha\leq d}\kappa_{2}^{\frac{d-1}{m}}(\mathcal{L}_{d}^{\alpha})^{-1}|Q_{\alpha}^{\ast}[\varphi]|,& d+1\leq m<d+k-1,\\
\frac{\max\limits_{1\leq\alpha\leq d}\kappa_{2}^{\frac{d-1}{m}}(\mathcal{L}_{d}^{\alpha})^{-1}|\det\mathbb{F}_{1}^{\ast\alpha}[\varphi]|}{\det\mathbb{D}^{\ast}},&d-1\leq m<d+1,\\
\frac{\max\limits_{1\leq\alpha\leq d}|\det\mathbb{F}_{3}^{\ast\alpha}[\varphi]|}{\det\mathbb{F}^{\ast}},& m<d-1,
\end{cases}
\end{align}
and
\begin{align}\label{MRA002}
&\mathcal{H}_{B}^{\ast}(m,d,k;\varphi):=\notag\\
&\begin{cases}
\eta\kappa_{2}^{\frac{d+1}{m}}\kappa_{1}^{-\frac{d+k}{m}},&m\geq d+k,\;\text{for\;case\;({\bf{A2}})},\\
\max\limits_{d+1\leq\alpha\leq\frac{d(d+1)}{2}}\kappa_{2}^{\frac{d+1}{m}}(\mathcal{L}_{d}^{\alpha})^{-1}|Q_{\alpha}^{\ast}[\varphi]|,&m\geq d+k,\;\text{for\;case\;({\bf{A1}})\;or\;({\bf{A3}})},\\
\max\limits_{d+1\leq\alpha\leq\frac{d(d+1)}{2}}\kappa_{2}^{\frac{d+1}{m}}(\mathcal{L}_{d}^{\alpha})^{-1}|Q_{\alpha}^{\ast}[\varphi]|,& d+1\leq m<d+k,\\
\frac{\max\limits_{d+1\leq\alpha\leq\frac{d(d+1)}{2}}|\det\mathbb{F}_{2}^{\ast\alpha}[\varphi]|}{\det\mathbb{D}^{\ast}},&d-1\leq m<d+1,\\
\frac{\max\limits_{d+1\leq\alpha\leq\frac{d(d+1)}{2}}|\det\mathbb{F}_{3}^{\ast\alpha}[\varphi]|}{\det\mathbb{F}^{\ast}},& m<d-1.
\end{cases}
\end{align}
Then the first corollary is listed as follows:
\begin{corollary}
Assume that $D_{1}\subset D\subseteq\mathbb{R}^{d}\,(d\geq2)$ are defined as above, conditions $\mathrm{(}${\bf{H1}}$\mathrm{)}$--$\mathrm{(}${\bf{H3}}$\mathrm{)}$ hold. Let $u\in H^{1}(D;\mathbb{R}^{d})\cap C^{1}(\overline{\Omega};\mathbb{R}^{d})$ be the solution of \eqref{La.002}. Assume that $\mathrm{(}${\bf{A1}}$\mathrm{)}$, $\mathrm{(}${\bf{A2}}$\mathrm{)}$ or $\mathrm{(}${\bf{A3}}$\mathrm{)}$ holds. If $\varphi\in C^{2}(\partial D;\mathbb{R}^{d})$ satisfies the $k$-order growth condition \eqref{growth}, then for a sufficiently small $\varepsilon>0$, $x\in\Omega_{R}$,
\begin{align*}
|\nabla u|\lesssim\frac{1}{\varepsilon+\kappa_{1}|x'|^{m}}\left(\mathcal{H}_{A}^{\ast}(m,d,k;\varphi)\rho_{A}(\varepsilon)+\mathcal{H}_{B}^{\ast}(m,d,k;\varphi)\rho_{B}(\varepsilon)|x'|+\eta|x'|^{k}\right),
\end{align*}
where $\rho_{A}(\varepsilon)$ and $\rho_{B}(\varepsilon)$ are, respectively, defined by \eqref{LT001}--\eqref{LT002}, $\mathcal{H}_{A}^{\ast}(m,d,k;\varphi)$ and $\mathcal{H}_{B}^{\ast}(m,d,k;\varphi)$ are defined by \eqref{MRA001}--\eqref{MRA002}, respectively.
\end{corollary}
\begin{remark}
In contrast to \eqref{DALN666}, we improve its result by making clear the dependence on the blow-up factor matrices, the Lam\'{e} constants $\mathcal{L}_{d}^{\alpha}$ and the curvature parameters $\kappa_{1}$ and $\kappa_{2}$ as shown in $\mathcal{H}_{A}^{\ast}(m,d,k;\varphi)$ and $\mathcal{H}_{B}^{\ast}(m,d,k;\varphi)$.
\end{remark}

Our results can also be extended to the inclusions with partially flat boundaries as follows:
\begin{enumerate}
{\it\item[(\bf{S1})]
$\kappa_{1}\mathrm{dist}^{m}(x',\Sigma')\leq h_{1}(x')-h(x')\leq\kappa_{2}\mathrm{dist}^{m}(x',\Sigma'),\;\mbox{if}\;\,x'\in B_{2R}'\setminus\Sigma',$
\item[(\bf{S2})]
$|\nabla_{x'}^{j}h_{1}(x')|,\,|\nabla_{x'}^{j}h(x')|\leq \kappa_{3}\mathrm{dist}^{m-j}(x',\Sigma'),\;\mbox{if}\;\,x'\in B_{2R}',\;j=1,2.$}
\end{enumerate}
For the convenience of notation, let $a\approx b$ represent that both $a\lesssim b$ and $a\gtrsim b$ hold. For $i=0,2,k,k+1$ and $j=1,2$, denote
\begin{align}\label{FNT001}
\mathcal{G}_{ij}^{\varepsilon}(|\Sigma'|):=|\Sigma'|^{\frac{d+i-1}{d-1}}+|\Sigma'|^{\frac{d-2+i}{d-1}}\kappa_{j}^{-\frac{1}{m}}\varepsilon^{\frac{1}{m}}+\kappa_{j}^{-\frac{d-1+i}{m}}\varepsilon\rho_{i}(d,m;\varepsilon).
\end{align}
With this notation, we state the second corollary as follows:
\begin{corollary}\label{coroLL02}
Assume that $D_{1}\subset D\subseteq\mathbb{R}^{d}\,(d\geq2)$ are defined as above, conditions $\mathrm{(}${\bf{S1}}$\mathrm{)}$--$\mathrm{(}${\bf{S2}}$\mathrm{)}$ and $\mathrm{(}${\bf{H3}}$\mathrm{)}$ hold with $\Sigma'=B_{r}'(0')$ for some $0<r<R$. Let $u\in H^{1}(D;\mathbb{R}^{d})\cap C^{1}(\overline{\Omega};\mathbb{R}^{d})$ be the solution of \eqref{La.002}. Then for a sufficiently small $\varepsilon>0$, $x\in\{x'=(r,0,...,0)'\}\cap\Omega$,

$(\rm{i})$ if condition $\mathrm{(}${\bf{E1}}$\mathrm{)}$ holds,
\begin{align*}
\left(\frac{\mathcal{G}_{k2}^{\varepsilon}(|\Sigma'|)}{\mathcal{G}_{01}^{\varepsilon}(|\Sigma'|)}+|\Sigma'|^{\frac{k}{d-1}}\right)\frac{\eta}{\varepsilon}\lesssim|\nabla u|\lesssim\left(\frac{\mathcal{G}_{k1}^{\varepsilon}(|\Sigma'|)}{\mathcal{G}_{02}^{\varepsilon}(|\Sigma'|)}+|\Sigma'|^{\frac{k}{d-1}}\right)\frac{\eta}{\varepsilon};
\end{align*}

$(\rm{ii})$ if condition $\mathrm{(}${\bf{E2}}$\mathrm{)}$ holds,
\begin{align*}
\left(\frac{|\Sigma'|^{\frac{1}{d-1}}\mathcal{G}_{k+1\,2}^{\varepsilon}(|\Sigma'|)}{\mathcal{G}_{21}^{\varepsilon}(|\Sigma'|)}+|\Sigma'|^{\frac{k}{d-1}}\right)\frac{\eta}{\varepsilon}\lesssim|\nabla u|\lesssim\left(\frac{|\Sigma'|^{\frac{1}{d-1}}\mathcal{G}_{k+1\,1}^{\varepsilon}(|\Sigma'|)}{\mathcal{G}_{22}^{\varepsilon}(|\Sigma'|)}+|\Sigma'|^{\frac{k}{d-1}}\right)\frac{\eta}{\varepsilon};
\end{align*}

$(\rm{iii})$ if condition $\mathrm{(}${\bf{E3}}$\mathrm{)}$ holds,
\begin{align*}
|\nabla u|\approx\frac{\eta|\Sigma'|^{\frac{k}{d-1}}}{\varepsilon},
\end{align*}
where $\mathcal{G}_{ij}^{\varepsilon}(|\Sigma'|)$, $i=0,2,k,k+1$, $j=1,2$ are defined by \eqref{FNT001}.

\end{corollary}
\begin{remark}
In fact, \eqref{FNT001} can be rewritten as $\mathcal{G}_{ij}^{\varepsilon}(|\Sigma'|)=|\Sigma'|^{\frac{d+i-1}{d-1}}+O(\varepsilon^{\frac{1}{m}})$, which implies that if condition ({\bf{E1}}), ({\bf{E2}}) or ({\bf{E3}}) holds, then for $x\in\{x'=(r,0,...,0)'\}\cap\Omega$, the results in Corollary \ref{coroLL02} have an unified expression as follows:
\begin{align*}
|\nabla u|\approx\frac{\eta|\Sigma'|^{\frac{k}{d-1}}}{\varepsilon}.
\end{align*}
\end{remark}
\begin{remark}
It is worth mentioning that in view of decomposition \eqref{Le2.015} below, we conclude from Corollary \ref{coroLL02} that if condition ({\bf{E1}}) holds, the singular behavior of the gradient is determined by the first part $\sum^{d}_{\alpha=1}C^{\alpha}\nabla u_{\alpha}$ and the third part $\nabla u_{0}$ together; if condition ({\bf{E2}}) holds, the major singularity of the gradient lies in the second part $\sum^{\frac{d(d+1)}{2}}_{\alpha=d+1}C^{\alpha}\nabla u_{\alpha}$ and the third part $\nabla u_{0}$; if condition ({\bf{E3}}) holds, the singular behavior of the gradient is determined only by the third part $\nabla u_{0}$. These fact indicates that when the inclusions with partially flat boundaries are closely located at the matrix boundary, $\nabla u_{0}$ possesses the largest blow-up rate $\varepsilon^{-1}$ under these three cases all the time, which is different from the blow-up phenomenon with $m$-convex inclusions as seen in Theorems \ref{Lthm066} and \ref{thma002}. Moreover, this blow-up phenomena also differs from that in the interior estimates of \cite{HJL2018}, where Hou, Ju and Li proved that there appears no blow-up of the gradient in the presence of two nearby inclusions with partially flat boundaries for any given boundary data.

\end{remark}

\section{Preliminary}\label{SL003}

As shown in \cite{BJL2017,LZ2020}, the solution of \eqref{La.002} can be decomposed as follows:
\begin{align*}
u=\sum^{d}_{\alpha=1}C^{\alpha}u_{\alpha}+\sum^{\frac{d(d+1)}{2}}_{\alpha=d+1}C^{\alpha}u_{\alpha}+u_{0},\quad\;\,\mathrm{in}\;\Omega,
\end{align*}
where the free constants $C^{\alpha},\,\alpha=1,2,...,\frac{d(d+1)}{2},$ will be determined later by utilizing the third line of \eqref{La.002}, and $u_{\alpha}\in C^{1}(\overline{\Omega};\mathbb{R}^{d})\cap C^{2}(\Omega;\mathbb{R}^{d}),\,\alpha=0,1,2,...,\frac{d(d+1)}{2}$, verify
\begin{equation}\label{P2.005}
\begin{cases}
\mathcal{L}_{\lambda,\mu}u_{0}=0,&\mathrm{in}\;\Omega,\\
u_{0}=0,&\mathrm{on}\;\partial D_{1},\\
u_{0}=\varphi,&\mathrm{on}\;\partial D,
\end{cases}\quad
\begin{cases}
\mathcal{L}_{\lambda,\mu}u_{\alpha}=0,\quad\;\,&\mathrm{in}\;\,\Omega,\\
u_{\alpha}=\psi_{\alpha},\quad\;\,&\mathrm{on}\;\,\partial D_{1},\\
u_{\alpha}=0,\quad\;\,&\mathrm{on}\;\,\partial D,
\end{cases}
\end{equation}
respectively. Therefore,
\begin{align}\label{Le2.015}
\nabla u=\sum^{d}_{\alpha=1}C^{\alpha}\nabla u_{\alpha}+\sum^{\frac{d(d+1)}{2}}_{\alpha=d+1}C^{\alpha}\nabla u_{\alpha}+\nabla u_{0},\quad\;\,\mathrm{in}\;\,\Omega.
\end{align}
With this, we transfer the original problem to the estimates of the free constants $C^{\alpha}$, $\alpha=1,2,...,\frac{d(d+1)}{2}$ and the gradients $\nabla u_{\alpha}$, $\alpha=0,1,2,...,\frac{d(d+1)}{2}$. Moreover, decomposition \eqref{Le2.015} splits $\nabla u$ into three parts, that is, $\sum^{d}_{\alpha=1}C^{\alpha}\nabla u_{\alpha}$, $\sum^{\frac{d(d+1)}{2}}_{\alpha=d+1}C^{\alpha}\nabla u_{\alpha}$ and $\nabla u_{0}$. Then it suffices to compare the singularities of these three parts and then identify the largest blow-up rate of them.

For the purpose of studying the singular behavior of $\nabla u_{\alpha}$, $\alpha=0,1,2,...,\frac{d(d+1)}{2}$, we first consider a general boundary value problem as follows:
\begin{equation}\label{P2.008}
\begin{cases}
\mathcal{L}_{\lambda,\mu}v:=\nabla\cdot(\mathbb{C}^{0}e(v))=0,\quad\;\,&\mathrm{in}\;\,\Omega,\\
v=\psi(x),&\mathrm{on}\;\,\partial D_{1},\\
v=\phi(x),&\mathrm{on}\;\,\partial D,
\end{cases}
\end{equation}
where $\psi\in C^{2}(\partial D_{1};\mathbb{R}^{d})$ and $\phi\in C^{2}(\partial D;\mathbb{R}^{d})$ are two given vector-valued functions.

Introduce a vector-valued auxiliary function as follows:
\begin{align}\label{AHM001}
\tilde{v}=&\psi(x',\varepsilon+h_{1}(x'))\bar{v}+\phi(x',h(x'))(1-\bar{v})\notag\\
&+\frac{\lambda+\mu}{\mu}f(\bar{v})(\psi^{d}(x',\varepsilon+h_{1}(x'))-\phi^{d}(x',h(x')))\sum^{d-1}_{i=1}\partial_{x_{i}}\delta\,e_{i}\notag\\
&+\frac{\lambda+\mu}{\lambda+2\mu}f(\bar{v})\sum^{d-1}_{i=1}\partial_{x_{i}}\delta(\psi^{i}(x',\varepsilon+h_{1}(x'))-\phi^{i}(x',h(x')))\,e_{d},
\end{align}
where $\bar{v}$ is defined by \eqref{BATL001}, $\delta$ and $f(\bar{v})$ are defined in \eqref{deta}. For the remaining term, write
\begin{align}\label{remin001}
\mathcal{R}_{\delta}(\psi,\phi)=&|\psi(x',\varepsilon+h_{1}(x'))-\phi(x',h(x'))|\delta^{\frac{m-2}{m}}+\delta\big(\|\psi\|_{C^{2}(\partial D_{1})}+\|\phi\|_{C^{2}(\partial D_{2})})\notag\\
&+|\nabla_{x'}(\psi(x',\varepsilon+h_{1}(x'))-\phi(x',h(x')))|.
\end{align}
\begin{theorem}\label{thm8698}
Assume as above. Let $v\in H^{1}(\Omega;\mathbb{R}^{d})$ be a weak solution of (\ref{P2.008}). Then for a sufficiently small $\varepsilon>0$,
\begin{align*}
\nabla v=\nabla\tilde{v}+O(1)\mathcal{R}_{\delta}(\psi,\phi),\quad\mathrm{in}\;\Omega_{R},
\end{align*}
where $\delta$ is defined in \eqref{deta}, the leading term $\tilde{v}$ and the remaining term $\mathcal{R}_{\delta}(\psi,\phi)$ are defined by \eqref{AHM001}--\eqref{remin001}, respectively.
\end{theorem}
For readers' convenience, the proof of Theorem \ref{thm8698} is left in the Appendix.

To end this section, we recall some properties of the tensor $\mathbb{C}^{0}$. For the isotropic elastic material, set
\begin{align*}
\mathbb{C}^{0}:=(C_{ijkl}^{0})=(\lambda\delta_{ij}\delta_{kl}+\mu(\delta_{ik}\delta_{jl}+\delta_{il}\delta_{jk})),\quad \mu>0,\quad d\lambda+2\mu>0,
\end{align*}
whose components $C_{ijkl}^{0}$ possess the symmetry property as follows:
\begin{align}\label{symm}
C_{ijkl}^{0}=C_{klij}^{0}=C_{klji}^{0},\quad i,j,k,l=1,2,...,d.
\end{align}
For every pair of $d\times d$ matrices $A=(A_{ij})$ and $B=(B_{ij})$, we introduce the following notations:
\begin{align*}
(\mathbb{C}^{0}A)_{ij}=\sum_{k,l=1}^{n}C_{ijkl}^{0}A_{kl},\quad\hbox{and}\quad(A,B)\equiv A:B=\sum_{i,j=1}^{n}A_{ij}B_{ij}.
\end{align*}
Obviously,
\begin{align*}
(\mathbb{C}^{0}A,B)=(A, \mathbb{C}^{0}B).
\end{align*}
Using \eqref{symm}, we get that $\mathbb{C}^{0}$ verifies the ellipticity condition: for every $d\times d$ real symmetric matrix $\xi=(\xi_{ij})$,
\begin{align*}
\min\{2\mu, d\lambda+2\mu\}|\xi|^2\leq(\mathbb{C}^{0}\xi, \xi)\leq\max\{2\mu, d\lambda+2\mu\}|\xi|^2,
\end{align*}
where $|\xi|^2=\sum\limits_{ij}\xi_{ij}^2.$ Especially,
\begin{align}\label{Le2.012}
\min\{2\mu, d\lambda+2\mu\}|A+A^T|^2\leq(\mathbb{C}^{0}(A+A^T), (A+A^T)).
\end{align}
Additionally, it is well known that for any open set $O$ and $u, v\in C^2(O;\mathbb{R}^{d})$,
\begin{align}\label{Le2.01222}
\int_O(\mathbb{C}^0e(u), e(v))\,dx=-\int_O\left(\mathcal{L}_{\lambda, \mu}u\right)\cdot v+\int_{\partial O}\frac{\partial u}{\partial \nu_0}\Big|_{+}\cdot v.
\end{align}

\section{Proofs of Theorems \ref{Lthm066} and \ref{thma002}}\label{SL004}

Taking $\psi=\psi_{\alpha}$, $\phi=0$, $\alpha=1,2,...,\frac{d(d+1)}{2}$, or $\psi=0$, $\phi=\varphi$ in Theorem \ref{thm8698}, we have
\begin{corollary}\label{thm86}
Assume as above. Let $u_{\alpha}\in H^{1}(\Omega;\mathbb{R}^{d})$, $\alpha=1,2,...,\frac{d(d+1)}{2}$ be a weak solution of (\ref{P2.005}). Then, for a sufficiently small $\varepsilon>0$, $x\in\Omega_{R}$,
\begin{align}\label{ATCG001}
\nabla u_{\alpha}=&\nabla\bar{u}_{\alpha}+O(1)
\begin{cases}
\|\varphi\|_{C^{2}(\partial D)},&\alpha=0,\\
\delta^{\frac{m-2}{m}},&\alpha=1,2,...,d,\\
1,&\alpha=d+1,...,\frac{d(d+1)}{2},
\end{cases}
\end{align}
and
\begin{align*}
\|\nabla u_{\alpha}\|_{L^{\infty}(\Omega\setminus\Omega_{R})}=O(1)
\begin{cases}
\|\varphi\|_{C^{2}(\partial D)},&\alpha=0,\\
1,&\alpha=1,2,...,\frac{d(d+1)}{2},
\end{cases}
\end{align*}
where $\bar{u}_{\alpha}$, $\alpha=0,1,2,...,\frac{d(d+1)}{2}$, are defined in \eqref{CAN01}--\eqref{OPQ}, $O(1)$ denotes some quantity satisfying $|O(1)|\leq C$ for some $\varepsilon$-independent $C$ .
\end{corollary}
Recalling decomposition \eqref{Le2.015}, it remains to provide a precise calculation for the free constants $C^{\alpha}$, $\alpha=1,2,...,\frac{d(d+1)}{2}$ in the following. Observe that by utilizing the third line of \eqref{La.002}, it follows from \eqref{Le2.015} that
\begin{align}\label{Le3.078}
\sum^{\frac{d(d+1)}{2}}_{\alpha=1}C^{\alpha}a_{\alpha\beta}=Q_{\beta}[\varphi],\quad\beta=1,2,...,\frac{d(d+1)}{2},
\end{align}
where, for $\alpha,\beta=1,2,...,\frac{d(d+1)}{2}$,
\begin{align*}
a_{\alpha\beta}:=-\int_{\partial D_{1}}\frac{\partial u_{\alpha}}{\partial\nu_{0}}\Big|_{+}\cdot\psi_{\beta},\quad Q_{\alpha}[\varphi]:=\int_{\partial D_{1}}\frac{\partial u_{0}}{\partial\nu_{0}}\Big|_{+}\cdot\psi_{\alpha}.
\end{align*}
In light of \eqref{Le3.078}, we need to calculate each element $a_{\alpha\beta}$ and every blow-up factor $Q_{\beta}[\varphi]$ for the purpose of estimating the free constants $C^{\alpha}$, $\alpha=1,2,...,\frac{d(d+1)}{2}$.

\subsection{Estimates and asymptotics of $Q_{\alpha}[\varphi]$, $\alpha=1,2,...,\frac{d(d+1)}{2}$.}\label{subsec31}

To begin with, the unit outer normals of $\partial D_{1}$ and $\partial D$ near the origin are, respectively, written by
\begin{align}\label{KHA001}
\nu:=(\nu_{1},\nu_{2},...,\nu_{d})=\bigg(-\frac{\nabla_{x'}h_{1}}{\sqrt{1+|\nabla_{x'}h_{1}|^{2}}},\frac{1}{\sqrt{1+|\nabla_{x'}h_{1}|^{2}}}\bigg),
\end{align}
and
\begin{align}\label{KHA002}
\nu:=(\nu_{1},\nu_{2},...,\nu_{d})=\bigg(\frac{\nabla_{x'}h}{\sqrt{1+|\nabla_{x'}h|^{2}}},\frac{-1}{\sqrt{1+|\nabla_{x'}h|^{2}}}\bigg).
\end{align}

\begin{lemma}\label{KM323}
Assume as above. Then for a sufficiently small $\varepsilon>0$,

$(\rm{i})$ for $\alpha=1,2,...,d$, if condition $\mathrm{(}${\bf{E1}}$\mathrm{)}$ holds,
\begin{align}\label{KAWQ001}
\begin{cases}
\kappa_{2}^{-\frac{d+k-1}{m}}\lesssim\frac{Q_{\alpha}[\varphi]}{\eta\mathcal{L}_{d}^{\alpha}\rho_{k}(d,m;\varepsilon)}\lesssim\kappa_{1}^{-\frac{d+k-1}{m}},&m\geq d+k-1,\\
Q_{\alpha}[\varphi]=Q_{\alpha}^{\ast}[\varphi]+O(1)\varepsilon^{\frac{d+k-1-m}{2(d+k-1)}},&m<d+k-1,
\end{cases}
\end{align}
and, if condition $\mathrm{(}${\bf{E2}}$\mathrm{)}$ or $\mathrm{(}${\bf{E3}}$\mathrm{)}$ holds,
\begin{align}\label{KAWQ002}
Q_{\alpha}[\varphi]=&Q_{\alpha}^{\ast}[\varphi]+O(1)\varepsilon^{\frac{d+k-2}{2(m+d+k-1)}};
\end{align}

$(\rm{ii})$ for $\alpha=d+1$, if condition $\mathrm{(}${\bf{E2}}$\mathrm{)}$ holds,
\begin{align}\label{KAT90}
\begin{cases}
\kappa_{2}^{-\frac{d+k}{m}}\lesssim\frac{Q_{d+1}[\varphi]}{\mathcal{L}_{d}^{d+1}\eta\rho_{k+1}(d,m;\varepsilon)}\lesssim\kappa_{1}^{-\frac{d+k}{m}},&m\geq d+k,\\
Q_{d+1}[\varphi]=Q_{d+1}^{\ast}[\varphi]+O(1)\varepsilon^{\frac{d+k-m}{m(d+k)}},&m<d+k,
\end{cases}
\end{align}
and, if condition $\mathrm{(}${\bf{E1}}$\mathrm{)}$ or $\mathrm{(}${\bf{E3}}$\mathrm{)}$ holds,
\begin{align}\label{AUOP001}
Q_{d+1}[\varphi]=&Q_{d+1}^{\ast}[\varphi]+O(1)\varepsilon^{\frac{d+k-1}{m(m+d+k-1)}};
\end{align}

$(\rm{iii})$ for $\alpha=d+2,...,\frac{d(d+1)}{2},\; d\geq3$, if condition $\mathrm{(}${\bf{E1}}$\mathrm{)}$, $\mathrm{(}${\bf{E2}}$\mathrm{)}$ or $\mathrm{(}${\bf{E3}}$\mathrm{)}$ holds,
\begin{align}\label{TAKL001}
Q_{\alpha}[\varphi]&=Q_{\alpha}^{\ast}[\varphi]+O(1)\varepsilon^{\frac{d+k-1}{m(m+d+k-1)}},
\end{align}
where $Q_{\alpha}^{\ast}[\varphi]$, $\alpha=1,2,...,\frac{d(d+1)}{2}$ are defined in \eqref{FNCL001}.
\end{lemma}

\begin{proof}[Proof of Lemma \ref{KM323}]
{\bf Step 1.} Proofs of \eqref{KAWQ001}--\eqref{KAWQ002}. We divide into two subparts to prove \eqref{KAWQ001}--\eqref{KAWQ002} in the following.

{\bf Step 1.1.} By definition, we have
\begin{align*}
Q_{\alpha}[\varphi]=&\sum^{d}_{i=1}\int_{\partial D_{1}}\frac{\partial u_{0i}}{\partial\nu_{0}}\Big|_{+}\cdot\psi_{\alpha}\notag\\
=&\int_{\partial D_{1}}\bigg[\lambda\sum^{d}_{i,j=1}\partial_{x_{j}}u_{0i}^{j}\nu_{\alpha}+\mu\sum^{d}_{i,j=1}(\partial_{x_{j}}u_{0i}^{\alpha}+\partial_{x_{\alpha}}u_{0i}^{j})\nu_{j}\bigg].
\end{align*}

Then in view of \eqref{KHA001}, it follows from Corollary \ref{thm86} that

$(a)$ for $\alpha=1,...,d-1$,
\begin{align*}
Q_{\alpha}[\varphi]=&\int_{\partial D_{1}}\bigg[\lambda\sum^{d}_{i,j=1}\partial_{x_{j}}u_{0i}^{j}\nu_{\alpha}+\mu\sum^{d}_{(i,j)\neq(\alpha,d)}\partial_{x_{j}}u_{0i}^{\alpha}\nu_{j}+\mu\sum^{d}_{i,j=1}\partial_{x_{\alpha}}u_{0i}^{j}\nu_{j}\notag\\
&\quad\quad\quad+\mu\partial_{x_{d}}(u_{0\alpha}^{\alpha}-\bar{u}_{0\alpha}^{\alpha})\nu_{d}\bigg]+\int_{\partial D_{1}}\mu\partial_{x_{d}}\bar{u}_{0\alpha}^{\alpha}\nu_{d}\notag\\
=&-\mu\int_{\partial D_{1}\cap\Gamma^{+}_{R}}\frac{\varphi^{\alpha}(x',h(x'))}{\varepsilon+h_{1}(x')-h(x')}\nu_{d}+O(1)\|\varphi\|_{C^{2}(\partial D)},
\end{align*}
which implies that
\begin{align*}
Q_{\alpha}[\varphi]\lesssim&\mu\eta\int_{|x'|<R}\frac{|x'|^{k}}{\varepsilon+\kappa_{1}|x'|^{m}}\lesssim\mu\eta\int^{R}_{0}\frac{s^{d+k-2}}{\varepsilon+\kappa_{1}s^{m}}\lesssim\frac{\mu\eta}{\kappa_{1}^{\frac{d+k-1}{m}}}\rho_{k}(d,m;\varepsilon),
\end{align*}
and
\begin{align*}
Q_{\alpha}[\varphi]\gtrsim&\mu\eta\int_{|x'|<R}\frac{|x'|^{k}}{\varepsilon+\kappa_{2}|x'|^{m}}\gtrsim\mu\eta\int^{R}_{0}\frac{ s^{d+k-2}}{\varepsilon+\kappa_{2}s^{m}}\gtrsim\frac{\mu\eta}{\kappa_{2}^{\frac{d+k-1}{m}}}\rho_{k}(d,m;\varepsilon);
\end{align*}

$(b)$ for $\alpha=d$,
\begin{align*}
Q_{d}[\varphi]=&\int_{\partial D_{1}}\bigg[\sum^{d}_{(i,j)\neq(d,d)}\big(\lambda\partial_{x_{j}}u_{0i}^{j}\nu_{d}+\mu(\partial_{x_{j}}u_{0i}^{d}+\partial_{x_{d}}u_{0i}^{j})\nu_{j}\big)\notag\\
&\quad\quad\quad+(\lambda+2\mu)\partial_{x_{d}}(u_{0d}^{d}-\bar{u}_{0d}^{d})\nu_{d}\bigg]+\int_{\partial D_{1}}(\lambda+2\mu)\partial_{x_{d}}\bar{u}_{0d}^{d}\nu_{d}\notag\\
=&-(\lambda+2\mu)\int_{\partial D_{1}\cap\Gamma^{+}_{R}}\frac{\varphi^{d}(x',h(x'))}{\varepsilon+h_{1}(x')-h(x')}\nu_{d}+O(1)\|\varphi\|_{C^{2}(\partial D)},
\end{align*}
which yields that
\begin{align*}
Q_{d}[\varphi]\lesssim&(\lambda+2\mu)\eta\int_{|x'|<R}\frac{|x'|^{k}}{\varepsilon+\kappa_{1}|x'|^{m}}\lesssim\frac{(\lambda+2\mu)\eta}{\kappa_{1}^{\frac{d+k-1}{m}}}\rho_{k}(d,m;\varepsilon),
\end{align*}
and
\begin{align*}
Q_{d}[\varphi]\gtrsim&(\lambda+2\mu)\eta\int_{|x'|<R}\frac{|x'|^{k}}{\varepsilon+\kappa_{2}|x'|^{m}}\gtrsim\frac{(\lambda+2\mu)\eta}{\kappa_{2}^{\frac{d+k-1}{m}}}\rho_{k}(d,m;\varepsilon).
\end{align*}

{\bf Step 1.2.} From \eqref{Le2.01222}, we deduce that for $\alpha=1,2,...,d$,
\begin{align*}
Q_{\alpha}[\varphi]-Q^{\ast}_{\alpha}[\varphi]=\int_{\partial D}\frac{\partial(u_{\alpha}-u_{\alpha}^{\ast})}{\partial\nu_{0}}\Big|_{+}\cdot\varphi(x),
\end{align*}
where $u_{\alpha}^{\ast}$ and $u_{\alpha}$ are defined by \eqref{l03.001} and \eqref{P2.005}, respectively.

For $0<t\leq2R$, denote $\Omega_{t}^{\ast}:=\Omega^{\ast}\cap\{|x'|<t\}$. Introduce a scalar auxiliary function $\bar{v}^{\ast}\in C^{2}(\mathbb{R}^{d})$ such that $\bar{v}^{\ast}=1$ on $\partial D_{1}^{\ast}\setminus\{0\}$, $\bar{v}^{\ast}=0$ on $\partial D$, and
\begin{align*}
\bar{v}^{\ast}(x',x_{d})=\frac{x_{d}-h(x')}{h_{1}(x')-h(x')},\;\,\mathrm{in}\;\Omega_{2R}^{\ast},\quad\|\bar{v}^{\ast}\|_{C^{2}(\Omega^{\ast}\setminus\Omega^{\ast}_{R})}\leq C.
\end{align*}
Define a family of auxiliary functions as follows:
\begin{align*}
\bar{u}^{\ast}_{\alpha}=&\psi_{\alpha}\bar{v}^{\ast}+\mathcal{F}_{\alpha}^{\ast},\quad \alpha=1,2,...,\frac{d(d+1)}{2},
\end{align*}
where
\begin{align}\label{PASK001}
\mathcal{F}_{\alpha}^{\ast}=\frac{\lambda+\mu}{\mu}f(\bar{v}^{\ast})\psi^{d}_{\alpha}\sum^{d-1}_{i=1}\partial_{x_{i}}\delta\,e_{i}+\frac{\lambda+\mu}{\lambda+2\mu}f(\bar{v}^{\ast})\sum^{d-1}_{i=1}\psi^{i}_{\alpha}\partial_{x_{i}}\delta\,e_{d}.
\end{align}
Making use of conditions ({\bf{H1}})--({\bf{H2}}), we obtain that for $x\in\Omega_{R}^{\ast}$,
\begin{align}\label{Lw3.003}
|\nabla_{x'}(\bar{u}_{\alpha}-\bar{u}^{\ast}_{\alpha})|\leq\frac{C}{|x'|},\quad |\partial_{x_{d}}(\bar{u}_{\alpha}-\bar{u}^{\ast}_{\alpha})|\leq\frac{C\varepsilon}{|x'|^{m}(\varepsilon+|x'|^{m})}.
\end{align}
By applying Corollary \ref{thm86} to $u_{\alpha}^{\ast}$ defined in \eqref{l03.001}, we have
\begin{align}\label{Lw3.005}
|\nabla(u_{\alpha}^{\ast}-\bar{u}_{\alpha}^{\ast})|\leq C|x'|^{m-2},\quad x\in\Omega_{R}^{\ast}.
\end{align}
For $0<r<R$, write
\begin{align}\label{LNM656}
\mathcal{C}_{r}:=\left\{x\in\mathbb{R}^{d}\Big|\;|x'|<r,\,\frac{1}{2}\min_{|x'|\leq r}h(x')\leq x_{d}\leq\varepsilon+2\max_{|x'|\leq r}h_{1}(x')\right\}.
\end{align}

We next estimate the difference $|Q_{\alpha}[\varphi]-Q^{\ast}_{\alpha}[\varphi]|$ step by step. Observe that $u_{\alpha}-u_{\alpha}^{\ast}$ solves
\begin{align*}
\begin{cases}
\mathcal{L}_{\lambda,\mu}(u_{\alpha}-u_{\alpha}^{\ast})=0,&\mathrm{in}\;\,D\setminus(\overline{D_{1}\cup D_{1}^{\ast}}),\\
u_{\alpha}-u_{\alpha}^{\ast}=\psi_{\alpha}-u_{\alpha}^{\ast},&\mathrm{on}\;\,\partial D_{1}\setminus D_{1}^{\ast},\\
u_{\alpha}-u_{\alpha}^{\ast}=u_{\alpha}-\psi_{\alpha},&\mathrm{on}\;\,\partial D_{1}^{\ast}\setminus(D_{1}\cup\{0\}),\\
u_{\alpha}-u_{\alpha}^{\ast}=0,&\mathrm{on}\;\,\partial D.
\end{cases}
\end{align*}
To start with, we estimate $|u_{\alpha}-u_{\alpha}^{\ast}|$ on $\partial(D_{1}\cup D_{1}^{\ast})\setminus\mathcal{C}_{\varepsilon^{\gamma}}$, where $0<\gamma<1/2$ to be determined later. Recalling the definition of $u_{\alpha}^{\ast}$, it follows from the standard boundary and interior estimates of elliptic systems that
$$|\partial_{x_{d}}u_{\alpha}^{\ast}|\leq C,\quad\;\,\mathrm{in}\;\,\Omega^{\ast}\setminus\Omega^{\ast}_{R}.$$
Then for $x\in\partial D_{1}\setminus D_{1}^{\ast}$,
\begin{align}\label{Lw3.007}
|(u_{\alpha}-u_{\alpha}^{\ast})(x',x_{d})|=|u_{\alpha}^{\ast}(x',x_{d}-\varepsilon)-u_{\alpha}^{\ast}(x',x_{d})|\leq C\varepsilon.
\end{align}
Utilizing \eqref{ATCG001}, we obtain that for $x\in\partial D_{1}^{\ast}\setminus(D_{1}\cup\mathcal{C}_{\varepsilon^{\gamma}})$,
\begin{align}\label{Lw3.008}
|(u_{\alpha}-u_{\alpha}^{\ast})(x',x_{d})|=|u_{\alpha}(x',x_{d})-u_{\alpha}(x',x_{d}+\varepsilon)|\leq C\varepsilon^{1-m\gamma}.
\end{align}
From \eqref{ATCG001} and \eqref{Lw3.003}--\eqref{Lw3.005}, we deduce that for $x\in\Omega_{R}^{\ast}\cap\{|x'|=\varepsilon^{\gamma}\}$,
\begin{align*}
|\partial_{x_{d}}(u_{\alpha}-u_{\alpha}^{\ast})|\leq&|\partial_{x_{d}}(u_{\alpha}-\bar{u}_{\alpha})|+|\partial_{x_{d}}(\bar{u}_{\alpha}-\bar{u}_{\alpha}^{\ast})|+|\partial_{x_{d}}(u_{\alpha}^{\ast}-\bar{u}_{\alpha}^{\ast})|\\
\leq&C\left(\frac{1}{\varepsilon^{2m\gamma-1}}+1\right).
\end{align*}
This, together with the fact that $\bar{u}_{\alpha}-\bar{u}_{\alpha}^{\ast}=0$ on $\partial D$, reads that
\begin{align}\label{Lw3.009}
|(u_{\alpha}-u_{\alpha}^{\ast})(x',x_{d})|=&|(u_{\alpha}-u_{\alpha}^{\ast})(x',x_{d})-(u_{\alpha}-u_{\alpha}^{\ast})(x',h(x'))|\notag\\
\leq&C\big(\varepsilon^{1-m\gamma}+\varepsilon^{m\gamma}\big).
\end{align}
Let $\gamma=\frac{1}{2m}$. Then it follows from \eqref{Lw3.007}--\eqref{Lw3.009} that
$$|u_{\alpha}-u_{\alpha}^{\ast}|\leq C\varepsilon^{\frac{1}{2}},\quad\;\,\mathrm{on}\;\,\partial\big(D\setminus\big(\overline{D_{1}\cup D_{1}^{\ast}\cup\mathcal{C}_{\varepsilon^{\frac{1}{2m}}}}\big)\big),$$
which, in combination with the maximum principle for Lam\'{e} system in \cite{MMN2007}, yields that
\begin{align}\label{LMZQT001}
|u_{\alpha}-u_{\alpha}^{\ast}|\leq C\varepsilon^{\frac{1}{2}},\quad\;\,\mathrm{in}\;\,D\setminus\big(\overline{D_{1}\cup D_{1}^{\ast}\cup\mathcal{C}_{\varepsilon^{\frac{1}{2m}}}}\big).
\end{align}
Then from the standard interior and boundary estimates for Lam\'{e} system, it follows that for any $0<\bar{\gamma}<\frac{1}{2m}$,
$$|\nabla(u_{\alpha}-u_{\alpha}^{\ast})|\leq C\varepsilon^{m\bar{\gamma}},\quad\;\,\mathrm{in}\;\,D\setminus\big(\overline{D_{1}\cup D_{1}^{\ast}\cup\mathcal{C}_{\varepsilon^{\frac{1}{2m}-\bar{\gamma}}}}\big),$$
which indicates that
\begin{align}\label{Lw3.010}
|\mathcal{A}^{out}|:=\left|\int_{\partial D\setminus\mathcal{C}_{\varepsilon^{\frac{1}{2m}-\bar{\gamma}}}}\frac{\partial(u_{\alpha}-u_{\alpha}^{\ast})}{\partial\nu_{0}}\Big|_{+}\cdot\varphi(x)\right|\leq C\varepsilon^{m\bar{\gamma}}\|\varphi\|_{L^{\infty}(\partial D)},
\end{align}
where $0<\bar{\gamma}<\frac{1}{2m}$ will be determined later.

We next estimate the remainder as follows:
\begin{align*}
\mathcal{A}^{in}:=&\int_{\partial D\cap\mathcal{C}_{\varepsilon^{\frac{1}{2m}-\bar{\gamma}}}}\frac{\partial(u_{\alpha}-u_{\alpha}^{\ast})}{\partial\nu_{0}}\Big|_{+}\cdot\varphi(x)\\
=&\int_{\partial D\cap\mathcal{C}_{\varepsilon^{\frac{1}{2m}-\bar{\gamma}}}}\frac{\partial(\bar{u}_{\alpha}-\bar{u}_{\alpha}^{\ast})}{\partial\nu_{0}}\Big|_{+}\cdot\varphi(x)+\int_{\partial D\cap\mathcal{C}_{\varepsilon^{\frac{1}{2m}-\bar{\gamma}}}}\frac{\partial(w_{\alpha}-w_{\alpha}^{\ast})}{\partial\nu_{0}}\Big|_{+}\cdot\varphi(x)\\
=&:\mathcal{A}_{\bar{u}}+\mathcal{A}_{w},
\end{align*}
where $w_{\alpha}=u_{\alpha}-\bar{u}_{\alpha}$, $w_{\alpha}^{\ast}=u_{\alpha}^{\ast}-\bar{u}_{\alpha}^{\ast}$. On one hand, if $\alpha=1,...,d-1$, then
\begin{align*}
\mathcal{A}_{\bar{u}}=&\int_{\partial D\cap\mathcal{C}_{\varepsilon^{\frac{1}{2m}-\bar{\gamma}}}}\bigg\{\lambda\sum^{d}_{i=1}\partial_{x_{\alpha}}(\bar{u}^{\alpha}_{\alpha}-\bar{u}^{\ast\alpha}_{\alpha})\nu_{i}\varphi^{i}+\mu\partial_{x_{d}}(\bar{u}^{\alpha}_{\alpha}-\bar{u}^{\ast\alpha}_{\alpha})\nu_{\alpha}\varphi^{d}\\
&\qquad+\mu\sum^{d-1}_{i=1}\partial_{x_{i}}(\bar{u}^{\alpha}_{\alpha}-\bar{u}^{\ast\alpha}_{\alpha})(\nu_{\alpha}\varphi^{i}+\nu_{i}\varphi^{\alpha})+\lambda\sum^{d}_{i=1}\partial_{x_{d}}(\mathcal{F}_{\alpha}^{d}-\mathcal{F}_{\alpha}^{\ast d})\nu_{i}\varphi^{i}\notag\\
&\qquad+\mu\sum^{d}_{i=1}\partial_{x_{i}}(\mathcal{F}_{\alpha}^{d}-\mathcal{F}_{\alpha}^{\ast d})(\nu_{d}\varphi^{i}+\nu_{i}\varphi^{d})\bigg\}\\
&+\int_{\partial D\cap\mathcal{C}_{\varepsilon^{\frac{1}{2m}-\bar{\gamma}}}}\mu\partial_{x_{d}}(\bar{u}^{\alpha}_{\alpha}-\bar{u}^{\ast\alpha}_{\alpha})\nu_{d}\varphi^{\alpha}\\
=&:\mathcal{A}^{1}_{\bar{u}}+\mathcal{A}^{2}_{\bar{u}}.
\end{align*}
Combining \eqref{OPAKLN01}, \eqref{KHA002} and \eqref{PASK001}--\eqref{Lw3.003}, we obtain
\begin{align}\label{Lw3.011}
|\mathcal{A}^{1}_{\bar{u}}|\leq\int_{|x'|<\varepsilon^{\frac{1}{2m}-\bar{\gamma}}}C|x'|^{k-1}\leq C\varepsilon^{\left(\frac{1}{2m}-\bar{\gamma}\right)(d+k-2)},
\end{align}
while, for the second term $\mathcal{A}^{2}_{\bar{u}}$, we have

$(\rm{i})$ if condition ({\bf{E1}}) holds, then for $m<d+k-1$,
\begin{align}\label{Lw3.012}
|\mathcal{A}^{2}_{\bar{u}}|\leq\int_{|x'|<\varepsilon^{\frac{1}{2m}-\bar{\gamma}}}C|x'|^{k-m}\leq C\varepsilon^{\left(\frac{1}{2m}-\bar{\gamma}\right)(d+k-1-m)};
\end{align}

$(\rm{ii})$ if condition ({\bf{E2}}) or ({\bf{E3}}) holds, then in view of the fact that the integrand is odd and the integrating domain is symmetric, we have
\begin{align*}
\mathcal{A}^{2}_{\bar{u}}=0.
\end{align*}

On the other hand, if $\alpha=d$, then
\begin{align*}
\mathcal{A}_{\bar{u}}=&\int_{\partial D\cap\mathcal{C}_{\varepsilon^{\frac{1}{2m}-\bar{\gamma}}}}\bigg\{\lambda\sum^{d-1}_{i=1}\partial_{x_{d}}(\bar{u}^{d}_{d}-\bar{u}^{\ast d}_{d})\nu_{i}\varphi^{i}+\mu\sum^{d-1}_{i=1}\partial_{x_{i}}(\bar{u}^{\alpha}_{\alpha}-\bar{u}^{\ast\alpha}_{\alpha})(\nu_{\alpha}\varphi^{i}+\nu_{i}\varphi^{\alpha})\notag\\
&\qquad+\lambda\sum^{d-1}_{i=1}\sum^{d}_{j=1}\partial_{x_{i}}(\mathcal{F}_{d}^{i}-\mathcal{F}_{d}^{\ast i})\nu_{j}\varphi^{j}+\mu\sum^{d-1}_{i=1}\sum^{d}_{j=1}\partial_{x_{j}}(\mathcal{F}_{d}^{i}-\mathcal{F}_{d}^{\ast i})(\nu_{i}\varphi^{j}+\nu_{j}\varphi^{i})\bigg\}\\
&+\int_{\partial D\cap\mathcal{C}_{\varepsilon^{\frac{1}{2m}-\bar{\gamma}}}}(\lambda+2\mu)\partial_{x_{d}}(\bar{u}^{d}_{d}-\bar{u}^{\ast d}_{d})\nu_{d}\varphi^{d}\\
=&:\mathcal{A}^{1}_{\bar{u}}+\mathcal{A}^{2}_{\bar{u}}.
\end{align*}
By the same argument as in \eqref{Lw3.011}--\eqref{Lw3.012}, we have
\begin{align}\label{KAY001}
|\mathcal{A}^{1}_{\bar{u}}|\leq C\varepsilon^{\left(\frac{1}{2m}-\bar{\gamma}\right)(d+k-2)},
\end{align}
and
\begin{align}\label{LTAN001}
\mathcal{A}^{2}_{\bar{u}}=&
\begin{cases}
O(1)\varepsilon^{(\frac{1}{2m}-\bar{\gamma})(d+k-1-m)},&\text{if condition}\; ({\bf{E1}})\; \text{holds},\;m<d+k-1,\\
0,&\text{if condition}\;({\bf{E2}})\;\text{or}\;({\bf{E3}})\;\text{holds}.
\end{cases}
\end{align}

We now proceed to estimate $|\mathcal{A}_{w}|$. An immediate consequence of Corollary \ref{thm86} gives that for $0<|x'|\leq R$,
\begin{align}\label{Lw3.016}
|\nabla w_{\alpha}|+|\nabla w_{\alpha}^{\ast}|\leq C,\quad\alpha=1,2,...,d.
\end{align}
Since
\begin{align*}
\mathcal{A}_{w}=&\int_{\partial D\cap\mathcal{C}_{\varepsilon^{\frac{1}{2m}-\bar{\gamma}}}}\bigg\{\lambda\sum^{d}_{i,j=1}\partial_{x_{i}}(w_{\alpha}^{i}-w_{\alpha}^{\ast i})\nu_{j}\varphi^{j}\\
&\quad+\mu\sum^{d}_{i,j=1}[\partial_{x_{j}}(w_{\alpha}^{i}-w_{\alpha}^{\ast i})+\partial_{x_{i}}(w_{\alpha}^{j}-w_{\alpha}^{\ast j})]\nu_{j}\varphi^{i}\bigg\},
\end{align*}
then we deduce from \eqref{Lw3.016} that
\begin{align}\label{Lw3.017}
|\mathcal{A}_{w}|\leq\int_{|x'|<\varepsilon^{\frac{1}{2m}-\bar{\gamma}}}C|x'|^{k}\leq C\varepsilon^{\left(\frac{1}{2m}-\bar{\gamma}\right)(d+k-1)}.
\end{align}
By using \eqref{Lw3.010}--\eqref{LTAN001} and \eqref{Lw3.017}, we derive that for $\alpha=1,2,...,d,$

$(\rm{i})$ if condition ({\bf{E1}}) holds, by picking $\bar{\gamma}=\frac{d+k-1-m}{2m(d+k-1)}$, then
\begin{align*}
|Q_{\alpha}[\varphi]-Q^{\ast}_{\alpha}[\varphi]|\leq C\varepsilon^{\frac{d+k-1-m}{2(d+k-1)}},\quad m<d+k-1;
\end{align*}

$(\rm{ii})$ if condition ({\bf{E2}}) or ({\bf{E3}}) holds, by picking $\bar{\gamma}=\frac{d+k-2}{2m(m+d+k-2)}$, then
\begin{align*}
|Q_{\alpha}[\varphi]-Q^{\ast}_{\alpha}[\varphi]|\leq C\varepsilon^{\frac{d+k-2}{2(m+d+k-2)}}.
\end{align*}

{\bf Step 2.} Proofs of \eqref{KAT90}--\eqref{AUOP001}. We next divide into two substeps to complete the proofs of \eqref{KAT90}--\eqref{AUOP001}.

{\bf Step 2.1.} If condition ({\bf{E2}}) holds for $m\geq d+k$, then it follows from Corollary \ref{thm86} and \eqref{KHA001} again that
\begin{align*}
Q_{d+1}[\varphi]=&\sum^{d}_{i=1}\int_{\partial D_{1}}\frac{\partial u_{0i}}{\partial\nu_{0}}\Big|_{+}\cdot\psi_{d+1}\notag\\
=&\int_{\partial D_{1}}\bigg[\sum^{d}_{i,j=1}\left(\lambda\partial_{x_{j}}u_{0i}^{j}\nu_{1}+\mu(\partial_{x_{j}}u_{0i}^{1}+\partial_{x_{1}}u_{0i}^{j})\nu_{j}\right)x_{d}-\sum^{d}_{i,j=1}\lambda\partial_{x_{j}}u_{0i}^{j}\nu_{d}x_{1}\\
&\quad\quad\quad-\sum_{(i,j)\neq(d,d)}\mu(\partial_{x_{j}}u_{0i}^{d}+\partial_{x_{d}}u_{0i}^{j})\nu_{j}x_{1}-(\lambda+2\mu)\partial_{x_{d}}(u_{0d}^{d}-\bar{u}_{0d}^{d})\nu_{d}x_{1}\bigg]\\
&-\int_{\partial D_{1}}(\lambda+2\mu)\partial_{x_{d}}\bar{u}_{0d}^{d}\nu_{d}x_{1}\\
=&(\lambda+2\mu)\int_{\partial D_{1}\cap\Gamma^{+}_{R}}\frac{\varphi^{d}(x',h(x'))x_{1}}{\varepsilon+h_{1}(x')-h(x')}\nu_{d}+O(1)\|\varphi\|_{C^{2}(\partial D)},
\end{align*}
where we used the fact that $|x_{d}|=|\varepsilon+h_{1}(x')|\leq C(\varepsilon+|x'|^{m})\;\mathrm{on}\;\Gamma_{R}^{+}.$ This yields that
\begin{align*}
Q_{d+1}[\varphi]\lesssim&(\lambda+2\mu)\eta\int_{|x'|<R}\frac{|x_{1}|^{k+1}}{\varepsilon+\kappa_{1}|x'|^{m}}\lesssim(\lambda+2\mu)\eta\int_{0}^{R}\frac{s^{d+k-1}}{\varepsilon+\kappa_{1}s^{m}}\notag\\
\lesssim&\frac{(\lambda+2\mu)\eta}{\kappa_{1}^{\frac{d+k}{m}}}\rho_{k+1}(d,m;\varepsilon),
\end{align*}
and
\begin{align*}
Q_{d+1}[\varphi]\gtrsim&(\lambda+2\mu)\eta\int_{|x'|<R}\frac{|x_{1}|^{k+1}}{\varepsilon+\kappa_{2}|x'|^{m}}\gtrsim\frac{(\lambda+2\mu)\eta}{\kappa_{2}^{\frac{d+k}{m}}}\rho_{k+1}(d,m;\varepsilon).
\end{align*}

{\bf Step 2.2.} First, using \eqref{Le2.01222}, we have
\begin{align*}
Q_{d+1}[\varphi]-Q^{\ast}_{d+1}[\varphi]=\int_{\partial D}\frac{\partial(u_{d+1}-u_{d+1}^{\ast})}{\partial\nu_{0}}\Big|_{+}\cdot\varphi(x),
\end{align*}
where $u_{d+1}^{\ast}$ and $u_{d+1}$ are, respectively, defined by \eqref{l03.001} and \eqref{P2.005} with picking $\alpha=d+1$.

Similarly as above, it follows from ({\bf{H1}})--({\bf{H2}}) that for $x\in\Omega_{R}^{\ast}$,
\begin{align}\label{LKT6.003}
|\nabla_{x'}(\bar{u}_{d+1}-\bar{u}^{\ast}_{d+1})|\leq C,\quad|\partial_{x_{d}}(\bar{u}_{d+1}-\bar{u}^{\ast}_{d+1})|\leq\frac{C\varepsilon}{|x'|^{m-1}(\varepsilon+|x'|^{m})}.
\end{align}
Applying Corollary \ref{thm86} to $u_{d+1}^{\ast}$, we derive
\begin{align}\label{LKT6.005}
|\nabla(u_{d+1}^{\ast}-\bar{u}_{d+1}^{\ast})|\leq C,\quad x\in\Omega_{R}^{\ast}.
\end{align}

Observe that $u_{d+1}-u_{d+1}^{\ast}$ satisfies
\begin{align*}
\begin{cases}
\mathcal{L}_{\lambda,\mu}(u_{d+1}-u_{d+1}^{\ast})=0,&\mathrm{in}\;\,D\setminus(\overline{D_{1}\cup D_{1}^{\ast}}),\\
u_{d+1}-u_{d+1}^{\ast}=\psi_{d+1}-u_{d+1}^{\ast},&\mathrm{on}\;\,\partial D_{1}\setminus D_{1}^{\ast},\\
u_{d+1}-u_{d+1}^{\ast}=u_{d+1}-\psi_{d+1},&\mathrm{on}\;\,\partial D_{1}^{\ast}\setminus(D_{1}\cup\{0\}),\\
u_{d+1}-u_{d+1}^{\ast}=0,&\mathrm{on}\;\,\partial D.
\end{cases}
\end{align*}
We first give the estimation of $|u_{d+1}-u_{d+1}^{\ast}|$ on $\partial(D_{1}\cup D_{1}^{\ast})\setminus\mathcal{C}_{\varepsilon^{\gamma}}$, where $\mathcal{C}_{\varepsilon^{\gamma}}$ is defined in (\ref{LNM656}) and $\frac{1}{2m-1}<\gamma<1/2$ to be chosen later. Analogously as before, by using the standard boundary and interior estimates of elliptic systems, we obtain that for $x\in\partial D_{1}\setminus D_{1}^{\ast}$,
\begin{align}\label{LKT6.007}
|(u_{d+1}-u_{d+1}^{\ast})(x',x_{d})|=|u_{d+1}^{\ast}(x',x_{d}-\varepsilon)-u^{\ast}_{d+1}(x',x_{d})|\leq C\varepsilon.
\end{align}
Utilizing \eqref{ATCG001}, we get that for $x\in\partial D_{1}^{\ast}\setminus(D_{1}\cup\mathcal{C}_{\varepsilon^{\gamma}})$,
\begin{align}\label{LKT6.008}
|(u_{d+1}-u_{d+1}^{\ast})(x',x_{d})|=|u_{d+1}(x',x_{d})-u_{d+1}(x',x_{d}+\varepsilon)|\leq C\varepsilon^{1-(m-1)\gamma}.
\end{align}
From Corollary \ref{thm86} and \eqref{LKT6.003}--\eqref{LKT6.005}, it follows that for $x\in\Omega_{R}^{\ast}\cap\{|x'|=\varepsilon^{\gamma}\}$,
\begin{align*}
|\partial_{x_{d}}(u_{d+1}-u_{d+1}^{\ast})|\leq&|\partial_{x_{d}}(u_{d+1}-\bar{u}_{d+1})|+|\partial_{x_{d}}(\bar{u}_{d+1}-\bar{u}^{\ast}_{d+1})|+|\partial_{x_{d}}(u_{d+1}^{\ast}-\bar{u}^{\ast}_{d+1})|\notag\\
\leq&\frac{C}{\varepsilon^{(2m-1)\gamma-1}}.
\end{align*}
This, in combination with the fact that $\bar{u}_{d+1}-\bar{u}_{d+1}^{\ast}=0$ on $\partial D$, yields that
\begin{align}\label{LKT6.009}
|(u_{d+1}-u_{d+1}^{\ast})(x',x_{d})|=&|(u_{d+1}-u_{d+1}^{\ast})(x',x_{d})-(u_{d+1}-u_{d+1}^{\ast})(x',h(x'))|\notag\\
\leq& C\varepsilon^{1-(m-1)\gamma},\quad\;\,\mathrm{for}\;\,x\in\Omega_{R}^{\ast}\cap\{|x'|=\varepsilon^{\gamma}\}.
\end{align}
Take $\gamma=\frac{1}{m}$. From \eqref{LKT6.007}--\eqref{LKT6.009}, we have
$$|u_{d+1}-u_{d+1}^{\ast}|\leq C\varepsilon^{\frac{1}{m}},\quad\;\,\mathrm{on}\;\,\partial\big(D\setminus\big(\overline{D_{1}\cup D_{1}^{\ast}\cup\mathcal{C}_{\varepsilon^{\frac{1}{m}}}}\big)\big),$$
which, together with the maximum principle for the Lam\'{e} system in \cite{MMN2007}, reads that
$$|u_{d+1}-u_{d+1}^{\ast}|\leq C\varepsilon^{\frac{1}{m}},\quad\;\,\mathrm{on}\;\,D\setminus\big(\overline{D_{1}\cup D_{1}^{\ast}\cup\mathcal{C}_{\varepsilon^{\frac{1}{m}}}}\big).$$
Then using the standard interior and boundary estimates for the Lam\'{e} system, we get that for any $\frac{m-1}{m^{2}}<\bar{\gamma}<\frac{1}{m}$,
$$|\nabla(u_{d+1}-u_{d+1}^{\ast})|\leq C\varepsilon^{m\bar{\gamma}-\frac{m-1}{m}},\quad\;\,\mathrm{on}\;\,\partial D\setminus\mathcal{C}_{\varepsilon^{\frac{1}{m}-\bar{\gamma}}},$$
from which we obtain
\begin{align}\label{LKT6.010}
|\mathcal{A}^{out}|:=\left|\int_{\partial D\setminus\mathcal{C}_{\varepsilon^{\frac{1}{m}-\bar{\gamma}}}}\frac{\partial(u_{d+1}-u_{d+1}^{\ast})}{\partial\nu_{0}}\Big|_{+}\cdot\varphi(x)\right|\leq C\varepsilon^{m\bar{\gamma}-\frac{m-1}{m}},
\end{align}
where $\frac{m-1}{m^{2}}<\bar{\gamma}<\frac{1}{m}$ will be given in the following.

It remains to estimate the residual part $\mathcal{A}^{in}$ as follows:
\begin{align*}
\mathcal{A}^{in}:=&\int_{\partial D\cap\mathcal{C}_{\varepsilon^{\frac{1}{m}-\bar{\gamma}}}}\frac{\partial(u_{d+1}-u_{d+1}^{\ast})}{\partial\nu_{0}}\Big|_{+}\cdot\varphi(x)\\
=&\int_{\partial D\cap\mathcal{C}_{\varepsilon^{\frac{1}{m}-\bar{\gamma}}}}\frac{\partial(w_{d+1}-w_{d+1}^{\ast})}{\partial\nu_{0}}\Big|_{+}\cdot\varphi(x)+\int_{\partial D\cap\mathcal{C}_{\varepsilon^{\frac{1}{m}-\bar{\gamma}}}}\frac{\partial(\bar{u}_{d+1}-\bar{u}_{d+1}^{\ast})}{\partial\nu_{0}}\Big|_{+}\cdot\varphi(x)\\
=&:\mathcal{A}_{w}+\mathcal{A}_{\bar{u}},
\end{align*}
where $w_{d+1}=u_{d+1}-\bar{u}_{d+1}$, $w_{d+1}^{\ast}=u_{d+1}^{\ast}-\bar{u}_{d+1}^{\ast}$. Analogously as above, it follows from Corollary \ref{thm86} that
\begin{align}\label{TYU616}
|\nabla w_{d+1}(x)|\leq C,\;\,\,|\nabla w_{d+1}^{\ast}(x)|\leq C,\quad x\in\Omega_{R}^{\ast}.
\end{align}
By definition,
\begin{align*}
\mathcal{A}_{w}=&\int_{\partial D\cap\mathcal{C}_{\varepsilon^{\frac{1}{m}-\bar{\gamma}}}}\bigg\{\lambda\sum^{d}_{i,j=1}\partial_{x_{i}}(w_{d+1}^{i}-w_{d+1}^{\ast i})\nu_{j}\varphi^{j}(x)\\
&+\mu\sum^{d}_{i,j=1}[\partial_{x_{j}}(w_{d+1}^{i}-w_{d+1}^{\ast i})+\partial_{x_{i}}(w_{d+1}^{j}-w_{d+1}^{\ast j})]\nu_{j}\varphi^{i}(x)\bigg\}.
\end{align*}
Then from \eqref{KHA002} and \eqref{TYU616}, we get
\begin{align}\label{ZZ001}
|\mathcal{A}_{w}|\leq\int_{\partial D\cap\mathcal{C}_{\varepsilon^{\frac{1}{m}-\bar{\gamma}}}}C|x'|^{k}\leq C\varepsilon^{(\frac{1}{m}-\bar{\gamma})(d+k-1)}.
\end{align}

We next estimate the second term $\mathcal{A}_{\bar{u}}$. Recalling the definitions of $\mathcal{F}_{d+1}$ and $\mathcal{F}^{\ast}_{d+1}$, we obtain
\begin{align*}
|\nabla(\mathcal{F}_{d+1}-\mathcal{F}_{d+1}^{\ast})|\leq C,
\end{align*}
which reads that
\begin{align*}
\left|\int_{\partial D\cap\mathcal{C}_{\varepsilon^{\frac{1}{m}-\bar{\gamma}}}}\frac{\partial(\mathcal{F}_{d+1}-\mathcal{F}_{d+1}^{\ast})}{\partial\nu_{0}}\Big|_{+}\cdot\varphi(x)\right|\leq C\varepsilon^{(\frac{1}{m}-\bar{\gamma})(d+k-1)}.
\end{align*}
Then
\begin{align*}
\mathcal{A}_{\bar{u}}=&\int_{\partial D\cap\mathcal{C}_{\varepsilon^{\frac{1}{m}-\bar{\gamma}}}}\frac{\partial(\bar{u}_{d+1}-\bar{u}_{d+1}^{\ast})}{\partial\nu_{0}}\Big|_{+}\cdot\varphi(x)\notag\\
=&\int_{\partial D\cap\mathcal{C}_{\varepsilon^{\frac{1}{m}-\bar{\gamma}}}}\left(\frac{\partial(\psi_{d+1}\bar{v}-\psi_{d+1}\bar{v}^{\ast})}{\partial\nu_{0}}\Big|_{+}\cdot\varphi(x)+\frac{\partial(\mathcal{F}_{d+1}-\mathcal{F}_{d+1}^{\ast})}{\partial\nu_{0}}\Big|_{+}\cdot\varphi(x)\right)\notag\\
=&\int_{\partial D\cap\mathcal{C}_{\varepsilon^{\frac{1}{m}-\bar{\gamma}}}}\frac{\partial(\psi_{d+1}\bar{v}-\psi_{d+1}\bar{v}^{\ast})}{\partial\nu_{0}}\Big|_{+}\cdot\varphi(x)+O(1)\varepsilon^{(\frac{1}{m}-\tilde{\gamma})(d+k-1)}.
\end{align*}
Denote
\begin{align*}
\mathcal{A}_{\bar{u}}^{1}=\int_{\partial D\cap\mathcal{C}_{\varepsilon^{\frac{1}{m}-\bar{\gamma}}}}\frac{\partial(\psi_{d+1}\bar{v}-\psi_{d+1}\bar{v}^{\ast})}{\partial\nu_{0}}\Big|_{+}\cdot\varphi(x).
\end{align*}
We further decompose $\mathcal{A}_{\bar{u}}^{1}$ into two parts as follows:
\begin{align*}
\mathcal{A}_{\bar{u}}^{11}=&\int_{\partial D\cap\mathcal{C}_{\varepsilon^{\frac{1}{m}-\bar{\gamma}}}}\sum^{d}_{i=1}\left[\lambda x_{d}\partial_{x_{1}}(\bar{v}-\bar{v}^{\ast})\nu_{i}\varphi^{i}+\mu\partial_{x_{i}}(x_{d}\bar{v}-x_{d}\bar{v}^{\ast})(\nu_{i}\varphi^{1}+\nu_{1}\varphi^{i})\right]\notag\\
&-\int_{\partial D\cap\mathcal{C}_{\varepsilon^{\frac{1}{m}-\bar{\gamma}}}}\sum^{d-1}_{i=1}\left[\lambda x_{1}\partial_{x_{d}}(\bar{v}-\bar{v}^{\ast})\nu_{i}\varphi^{i}+\mu\partial_{x_{i}}(x_{1}\bar{v}-x_{1}\bar{v}^{\ast})(\nu_{i}\varphi^{d}+\nu_{d}\varphi^{i})\right],\notag\\
\mathcal{A}_{\bar{u}}^{12}=&-\int_{\partial D\cap\mathcal{C}_{\varepsilon^{\frac{1}{m}-\bar{\gamma}}}}(\lambda+2\mu) x_{1}\partial_{x_{d}}(\bar{v}-\bar{v}^{\ast})\nu_{d}\varphi^{d}.
\end{align*}
With regard to $\mathcal{A}_{\bar{u}}^{11}$, in view of the fact that $|x_{d}|=|h(x')|\leq C|x'|^{m}$ on $\Gamma_{R}^{-}$, we get that if condition ({\bf{E1}}), ({\bf{E2}}) or ({\bf{E3}}) holds,
\begin{align}\label{LATN001}
|\mathcal{A}_{\bar{u}}^{11}|\leq\int_{\partial D\cap\mathcal{C}_{\varepsilon^{\frac{1}{m}-\bar{\gamma}}}}C|x'|^{k}\leq C\varepsilon^{(\frac{1}{m}-\bar{\gamma})(d+k-1)}.
\end{align}

For the second term $\mathcal{A}_{\bar{u}}^{12}$, we have
\begin{align}\label{HJAN001}
\mathcal{A}_{\bar{u}}^{12}=&
\begin{cases}
O(1)\varepsilon^{(\frac{1}{m}-\bar{\gamma})(d+k-m)},&\text{if condition}\; ({\bf{E2}})\; \text{holds},\;m<d+k,\\
0,&\text{if condition}\;({\bf{E1}})\;\text{or}\;({\bf{E3}})\;\text{holds},
\end{cases}
\end{align}
where we used the fact that the integrand is odd with respect to $x_{1}$ in the case when condition ({\bf{E1}}) or ({\bf{E3}}) holds. Then from \eqref{LKT6.010} and \eqref{ZZ001}--\eqref{HJAN001}, we obtain that

$(a)$ if condition ({\bf{E1}}) or ({\bf{E3}}) holds, by taking $\bar{\gamma}=\frac{m+d+k-2}{m(m+d+k-1)}$, then
\begin{align*}
|Q_{d+1}[\varphi]-Q^{\ast}_{d+1}[\varphi]|\leq C\varepsilon^{\frac{d+k-1}{m(m+d+k-1)}};
\end{align*}

$(b)$ if condition ({\bf{E2}}) holds, by taking $\bar{\gamma}=\frac{d+k-1}{m(d+k)}$, then
\begin{align*}
|Q_{d+1}[\varphi]-Q^{\ast}_{d+1}[\varphi]|\leq C\varepsilon^{\frac{d+k-m}{m(d+k)}}.
\end{align*}

Consequently, we complete the proofs of \eqref{KAT90}--\eqref{AUOP001}.

{\bf Step 3.} Proof of \eqref{TAKL001}. Observe that for $i=1,2,...,d$ and $j=1,...,d-1$, $i\neq j$,

$(\rm{i})$ if condition ({\bf{E1}}) holds, then $\varphi^{i}(x',h(x'))x_{j}$ is odd with respect to $x_{j}$;

$(\rm{ii})$ if condition ({\bf{E2}}) holds, then $\varphi^{i}(x',h(x'))x_{j}=0$ for $i=1,...,d-1$, and $\varphi^{d}(x',h(x'))x_{j}$ is odd with respect to $x_{1}$ for $j=2,...,d-1$, $d\geq3$;

$(\rm{iii})$ if condition ({\bf{E3}}) holds, then $\varphi^{i}(x',h(x'))x_{j}$ is odd with respect to $x_{i}$ for $i=1,...,d-1$, and $\varphi^{d}(x',h(x'))x_{j}=0$.

Then following the same argument as in {\bf Step 2} with a slight modification, we deduce that for $\alpha=d+2,...,\frac{d(d+1)}{2}$, $d\geq3$,
\begin{align*}
|Q_{\alpha}[\varphi]-Q^{\ast}_{\alpha}[\varphi]|\leq C\varepsilon^{\frac{d+k-1}{m(m+d+k-1)}}.
\end{align*}
That is, \eqref{TAKL001} holds.

\end{proof}

\subsection{Estimates and asymptotics of $a_{\alpha\alpha}$, $\alpha=1,2,...,\frac{d(d+1)}{2}$.}\label{subsec32}

Multiplying the first line of \eqref{P2.005} by $u_{\beta}$, it follows from \eqref{Le2.01222} that
\begin{align*}
a_{\alpha\beta}=\int_{\Omega}(\mathbb{C}^{0}e(u_{\alpha}),e(u_{\beta})),\quad\alpha,\beta=1,2,...,\frac{d(d+1)}{2}.
\end{align*}
\begin{lemma}\label{lemmabc}
Assume as above. Then for a sufficiently small $\varepsilon>0$,

$(\rm{i})$ if $\alpha=1,2,...,d$, then
\begin{align}\label{LMC1}
\begin{cases}
\kappa_{2}^{-\frac{d-1}{m}}\lesssim\frac{a_{\alpha\alpha}}{\mathcal{L}_{d}^{\alpha}\rho_{0}(d,m;\varepsilon)}\lesssim\kappa_{1}^{-\frac{d-1}{m}},&m\geq d-1,\\
a_{\alpha\alpha}=a_{\alpha\alpha}^{\ast}+O(\varepsilon^{\min\{\frac{1}{6},\frac{d-1-m}{12m}\}}),&m<d-1;
\end{cases}
\end{align}

$(\rm{ii})$ if $\alpha=d+1,...,\frac{d(d+1)}{2}$, then
\begin{align}\label{LMC}
\begin{cases}
\kappa_{2}^{-\frac{d+1}{m}}\lesssim\frac{a_{\alpha\alpha}}{\mathcal{L}_{d}^{\alpha}\rho_{2}(d,m;\varepsilon)}\lesssim\kappa_{1}^{-\frac{d+1}{m}},&m\geq d+1,\\
a_{\alpha\alpha}=a_{\alpha\alpha}^{\ast}+O(\varepsilon^{\min\{\frac{1}{6},\frac{d+1-m}{12m}\}}),&m<d+1;
\end{cases}
\end{align}

$(\rm{iii})$ if $d=2$, for $\alpha,\beta=1,2,\alpha\neq\beta$, then
\begin{align}\label{LVZQ001}
a_{12}=a_{21}=O(1)|\ln\varepsilon|,
\end{align}
and if $d\geq3$, for $\alpha,\beta=1,2,...,d,\,\alpha\neq\beta$, then
\begin{align}\label{ADzc}
a_{\alpha\beta}=a_{\beta\alpha}=a^{\ast}_{\alpha\beta}+O(1)\varepsilon^{\min\{\frac{1}{6},\frac{d-2}{12m}\}},
\end{align}
and if $d\geq2$, for $\alpha=1,2,...,d,\,\beta=d+1,...,\frac{d(d+1)}{2},$ then
\begin{align}\label{LVZQ0011gdw}
a_{\alpha\beta}=a_{\beta\alpha}=a^{\ast}_{\alpha\beta}+O(1)\varepsilon^{\min\{\frac{1}{6},\frac{d-1}{12m}\}},
\end{align}
and if $d\geq3$, for $\alpha,\beta=d+1,...,\frac{d(d+1)}{2},\,\alpha\neq\beta$, then
\begin{align}\label{LVZQ0011}
a_{\alpha\beta}=a_{\beta\alpha}=a^{\ast}_{\alpha\beta}+O(1)\varepsilon^{\min\{\frac{1}{6},\frac{d}{12m}\}},
\end{align}
where $a_{\alpha\beta}^{\ast}$, $\alpha,\beta=1,2,...,\frac{d(d+1)}{2}$ are defined in \eqref{FNCL001}.
\end{lemma}

\begin{proof}
{\bf Step 1.} Proof of (\ref{LMC1}). Fix $\bar{\gamma}=\frac{1}{12m}$. For $\alpha=1,2,...,d$, we begin with a decomposition for $a_{\alpha\alpha}$ as follows:
\begin{align}\label{a1111}
a_{\alpha\alpha}=&\int_{\Omega_{\varepsilon^{\bar{\gamma}}}}(\mathbb{C}^{0}e(u_{\alpha}),e(u_{\alpha}))+\int_{\Omega\setminus\Omega_{R}}(\mathbb{C}^{0}e(u_{\alpha}),e(u_{\alpha}))+\int_{\Omega_{R}\setminus\Omega_{\varepsilon^{\bar{\gamma}}}}(\mathbb{C}^{0}e(u_{\alpha}),e(u_{\alpha}))\nonumber\\
=&:\mathrm{I}+\mathrm{II}+\mathrm{III}.
\end{align}
In light of the definitions of $\bar{u}_{\alpha}$ and $\mathbb{C}^{0}$, it follows from a straightforward calculation that
\begin{align*}
(\mathbb{C}^{0}e(\bar{u}_{\alpha}),e(\bar{u}_{\alpha}))=&(\lambda+\mu)(\partial_{x_{\alpha}}\bar{v})^{2}+\mu\sum\limits^{d}_{i=1}(\partial_{x_{i}}\bar{v})^{2}+(\mathbb{C}^{0}e(\mathcal{F}_{\alpha}),e(\mathcal{F}_{\alpha}))\\
&+2(\mathbb{C}^{0}e(\psi_{\alpha}\bar{v}),e(\mathcal{F}_{\alpha})),\quad\alpha=1,2,...,d,
\end{align*}
which, in combination with Corollary \ref{thm86}, gives that
\begin{align}\label{con03365}
\mathrm{I}=&\int_{\Omega_{\varepsilon^{\bar{\gamma}}}}(\mathbb{C}^{0}e(\bar{u}_{\alpha}),e(\bar{u}_{\alpha}))+2\int_{\Omega_{\varepsilon^{\bar{\gamma}}}}(\mathbb{C}^{0}e(u_{\alpha}-\bar{u}_{\alpha}),e(\bar{u}_{\alpha}))\notag\\
&+\int_{\Omega_{\varepsilon^{\bar{\gamma}}}}(\mathbb{C}^{0}e(u_{\alpha}-\bar{u}_{\alpha}),e(u_{\alpha}-\bar{u}_{\alpha}))\notag\\
=&\,\mathcal{L}_{d}^{\alpha}\int_{|x'|<\varepsilon^{\bar{\gamma}}}\frac{dx'}{\varepsilon+h_{1}(x')-h(x')}+O(1)\varepsilon^{(d-1)\bar{\gamma}},
\end{align}
where $\mathcal{F}_{\alpha}$ is defined in \eqref{OPAKLN01} and $\mathcal{L}_{d}^{\alpha}$ is defined by \eqref{AZ}--\eqref{AZ110}.

To estimate the latter two terms, we first analyze the difference $|\nabla(u_{\alpha}-u_{\alpha}^{\ast})|$ in $D\setminus\big(\overline{D_{1}\cup D_{1}^{\ast}\cup\mathcal{C}_{\varepsilon^{\bar{\gamma}}}}\big)$. For $\varepsilon^{\bar{\gamma}}\leq|z'|\leq R$, by carrying out a change of variable as follows:
\begin{align*}
\begin{cases}
x'-z'=|z'|^{m}y',\\
x_{d}=|z'|^{m}y_{d},
\end{cases}
\end{align*}
we rescale $\Omega_{|z'|+|z'|^{m}}\setminus\Omega_{|z'|}$ and $\Omega_{|z'|+|z'|^{m}}^{\ast}\setminus\Omega_{|z'|}^{\ast}$ into $Q_{1}$ and $Q_{1}^{\ast}$ of nearly unit-size squares (or cylinders), respectively. Denote
\begin{align*}
U_{\alpha}(y):=u_{\alpha}(z'+|z'|^{m}y',|z'|^{m}y_{d}),\quad\mathrm{in}\;Q_{1},
\end{align*}
and
\begin{align*}
U^{\ast}_{\alpha}(y):=u^{\ast}_{\alpha}(z'+|z'|^{m}y',|z'|^{m}y_{d}),\quad\mathrm{in}\;Q_{1}^{\ast}.
\end{align*}
Then we conclude from the standard interior and boundary estimates of elliptic systems that
\begin{align*}
|\nabla^{2}U_{\alpha}|\leq C,\quad\mathrm{in}\;Q_{1},\quad\mathrm{and}\;|\nabla^{2}U^{\ast}_{\alpha}|\leq C,\quad\mathrm{in}\;Q_{1}^{\ast}.
\end{align*}
By utilizing an interpolation with \eqref{LMZQT001}, we get
\begin{align*}
|\nabla(U_{\alpha}-U^{\ast}_{\alpha})|\leq C\varepsilon^{\frac{1}{2}(1-\frac{1}{2})}\leq C\varepsilon^{\frac{1}{4}}.
\end{align*}
Then rescaling it back to $u_{\alpha}-u_{\alpha}^{\ast}$ and in view of $\varepsilon^{\bar{\gamma}}\leq|z'|\leq R$, we have
\begin{align*}
|\nabla(u_{\alpha}-u_{\alpha}^{\ast})(x)|\leq C\varepsilon^{\frac{1}{4}}|z'|^{-m}\leq C\varepsilon^{\frac{1}{6}},\quad x\in\Omega^{\ast}_{|z'|+|z'|^{m}}\setminus\Omega_{|z'|}^{\ast},
\end{align*}
which implies that for $\alpha=1,2,...,d$,
\begin{align}\label{con035}
|\nabla(u_{\alpha}-u_{\alpha}^{\ast})|\leq C\varepsilon^{\frac{1}{6}},\quad\;\,\mathrm{in}\;\,D\setminus\big(\overline{D_{1}\cup D_{1}^{\ast}\cup\mathcal{C}_{\varepsilon^{\bar{\gamma}}}}\big).
\end{align}
Then from \eqref{con035}, we obtain
\begin{align}\label{KKAA123}
\mathrm{II}=&\int_{D\setminus(D_{1}\cup D_{1}^{\ast}\cup\Omega_{R})}(\mathbb{C}^{0}e(u_{\alpha}),e(u_{\alpha}))+O(1)\varepsilon\notag\\
=&\int_{D\setminus(D_{1}\cup D_{1}^{\ast}\cup\Omega_{R})}\left((\mathbb{C}^{0}e(u_{\alpha}^{\ast}),e(u_{\alpha}^{\ast}))+2(\mathbb{C}^{0}e(u_{\alpha}-u_{\alpha}^{\ast}),e(u_{\alpha}^{\ast}))\right)\notag\\
&+\int_{D\setminus(D_{1}\cup D_{1}^{\ast}\cup\Omega_{R})}(\mathbb{C}^{0}e(u_{\alpha}-u_{\alpha}^{\ast}),e(u_{\alpha}-u_{\alpha}^{\ast}))+O(1)\varepsilon\notag\\
=&\int_{\Omega^{\ast}\setminus\Omega^{\ast}_{R}}(\mathbb{C}^{0}e(u_{\alpha}^{\ast}),e(u_{\alpha}^{\ast}))+O(1)\varepsilon^{\frac{1}{6}},
\end{align}
where we used the fact that $|\nabla u_{\alpha}|$ remains bounded in $(D_{1}^{\ast}\setminus(D_{1}\cup\Omega_{R}))\cup(D_{1}\setminus D_{1}^{\ast})$ and the volumes of $D_{1}^{\ast}\setminus(D_{1}\cup\Omega_{R})$ and $D_{1}\setminus D_{1}^{\ast}$ are of order $O(\varepsilon)$.

With regard to the last term $\mathrm{III}$ of \eqref{a1111}, we first decompose it as follows:
\begin{align*}
\mathrm{III}_{1}=&\int_{(\Omega_{R}\setminus\Omega_{\varepsilon^{\bar{\gamma}}})\setminus(\Omega^{\ast}_{R}\setminus\Omega^{\ast}_{\varepsilon^{\bar{\gamma}}})}(\mathbb{C}^{0}e(u_{\alpha}),e(u_{\alpha})),\\
\mathrm{III}_{2}=&\int_{\Omega^{\ast}_{R}\setminus\Omega^{\ast}_{\varepsilon^{\bar{\gamma}}}}(\mathbb{C}^{0}e(u_{\alpha}-u_{\alpha}^{\ast}),e(u_{\alpha}-u^{\ast}_{\alpha}))+2\int_{\Omega^{\ast}_{R}\setminus\Omega^{\ast}_{\varepsilon^{\bar{\gamma}}}}(\mathbb{C}^{0}e(u_{\alpha}-u_{\alpha}^{\ast}),e(u_{\alpha}^{\ast})),\\
\mathrm{III}_{3}=&\int_{\Omega^{\ast}_{R}\setminus\Omega^{\ast}_{\varepsilon^{\bar{\gamma}}}}(\mathbb{C}^{0}e(u_{\alpha}^{\ast}),e(u_{\alpha}^{\ast})).
\end{align*}
Since the thickness of $(\Omega_{R}\setminus\Omega_{\varepsilon^{\bar{\gamma}}})\setminus(\Omega^{\ast}_{R}\setminus\Omega^{\ast}_{\varepsilon^{\bar{\gamma}}})$ is $\varepsilon$, then we deduce from Corollary \ref{thm86} that
\begin{align}\label{con0333355}
|\mathrm{III}_{1}|\leq\,C\varepsilon\int_{\varepsilon^{\bar{\gamma}}<|x'|<R}\frac{dx'}{|x'|^{2m}}\leq &C
\begin{cases}
\varepsilon,&m<\frac{d-1}{2},\\
\varepsilon|\ln\varepsilon|,&m=\frac{d-1}{2},\\
\varepsilon^{\frac{10m+d-1}{12m}},&m>\frac{d-1}{2}.
\end{cases}
\end{align}

On the other hand, a consequence of \eqref{Lw3.005} and \eqref{con035} gives
\begin{align}\label{con036666}
|\mathrm{III}_{2}|\leq\,C\varepsilon^{\frac{1}{6}}.
\end{align}
Using \eqref{Lw3.005} again, we derive
\begin{align}\label{GNAT001}
\mathrm{III}_{3}=&\int_{\Omega_{R}^{\ast}\setminus\Omega^{\ast}_{\varepsilon^{\bar{\gamma}}}}(\mathbb{C}^{0}e(\bar{u}_{\alpha}^{\ast}),e(\bar{u}_{\alpha}^{\ast}))+2\int_{\Omega_{R}^{\ast}\setminus\Omega^{\ast}_{\varepsilon^{\bar{\gamma}}}}(\mathbb{C}^{0}e(u_{\alpha}^{\ast}-\bar{u}_{\alpha}^{\ast}),e(\bar{u}_{\alpha}^{\ast}))\notag\\
&+\int_{\Omega_{R}^{\ast}\setminus\Omega^{\ast}_{\varepsilon^{\bar{\gamma}}}}(\mathbb{C}^{0}e(u_{\alpha}^{\ast}-\bar{u}_{\alpha}^{\ast}),e(u_{\alpha}^{\ast}-\bar{u}_{\alpha}^{\ast}))\notag\\
=&\,\mathcal{L}_{d}^{\alpha}\int_{\varepsilon^{\bar{\gamma}}<|x'|<R}\frac{dx'}{h_{1}-h}-\int_{\Omega^{\ast}\setminus\Omega^{\ast}_{R}}(\mathbb{C}^{0}e(u_{\alpha}^{\ast}),e(u_{\alpha}^{\ast}))\notag\\
&+M_{d}^{\ast\alpha}+O(1)\varepsilon^{\min\{\frac{1}{6},\frac{d-1}{12 m}\}},
\end{align}
where
\begin{align}\label{LTFA001}
M_{d}^{\ast\alpha}=&\int_{\Omega^{\ast}\setminus\Omega^{\ast}_{R}}(\mathbb{C}^{0}e(\bar{u}_{\alpha}^{\ast}),e(\bar{u}_{\alpha}^{\ast}))+\int_{\Omega_{R}^{\ast}}(\mathbb{C}^{0}e(u_{\alpha}^{\ast}-\bar{u}_{\alpha}^{\ast}),e(u_{\alpha}^{\ast}-\bar{u}_{\alpha}^{\ast}))\notag\\
&+\int_{\Omega_{R}^{\ast}}\big[2(\mathbb{C}^{0}e(u_{\alpha}^{\ast}-\bar{u}_{\alpha}^{\ast}),e(\bar{u}_{\alpha}^{\ast}))+2(\mathbb{C}^{0}e(\psi_{\alpha}\bar{v}^{\ast}),e(\mathcal{F}_{\alpha}^{\ast}))+(\mathbb{C}^{0}e(\mathcal{F}_{\alpha}^{\ast}),e(\mathcal{F}_{\alpha}^{\ast}))\big]\notag\\
&+
\begin{cases}
\int_{\Omega_{R}^{\ast}}(\lambda+\mu)(\partial_{x_{\alpha}}\bar{v}^{\ast})^{2}+\mu\sum\limits^{d-1}_{i=1}(\partial_{x_{i}}\bar{v}^{\ast})^{2},&\alpha=1,...,d-1,\\
\int_{\Omega_{R}^{\ast}}\mu\sum\limits^{d-1}_{i=1}(\partial_{x_{i}}\bar{v}^{\ast})^{2},&\alpha=d.
\end{cases}
\end{align}
Then from \eqref{con03365} and \eqref{KKAA123}--\eqref{GNAT001}, we get
\begin{align}\label{ZPH001}
a_{\alpha\alpha}=&\mathcal{L}_{d}^{\alpha}\left(\int_{\varepsilon^{\bar{\gamma}}<|x'|<R}\frac{dx'}{h_{1}(x')-h(x')}+\int_{|x'|<\varepsilon^{\bar{\gamma}}}\frac{dx'}{\varepsilon+h_{1}(x')-h(x')}\right)\nonumber\\
&+M_{d}^{\ast\alpha}+O(1)\varepsilon^{\min\{\frac{1}{6},\frac{d-1}{12m}\}}.
\end{align}
We now divide into two cases to calculate $a_{\alpha\alpha}$ as follows:

$(\rm{i})$ if $m\geq d-1$, then
\begin{align*}
&\int_{\varepsilon^{\bar{\gamma}}<|x'|<R}\frac{1}{h_{1}-h}+\int_{|x'|<\varepsilon^{\bar{\gamma}}}\frac{1}{\varepsilon+h_{1}-h}\notag\\
=&\int_{|x'|<R}\frac{1}{\varepsilon+h_{1}-h}+\int_{\varepsilon^{\bar{\gamma}}<|x'|<R}\frac{\varepsilon}{(h_{1}-h)(\varepsilon+h_{1}-h)}\notag\\
=&\int_{|x'|<R}\frac{1}{\varepsilon+h_{1}-h}+O(1)
\begin{cases}
\varepsilon,&m<\frac{d-1}{2},\\
\varepsilon|\ln\varepsilon|,&m=\frac{d-1}{2},\\
\varepsilon^{\frac{10m+d-1}{12m}},&m>\frac{d-1}{2},
\end{cases}
\end{align*}
which implies that
\begin{align}\label{LNQ001}
a_{\alpha\alpha}\lesssim&\mathcal{L}_{d}^{\alpha}\int_{|x'|<R}\frac{1}{\varepsilon+h_{1}-h}\lesssim\mathcal{L}_{d}^{\alpha}\int_{|x'|<R}\frac{1}{\varepsilon+\kappa_{1}|x'|^{m}}\notag\\
\lesssim&\mathcal{L}_{d}^{\alpha}\int_{0}^{R}\frac{s^{d-2}}{\varepsilon+\kappa_{1}s^{m}}\lesssim\frac{\mathcal{L}_{d}^{\alpha}}{\kappa_{1}^{\frac{d-1}{m}}}\rho_{0}(d,m;\varepsilon),
\end{align}
and
\begin{align}\label{LAU001}
a_{\alpha\alpha}\gtrsim&\mathcal{L}_{d}^{\alpha}\int_{|x'|<R}\frac{1}{\varepsilon+\kappa_{2}|x'|^{m}}\gtrsim\frac{\mathcal{L}_{d}^{\alpha}}{\kappa_{2}^{\frac{d-1}{m}}}\rho_{0}(d,m;\varepsilon);
\end{align}

$(\rm{ii})$ if $m<d-1$, then
\begin{align}\label{LNQ002}
a_{\alpha\alpha}=&\mathcal{L}_{d}^{\alpha}\left(\int_{|x'|<R}\frac{dx'}{h_{1}-h}-\int_{|x'|<\varepsilon^{\bar{\gamma}}}\frac{\varepsilon\,dx'}{(h_{1}-h)(\varepsilon+h_{1}-h)}\right)\notag\\
&+M_{d}^{\ast\alpha}+O(1)\varepsilon^{\min\{\frac{1}{6},\frac{d-1}{12m}\}}\notag\\
=&\mathcal{L}_{d}^{\alpha}\int_{\Omega_{R}^{\ast}}|\partial_{x_{d}}\bar{u}_{\alpha}^{\ast}|^{2}+M_{d}^{\ast\alpha}+O(1)\varepsilon^{\min\{\frac{1}{6},\frac{d-1-m}{12m}\}}\notag\\
=&a_{\alpha\alpha}^{\ast}+O(1)\varepsilon^{\min\{\frac{1}{6},\frac{d-1-m}{12m}\}}.
\end{align}
Consequently, substituting \eqref{LNQ001}--\eqref{LNQ002} into \eqref{ZPH001}, we complete the proof of \eqref{LMC1}.

{\bf Step 2.} Proof of \eqref{LMC}. According to the definition of $\psi_{\alpha}$, we see that for $d+1\leq\alpha\leq\frac{d(d+1)}{2}$, there exist two indices $1\leq i_{\alpha}<j_{\alpha}\leq d$ such that
$\psi_{\alpha}=(0,...,0,x_{j_{\alpha}},0,...,0,-x_{i_{\alpha}},0,...,0)$. Especially when $d+1\leq\alpha\leq2d-1$, we have $i_{\alpha}=\alpha-d,\,j_{\alpha}=d$ and thus $\psi_{\alpha}=(0,...,0,x_{d},0,...,0,-x_{\alpha-d})$. Similar to \eqref{a1111}, for $\alpha=d+1,...,\frac{d(d+1)}{2}$, we split $a_{\alpha\alpha}$ as follows:
\begin{align*}
\mathrm{I}=&\int_{\Omega_{\varepsilon^{\bar{\gamma}}}}(\mathbb{C}^{0}e(u_{\alpha}),e(u_{\alpha})),\\
\mathrm{II}=&\int_{\Omega\setminus\Omega_{R}}(\mathbb{C}^{0}e(u_{\alpha}),e(u_{\alpha})),\\
\mathrm{III}=&\int_{\Omega_{R}\setminus\Omega_{\varepsilon^{\bar{\gamma}}}}(\mathbb{C}^{0}e(u_{\alpha}),e(u_{\alpha})),
\end{align*}
where $\bar{\gamma}=\frac{1}{12m}$. A direct computation shows that for $\alpha=d+1,...,\frac{d(d+1)}{2}$,
\begin{align*}
(\mathbb{C}^{0}e(\bar{u}_{\alpha}),e(\bar{u}_{\alpha}))=&\mu(x_{i_{\alpha}}^{2}+x_{j_{\alpha}}^{2})\sum^{d}_{k=1}(\partial_{x_{k}}\bar{v})^{2}+(\lambda+\mu)(x_{j_{\alpha}}\partial_{x_{i_{\alpha}}}\bar{v}-x_{i_{\alpha}}\partial_{x_{j_{\alpha}}}\bar{v})^{2}\notag\\
&+2(\mathbb{C}^{0}e(\psi_{\alpha}\bar{v}),e(\mathcal{F}_{\alpha}))+(\mathbb{C}^{0}e(\mathcal{F}_{\alpha}),e(\mathcal{F}_{\alpha})),
\end{align*}
where the correction term $\mathcal{F}_{\alpha}$ is defined by \eqref{OPAKLN01}. This, together with Corollary \ref{thm86}, reads that
\begin{align}\label{coJKL5}
\mathrm{I}
=&\frac{\mathcal{L}_{d}^{\alpha}}{d-1}\int_{|x'|<\varepsilon^{\bar{\gamma}}}\frac{|x'|^{2}}{\varepsilon+h_{1}(x')-h(x')}\,dx'+O(1)\varepsilon^{\frac{d-1}{12m}},
\end{align}
where $\mathcal{L}_{d}^{\alpha}$ is defined in (\ref{AZ})--(\ref{AZ110}).

Observe that by applying \eqref{LKT6.007}--\eqref{LKT6.009} with $\gamma=\frac{1}{2(m-1)}$, we get that for $\alpha=d+1,...,\frac{d(d+1)}{2}$,
\begin{align}\label{QMAJ001}
|u_{\alpha}-u_{\alpha}^{\ast}|\leq C\varepsilon^{\frac{1}{2}},\quad\;\,\mathrm{in}\;\,D\setminus\big(\overline{D_{1}\cup D_{1}^{\ast}\cup\mathcal{C}_{\varepsilon^{\frac{1}{2(m-1)}}}}\big).
\end{align}
Proceeding as before, making use of \eqref{QMAJ001}, the rescale argument, the interpolation inequality and the standard elliptic estimates, we deduce that for $\alpha=d+1,...,\frac{d(d+1)}{2}$,
\begin{align}\label{KKTN}
|\nabla(u_{\alpha}-u_{\alpha}^{\ast})|\leq C\varepsilon^{\frac{1}{6}},\quad\;\,\mathrm{in}\;\,D\setminus\big(\overline{D_{1}\cup D_{1}^{\ast}\cup\mathcal{C}_{\varepsilon^{\bar{\gamma}}}}\big).
\end{align}
Following the same argument as in \eqref{KKAA123}, we deduce from \eqref{KKTN} that for $\alpha=d+1,...,\frac{d(d+1)}{2}$,
\begin{align}\label{KHT01}
\mathrm{II}=&\int_{\Omega^{\ast}\setminus\Omega^{\ast}_{R}}(\mathbb{C}^{0}e(u^{\ast}_{\alpha}),e(u_{\alpha}^{\ast}))+O(1)\varepsilon^{\frac{1}{6}}.
\end{align}

With regard to the last term $\mathrm{III}$, similarly as before, we split it as follows:
\begin{align*}
\mathrm{III}=&\int_{(\Omega_{R}\setminus\Omega_{\varepsilon^{\bar{\gamma}}})\setminus(\Omega^{\ast}_{R}\setminus\Omega^{\ast}_{\varepsilon^{\bar{\gamma}}})}(\mathbb{C}^{0}e(u_{\alpha}),e(u_{\alpha}))+\int_{\Omega_{R}^{\ast}\setminus\Omega^{\ast}_{\varepsilon^{\bar{\gamma}}}}(\mathbb{C}^{0}e(\bar{u}_{\alpha}^{\ast}),e(\bar{u}_{\alpha}^{\ast}))\notag\\
&+2\int_{\Omega^{\ast}_{R}\setminus\Omega^{\ast}_{\varepsilon^{\bar{\gamma}}}}\left[(\mathbb{C}^{0}e(u_{\alpha}-u_{\alpha}^{\ast}),e(u_{\alpha}^{\ast}))+(\mathbb{C}^{0}e(u_{\alpha}^{\ast}-\bar{u}_{\alpha}^{\ast}),e(\bar{u}_{\alpha}^{\ast}))\right]\notag\\
&+\int_{\Omega^{\ast}_{R}\setminus\Omega^{\ast}_{\varepsilon^{\bar{\gamma}}}}\left[(\mathbb{C}^{0}e(u_{\alpha}-u_{\alpha}^{\ast}),e(u_{\alpha}-u_{\alpha}^{\ast}))+(\mathbb{C}^{0}e(u_{\alpha}^{\ast}-\bar{u}_{\alpha}^{\ast}),e(u_{\alpha}^{\ast}-\bar{u}_{\alpha}^{\ast}))\right].
\end{align*}
Analogously as above, in light of the fact that the thickness of $(\Omega_{R}\setminus\Omega_{\varepsilon^{\bar{\gamma}}})\setminus(\Omega^{\ast}_{R}\setminus\Omega^{\ast}_{\varepsilon^{\bar{\gamma}}})$ is $\varepsilon$, we deduce from \eqref{LKT6.005} and \eqref{KKTN} that
\begin{align}\label{QPQIJKL01}
\mathrm{III}=&\,\frac{\mathcal{L}_{d}^{\alpha}}{d-1}\int_{\varepsilon^{\bar{\gamma}}<|x'|<R}\frac{|x'|^{2}}{h_{1}(x')-h(x')}dx'-\int_{\Omega^{\ast}\setminus\Omega^{\ast}_{R}}(\mathbb{C}^{0}e(u_{\alpha}^{\ast}),e(u_{\alpha}^{\ast}))\notag\\
&+M_{d}^{\ast\alpha}+O(1)\varepsilon^{\min\{\frac{1}{6},\frac{d-1}{12 m}\}},
\end{align}
where
\begin{align*}
M_{d}^{\ast\alpha}=&\int_{\Omega^{\ast}\setminus\Omega^{\ast}_{R}}(\mathbb{C}^{0}e(\bar{u}_{\alpha}^{\ast}),e(\bar{u}_{\alpha}^{\ast}))+\int_{\Omega_{R}^{\ast}}(\mathbb{C}^{0}e(u_{\alpha}^{\ast}-\bar{u}_{\alpha}^{\ast}),e(u_{\alpha}^{\ast}-\bar{u}_{\alpha}^{\ast}))\notag\\
&+\int_{\Omega_{R}^{\ast}}\big[2(\mathbb{C}^{0}e(u_{\alpha}^{\ast}-\bar{u}_{\alpha}^{\ast}),e(\bar{u}_{\alpha}^{\ast}))+2(\mathbb{C}^{0}e(\psi_{\alpha}\bar{v}^{\ast}),e(\mathcal{F}_{\alpha}^{\ast}))+(\mathbb{C}^{0}e(\mathcal{F}_{\alpha}^{\ast}),e(\mathcal{F}_{\alpha}^{\ast}))\big]\\
&+
\begin{cases}
\int_{\Omega_{R}^{\ast}}\big[\mu(x_{\alpha-d}^{2}+x_{d}^{2})\sum\limits^{d-1}_{k=1}(\partial_{x_{k}}\bar{v}^{\ast})^{2}+\mu(x_{d}\partial_{x_{d}}\bar{v}^{\ast})^{2}+(\lambda+\mu)(x_{d}\partial_{x_{\alpha-d}}\bar{v}^{\ast})^{2}\notag\\
\quad\quad\;-2(\lambda+\mu)x_{\alpha-d}x_{d}\partial_{x_{\alpha-d}}\bar{v}^{\ast}\partial_{x_{d}}\bar{v}^{\ast}\big],\;\quad\alpha=d+1,...,2d-1,\\
\int_{\Omega_{R}^{\ast}}\big[\mu(x_{i_{\alpha}}^{2}+x_{j_{\alpha}}^{2})\sum\limits^{d-1}_{k=1}(\partial_{x_{k}}\bar{v}^{\ast})^{2}\notag\\
\quad\quad\;+(\lambda+\mu)(x_{j_{\alpha}}\partial_{x_{i_{\alpha}}}\bar{v}^{\ast}-x_{i_{\alpha}}\partial_{x_{j_{\alpha}}}\bar{v}^{\ast})^{2}\big],\;\quad\alpha=2d,...,\frac{d(d+1)}{2},\,d\geq3.
\end{cases}
\end{align*}
Then from \eqref{coJKL5} and \eqref{KHT01}--\eqref{QPQIJKL01}, we get that

$(a)$ if $m\geq d+1$, then
\begin{align}\label{QZH001}
a_{\alpha\alpha}=&\frac{\mathcal{L}_{d}^{\alpha}}{d-1}\left(\int_{\varepsilon^{\bar{\gamma}}<|x'|<R}\frac{|x'|^{2}}{h_{1}(x')-h(x')}+\int_{|x'|<\varepsilon^{\bar{\gamma}}}\frac{|x'|^{2}}{\varepsilon+h_{1}(x')-h(x')}\right)\nonumber\\
&+M_{d}^{\ast\alpha}+O(1)\varepsilon^{\min\{\frac{1}{6},\frac{d-1}{12 m}\}},\quad\alpha=d+1,...,\frac{d(d+1)}{2}.
\end{align}
Note that
\begin{align}\label{AGNH001}
&\int_{\varepsilon^{\bar{\gamma}}<|x'|<R}\frac{|x'|^{2}}{h_{1}-h}+\int_{|x'|<\varepsilon^{\bar{\gamma}}}\frac{|x'|^{2}}{\varepsilon+h_{1}-h}\notag\\
=&\int_{|x'|<R}\frac{|x'|^{2}}{\varepsilon+h_{1}-h}+\int_{\varepsilon^{\bar{\gamma}}<|x'|<R}\frac{\varepsilon|x'|^{2}}{(h_{1}-h)(\varepsilon+h_{1}-h)}\notag\\
=&\int_{|x'|<R}\frac{|x'|^{2}}{\varepsilon+h_{1}-h}+O(1)
\begin{cases}
\varepsilon,&m<\frac{d+1}{2},\\
\varepsilon|\ln\varepsilon|,&m=\frac{d+1}{2},\\
\varepsilon^{\frac{10m+d+1}{12m}},&m>\frac{d+1}{2}.
\end{cases}
\end{align}
Then substituting \eqref{AGNH001} into \eqref{QZH001}, we get
\begin{align*}
a_{\alpha\alpha}\lesssim&\frac{\mathcal{L}_{d}^{\alpha}}{d-1}\int_{|x'|<R}\frac{|x'|^{2}}{\varepsilon+h_{1}(x')-h(x')}\lesssim\frac{\mathcal{L}_{d}^{\alpha}}{d-1}\int_{|x'|<R}\frac{|x'|^{2}}{\varepsilon+\kappa_{1}|x'|^{m}}\notag\\
\lesssim&\mathcal{L}_{d}^{\alpha}\int_{0}^{R}\frac{s^{d}}{\varepsilon+\kappa_{1}s^{m}}\lesssim\frac{\mathcal{L}_{d}^{\alpha}}{\kappa_{1}^{\frac{d+1}{m}}}\rho_{2}(d,m;\varepsilon),
\end{align*}
and
\begin{align*}
a_{\alpha\alpha}\gtrsim&\frac{\mathcal{L}_{d}^{\alpha}}{d-1}\int_{|x'|<R}\frac{|x'|^{2}}{\varepsilon+\kappa_{2}|x'|^{m}}\gtrsim\frac{\mathcal{L}_{d}^{\alpha}}{\kappa_{2}^{\frac{d+1}{m}}}\rho_{2}(d,m;\varepsilon);
\end{align*}

$(b)$ if $m<d+1$, for $\alpha=d+1,...,2d-1$, we obtain
\begin{align*}
a_{\alpha\alpha}=&\mathcal{L}_{d}^{\alpha}\left(\int_{\varepsilon^{\bar{\gamma}}<|x'|<R}\frac{x_{\alpha-d}^{2}}{h_{1}-h}+\int_{|x'|<\varepsilon^{\bar{\gamma}}}\frac{x_{\alpha-d}^{2}}{\varepsilon+h_{1}-h}\right)+M_{d}^{\ast\alpha}+O(1)\varepsilon^{\min\{\frac{1}{6},\frac{d-1}{12 m}\}}\notag\\
=&\mathcal{L}_{d}^{\alpha}\left(\int_{|x'|<R}\frac{x_{\alpha-d}^{2}}{h_{1}-h}-\int_{|x'|<\varepsilon^{\bar{\gamma}}}\frac{\varepsilon x_{\alpha-d}^{2}}{(h_{1}-h)(\varepsilon+h_{1}-h)}\right)\notag\\
&+M_{d}^{\ast\alpha}+O(1)\varepsilon^{\min\{\frac{1}{6},\frac{d-1}{12 m}\}}\notag\\
=&\mathcal{L}_{d}^{\alpha}\int_{\Omega_{R}^{\ast}}|x_{\alpha-d}\partial_{x_{d}}\bar{v}^{\ast}|^{2}+M_{d}^{\ast\alpha}+O(1)\varepsilon^{\min\{\frac{1}{6},\frac{d+1-m}{12 m}\}}\notag\\
=&a_{\alpha\alpha}^{\ast}+O(1)\varepsilon^{\min\{\frac{1}{6},\frac{d+1-m}{12 m}\}},
\end{align*}
while, for $\alpha=2d,...,\frac{d(d+1)}{2},\,d\geq3$,
\begin{align*}
a_{\alpha\alpha}=&\frac{\mathcal{L}_{d}^{\alpha}}{2}\left(\int_{\varepsilon^{\bar{\gamma}}<|x'|<R}\frac{x_{i_{\alpha}}^{2}+x_{j_{\alpha}}^{2}}{h_{1}-h}+\int_{|x'|<\varepsilon^{\bar{\gamma}}}\frac{x_{i_{\alpha}}^{2}+x_{j_{\alpha}}^{2}}{\varepsilon+h_{1}-h}\right)\notag\\
&+M_{d}^{\ast\alpha}+O(1)\varepsilon^{\min\{\frac{1}{6},\frac{d-1}{12m}\}}\\
=&\frac{\mathcal{L}_{d}^{\alpha}}{2}\int_{|x'|<R}\frac{x_{i_{\alpha}}^{2}+x_{j_{\alpha}}^{2}}{h_{1}-h}+M_{d}^{\ast\alpha}+O(1)\varepsilon^{\min\{\frac{1}{6},\frac{d+1-m}{12m}\}}\\
=&\frac{\mathcal{L}_{d}^{\alpha}}{2}\int_{\Omega_{R}^{\ast}}(x_{i_{\alpha}}^{2}+x_{j_{\alpha}}^{2})|\partial_{x_{d}}\bar{v}^{\ast}|^{2}+M_{d}^{\ast\alpha}+O(1)\varepsilon^{\min\{\frac{1}{6},\frac{d+1-m}{12m}\}}\\
=&a_{\alpha\alpha}^{\ast}+O(1)\varepsilon^{\min\{\frac{1}{6},\frac{d+1-m}{12 m}\}}.
\end{align*}

Hence, \eqref{LMC} is established by combining the results above.

{\bf Step 3.} Proofs of \eqref{LVZQ001}--\eqref{LVZQ0011}. Set $\bar{\gamma}=\frac{1}{12m}$ again. By symmetry, we only need to calculate  $a_{\alpha\beta}$ under the case of $\alpha<\beta$. Analogously as above, for $\alpha,\beta=1,2,...,\frac{d(d+1)}{2}$, $\alpha\neq\beta,$ the element $a_{\alpha\beta}$ can be decomposed into three parts as follows:
\begin{align*}
a_{\alpha\beta}=&\int_{\Omega\setminus\Omega_{R}}(\mathbb{C}^{0}e(u_{\alpha}),e(u_{\beta}))+\int_{\Omega_{R}\setminus\Omega_{\varepsilon^{\bar{\gamma}}}}(\mathbb{C}^{0}e(u_{\alpha}),e(u_{\beta}))+\int_{\Omega_{\varepsilon^{\bar{\gamma}}}}(\mathbb{C}^{0}e(u_{\alpha}),e(u_{\beta}))\nonumber\\
=&:\mathrm{I}+\mathrm{II}+\mathrm{III}.
\end{align*}
First of all, similar to \eqref{KKAA123}, we get
\begin{align}\label{KKAA1233333}
\mathrm{I}=&\int_{D\setminus(D_{1}\cup D_{1}^{\ast}\cup\Omega_{R})}(\mathbb{C}^{0}e(u_{\alpha}),e(u_{\beta}))+O(1)\varepsilon\notag\\
=&\int_{D\setminus(D_{1}\cup D_{1}^{\ast}\cup\Omega_{R})}\big[(\mathbb{C}^{0}e(u_{\alpha}^{\ast}),e(u_{\beta}^{\ast}))+(\mathbb{C}^{0}e(u_{\alpha}-u_{\alpha}^{\ast}),e(u_{\beta}-u_{\beta}^{\ast}))\big]\notag\\
&+\int_{D\setminus(D_{1}\cup D_{1}^{\ast}\cup\Omega_{R})}\big[(\mathbb{C}^{0}e(u_{\alpha}^{\ast}),e(u_{\beta}-u_{\beta}^{\ast}))+(\mathbb{C}^{0}e(u_{\alpha}-u_{\alpha}^{\ast}),e(u_{\beta}^{\ast}))\big]\notag\\
=&\int_{\Omega^{\ast}\setminus\Omega^{\ast}_{R}}(\mathbb{C}^{0}e(u_{\alpha}^{\ast}),e(u_{\beta}^{\ast}))+O(1)\varepsilon^{\frac{1}{6}}.
\end{align}

With regard to the second term $\mathrm{II}$, we further decompose it as follows:
\begin{align*}
\mathrm{II}_{1}=&\int_{(\Omega_{R}\setminus\Omega_{\varepsilon^{\bar{\gamma}}})\setminus(\Omega^{\ast}_{R}\setminus\Omega^{\ast}_{\varepsilon^{\bar{\gamma}}})}(\mathbb{C}^{0}e(u_{\alpha}),e(u_{\beta}))+\int_{\Omega^{\ast}_{R}\setminus\Omega^{\ast}_{\varepsilon^{\bar{\gamma}}}}(\mathbb{C}^{0}e(u_{\alpha}-u_{\alpha}^{\ast}),e(u_{\beta}-u_{\beta}^{\ast}))\notag\\
&+\int_{\Omega^{\ast}_{R}\setminus\Omega^{\ast}_{\varepsilon^{\bar{\gamma}}}}(\mathbb{C}^{0}e(u_{\alpha}-u_{\alpha}^{\ast}),e(u_{\beta}^{\ast}))+\int_{\Omega^{\ast}_{R}\setminus\Omega^{\ast}_{\varepsilon^{\bar{\gamma}}}}(\mathbb{C}^{0}e(u_{\alpha}^{\ast}),e(u_{\beta}-u_{\beta}^{\ast})),\\
\mathrm{II}_{2}=&\int_{\Omega^{\ast}_{R}\setminus\Omega^{\ast}_{\varepsilon^{\bar{\gamma}}}}(\mathbb{C}^{0}e(u_{\alpha}^{\ast}),e(u_{\beta}^{\ast})).
\end{align*}
Based on the fact that the thickness of $(\Omega_{R}\setminus\Omega_{\varepsilon^{\bar{\gamma}}})\setminus(\Omega^{\ast}_{R}\setminus\Omega^{\ast}_{\varepsilon^{\bar{\gamma}}})$ is $\varepsilon$, we deduce from Corollary \ref{thm86}, \eqref{con035} and \eqref{KKTN} that
\begin{align}\label{con036}
\mathrm{II}_{1}=O(1)\varepsilon^{\frac{1}{6}}.
\end{align}
For the term $\mathrm{II}_{2}$, we divide into two cases to discuss as follows:

$(1)$ if $d=2$, $\alpha=1,\beta=2$, then from \eqref{Lw3.005} we get
\begin{align}\label{PAHN}
\mathrm{II}_{2}=&\int_{\Omega_{R}^{\ast}\setminus\Omega^{\ast}_{\varepsilon^{\bar{\gamma}}}}(\mathbb{C}^{0}e(\bar{u}_{1}^{\ast}),e(\bar{u}_{2}^{\ast}))+\int_{\Omega_{R}^{\ast}\setminus\Omega^{\ast}_{\varepsilon^{\bar{\gamma}}}}(\mathbb{C}^{0}e(u_{1}^{\ast}-\bar{u}_{1}^{\ast}),e(u_{2}^{\ast}-\bar{u}_{2}^{\ast}))\notag\\
&+\int_{\Omega_{R}^{\ast}\setminus\Omega^{\ast}_{\varepsilon^{\bar{\gamma}}}}(\mathbb{C}^{0}e(u_{1}^{\ast}-\bar{u}_{1}^{\ast}),e(\bar{u}_{2}^{\ast}))+\int_{\Omega_{R}^{\ast}\setminus\Omega^{\ast}_{\varepsilon^{\bar{\gamma}}}}(\mathbb{C}^{0}e(\bar{u}_{1}^{\ast}),e(u_{2}^{\ast}-\bar{u}_{2}^{\ast}))\notag\\
=&\int_{\Omega_{R}^{\ast}\setminus\Omega^{\ast}_{\varepsilon^{\bar{\gamma}}}}(\lambda+\mu)\partial_{x_{1}}\bar{v}^{\ast}\partial_{x_{2}}\bar{v}^{\ast}+O(1)\notag\\
=&O(1)|\ln\varepsilon|;
\end{align}

$(2)$ if $d\geq3,\,\alpha,\beta=1,2,...,d,\,\alpha<\beta$, if $d\geq2,\alpha=1,2,...,d,\,\beta=d+1,...,\frac{d(d+1)}{2},\alpha<\beta$, or if $d\geq3,\,\alpha,\beta=d+1,...,\frac{d(d+1)}{2},\,\alpha<\beta$, then from \eqref{Lw3.005} and \eqref{LKT6.005} we obtain
\begin{align}\label{QKL}
&\mathrm{II}_{2}-\int_{\Omega^{\ast}_{R}}(\mathbb{C}^{0}e(u_{\alpha}^{\ast}),e(u_{\beta}^{\ast}))=-\int_{\Omega^{\ast}_{\varepsilon^{\bar{\gamma}}}}(\mathbb{C}^{0}e(u_{\alpha}^{\ast}),e(u_{\beta}^{\ast}))\notag\\
=&\int_{\Omega_{\varepsilon^{\bar{\gamma}}}}(\mathbb{C}^{0}e(\bar{u}_{\alpha}^{\ast}),e(\bar{u}_{\beta}^{\ast}))+\int_{\Omega_{\varepsilon^{\bar{\gamma}}}}(\mathbb{C}^{0}e(u_{\alpha}^{\ast}-\bar{u}_{\alpha}^{\ast}),e(u_{\beta}^{\ast}-\bar{u}_{\beta}^{\ast}))\notag\\
&+\int_{\Omega_{\varepsilon^{\bar{\gamma}}}}(\mathbb{C}^{0}e(\bar{u}_{\alpha}^{\ast}),e(u_{\beta}^{\ast}-\bar{u}_{\beta}^{\ast}))+\int_{\Omega_{\varepsilon^{\bar{\gamma}}}}(\mathbb{C}^{0}e(u_{\alpha}^{\ast}-\bar{u}_{\alpha}^{\ast}),e(\bar{u}_{\beta}^{\ast}))\notag\\
=&\int_{\Omega_{\varepsilon^{\bar{\gamma}}}}(\mathbb{C}^{0}e(\psi_{\alpha}\bar{v}^{\ast}),e(\psi_{\beta}\bar{v}^{\ast}))\notag\\
&+O(1)
\begin{cases}
\varepsilon^{(d-1)\bar{\gamma}},&\alpha=1,2,...,d,\,\beta=1,2,...,\frac{d(d+1)}{2},\,\alpha<\beta,\\
\varepsilon^{d\bar{\gamma}},&\alpha,\beta=d+1,...,\frac{d(d+1)}{2},\,\alpha<\beta,
\end{cases}
\end{align}
where we utilized the fact that
\begin{align*}
(\mathbb{C}^{0}e(\bar{u}_{\alpha}^{\ast}),e(\bar{u}_{\beta}^{\ast}))=&(\mathbb{C}^{0}e(\psi_{\alpha}\bar{v}^{\ast}),e(\psi_{\beta}\bar{v}^{\ast}))+(\mathbb{C}^{0}e(\mathcal{F}_{\alpha}^{\ast}),e(\mathcal{F}_{\beta}^{\ast}))\notag\\
&+(\mathbb{C}^{0}e(\psi_{\alpha}\bar{v}^{\ast}),e(\mathcal{F}_{\beta}^{\ast}))+(\mathbb{C}^{0}e(\mathcal{F}_{\alpha}^{\ast}),e(\psi_{\beta}\bar{v}^{\ast})).
\end{align*}
Denote $E_{\alpha\beta}(\bar{v}^{\ast})=(\mathbb{C}^{0}e(\psi_{\alpha}\bar{v}^{\ast}),e(\psi_{\beta}\bar{v}^{\ast}))$. Then it follows from a direct calculation that

$(\rm{i})$ if $\alpha,\beta=1,2,...,d,$ $\alpha<\beta$, then
\begin{align}\label{ZH0000}
E_{\alpha\beta}(\bar{v}^{\ast})=(\lambda+\mu)\partial_{x_{\alpha}}\bar{v}^{\ast}\partial_{x_{\beta}}\bar{v}^{\ast};
\end{align}

$(\rm{ii})$ if $\alpha=1,2,...,d$, $\beta=d+1,...,\frac{d(d+1)}{2}$, then there exist two indices $1\leq i_{\beta}<j_{\beta}\leq d$ such that
$\psi_{\beta}\bar{v}^{\ast}=(0,...,0,x_{j_{\beta}}\bar{v}^{\ast},0,...,0,-x_{i_{\beta}}\bar{v}^{\ast},0,...,0)$. In the case of $i_{\beta}\neq\alpha,\,j_{\beta}\neq\alpha$, we have
\begin{align}\label{ZH000}
E_{\alpha\beta}(\bar{v}^{\ast})=(\lambda+\mu)\partial_{x_{\alpha}}\bar{v}^{\ast}(x_{j_{\beta}}\partial_{i_{\beta}}\bar{v}^{\ast}-x_{i_{\beta}}\partial_{x_{j_{\beta}}}\bar{v}^{\ast}),
\end{align}
and in the case of $i_{\beta}=\alpha,\,j_{\beta}\neq\alpha$, then
\begin{align}\label{ZH001}
E_{\alpha\beta}(\bar{v}^{\ast})=\mu x_{j_{\beta}}\sum^{d}_{k=1}(\partial_{x_{k}}\bar{v}^{\ast})^{2}+(\lambda+\mu)\partial_{x_{\alpha}}\bar{v}^{\ast}(x_{j_{\beta}}\partial_{i_{\beta}}\bar{v}^{\ast}-x_{i_{\beta}}\partial_{x_{j_{\beta}}}\bar{v}^{\ast}),
\end{align}
and in the case of $i_{\beta}\neq\alpha,\,j_{\beta}=\alpha$, then
\begin{align}\label{ZH002}
E_{\alpha\beta}(\bar{v}^{\ast})=&-\mu x_{i_{\beta}}\sum^{d}_{k=1}(\partial_{x_{k}}\bar{v}^{\ast})^{2}+(\lambda+\mu)\partial_{x_{\alpha}}\bar{v}^{\ast}(x_{j_{\beta}}\partial_{i_{\beta}}\bar{v}^{\ast}-x_{i_{\beta}}\partial_{x_{j_{\beta}}}\bar{v}^{\ast});
\end{align}

$(\rm{iii})$ if $\alpha,\beta=d+1,...,\frac{d(d+1)}{2}$, $\alpha<\beta$, then there exist four indices $1\leq i_{\alpha}<j_{\alpha}\leq d$ and $1\leq i_{\beta}<j_{\beta}\leq d$ such that $\psi_{\alpha}\bar{v}^{\ast}=(0,...,0,x_{j_{\alpha}}\bar{v}^{\ast},0,...,0,-x_{i_{\alpha}}\bar{v}^{\ast},0,...,0)$
and
$\psi_{\beta}\bar{v}^{\ast}=(0,...,0,x_{j_{\beta}}\bar{v}^{\ast},0,...,0,-x_{i_{\beta}}\bar{v}^{\ast},0,...,0)$. In view of the fact that $\alpha<\beta$, we get $j_{\beta}\leq j_{\alpha}$. In the case of $i_{\alpha}\neq i_{\beta},\,j_{\alpha}\neq j_{\beta},\,i_{\alpha}\neq j_{\beta}$, we get
\begin{align}\label{ZH003}
E_{\alpha\beta}(\bar{v}^{\ast})=(\lambda+\mu)(x_{j_{\alpha}}\partial_{x_{i_{\alpha}}}\bar{v}^{\ast}-x_{i_{\alpha}}\partial_{x_{j_{\alpha}}}\bar{v}^{\ast})(x_{j_{\beta}}\partial_{x_{i_{\beta}}}\bar{v}^{\ast}-x_{i_{\beta}}\partial_{x_{j_{\beta}}}\bar{v}^{\ast}),
\end{align}
and in the case of $i_{\alpha}=i_{\beta},\,j_{\alpha}\neq j_{\beta}$,
\begin{align}\label{ZH004}
E_{\alpha\beta}(\bar{v}^{\ast})=&\mu x_{j_{\alpha}}x_{j_{\beta}}\sum^{d}_{k=1}(\partial_{x_{k}}\bar{v}^{\ast})^{2}\notag\\
&+(\lambda+\mu)(x_{j_{\alpha}}\partial_{x_{i_{\alpha}}}\bar{v}^{\ast}-x_{i_{\alpha}}\partial_{x_{j_{\alpha}}}\bar{v}^{\ast})(x_{j_{\beta}}\partial_{x_{i_{\beta}}}\bar{v}^{\ast}-x_{i_{\beta}}\partial_{x_{j_{\beta}}}\bar{v}^{\ast}),
\end{align}
and in the case of $i_{\alpha}\neq i_{\beta},\,j_{\alpha}=j_{\beta}$,
\begin{align}\label{ZH005}
E_{\alpha\beta}(\bar{v}^{\ast})=&\mu x_{i_{\alpha}}x_{i_{\beta}}\sum^{d}_{k=1}(\partial_{x_{k}}\bar{v}^{\ast})^{2}\notag\\
&+(\lambda+\mu)(x_{j_{\alpha}}\partial_{x_{i_{\alpha}}}\bar{v}^{\ast}-x_{i_{\alpha}}\partial_{x_{j_{\alpha}}}\bar{v}^{\ast})(x_{j_{\beta}}\partial_{x_{i_{\beta}}}\bar{v}^{\ast}-x_{i_{\beta}}\partial_{x_{j_{\beta}}}\bar{v}^{\ast}),
\end{align}
and in the case of $i_{\beta}<j_{\beta}=i_{\alpha}<j_{\alpha}$,
\begin{align}\label{ZH006}
E_{\alpha\beta}(\bar{v}^{\ast})=&-\mu x_{i_{\beta}}x_{j_{\alpha}}\sum^{d}_{k=1}(\partial_{x_{k}}\bar{v}^{\ast})^{2}\notag\\
&+(\lambda+\mu)(x_{j_{\alpha}}\partial_{x_{i_{\alpha}}}\bar{v}^{\ast}-x_{i_{\alpha}}\partial_{x_{j_{\alpha}}}\bar{v}^{\ast})(x_{j_{\beta}}\partial_{x_{i_{\beta}}}\bar{v}^{\ast}-x_{i_{\beta}}\partial_{x_{j_{\beta}}}\bar{v}^{\ast}).
\end{align}

Hence, utilizing the parity of integrand and the symmetry of integral region and in view of the fact that
\begin{align*}
\left|\int^{h_{1}(x')}_{h(x')}x_{d}\,dx_{d}\right|\leq\frac{1}{2}|(h_{1}+h)(x')(h_{1}-h)(x')|\leq C|x'|^{2m},\quad \mathrm{in}\; B'_{R},
\end{align*}
we deduce from \eqref{QKL}--\eqref{ZH006} that
\begin{align}\label{QPQ01}
&\mathrm{II}_{2}-\int_{\Omega^{\ast}_{R}}(\mathbb{C}^{0}e(u_{\alpha}^{\ast}),e(u_{\beta}^{\ast}))\notag\\
=&O(1)
\begin{cases}
\varepsilon^{\frac{d-2}{12m}},&d\geq3,\,\alpha,\beta=1,2,...,d,\,\alpha<\beta,\\
\varepsilon^{\frac{d-1}{12m}},&d\geq2,\,\alpha=1,2,...,d,\,\beta=d+1,...,\frac{d(d+1)}{2},\\
\varepsilon^{\frac{d}{12m}},&d\geq3,\,\alpha,\beta=d+1,...,\frac{d(d+1)}{2},\,\alpha<\beta.
\end{cases}
\end{align}
Consequently, from \eqref{con036}--\eqref{PAHN} and \eqref{QPQ01} we get that
\begin{align}\label{FATL001}
\mathrm{II}=&O(1)|\ln\varepsilon|,\quad d=2,\alpha=1,\beta=2,
\end{align}
and
\begin{align}\label{FATL002}
&\mathrm{II}-\int_{\Omega^{\ast}_{R}}(\mathbb{C}^{0}e(u_{\alpha}^{\ast}),e(u_{\beta}^{\ast}))\notag\\
=&O(1)
\begin{cases}
\varepsilon^{\min\{\frac{1}{6},\frac{d-1}{12m}\}},&d\geq3,\alpha,\beta=1,2,...,d,\,\alpha<\beta\\
\varepsilon^{\min\{\frac{1}{6},\frac{d-2}{12m}\}},&d\geq2,\,\alpha=1,2,...,d,\,\beta=d+1,...,\frac{d(d+1)}{2},\\
\varepsilon^{\min\{\frac{1}{6},\frac{d}{12m}\}},&d\geq3,\,\alpha,\beta=d+1,...,\frac{d(d+1)}{2},\,\alpha<\beta.
\end{cases}
\end{align}

With regard to the last part $\mathrm{III}$, following the same argument as in \eqref{QKL}, we have
\begin{align}\label{KKTA}
\mathrm{III}=&\int_{\Omega_{\varepsilon^{\bar{\gamma}}}}(\mathbb{C}^{0}e(u_{\alpha}),e(u_{\beta}))\notag\\
=&\int_{\Omega_{\varepsilon^{\bar{\gamma}}}}(\mathbb{C}^{0}e(\psi_{\alpha}\bar{v}),e(\psi_{\beta}\bar{v}))\notag\\
&+O(1)
\begin{cases}
\varepsilon^{(d-1)\bar{\gamma}},&\alpha=1,2,...,d,\,\beta=1,2,...,\frac{d(d+1)}{2},\,\alpha<\beta,\\
\varepsilon^{d\bar{\gamma}},&\alpha,\beta=d+1,...,\frac{d(d+1)}{2},\,\alpha<\beta.
\end{cases}
\end{align}
Proceeding as in \eqref{ZH0000}--\eqref{ZH006} with $\bar{v}$ substituting for $\bar{v}^{\ast}$, we conclude from \eqref{KKTA} that
\begin{align}\label{MAH01}
\mathrm{III}=&O(1)
\begin{cases}
|\ln\varepsilon|,&d=2,\,\alpha=1,\beta=2,\\
\varepsilon^{\frac{d-2}{12m}},&d\geq3,\,\alpha,\beta=1,2,...,d,\,\alpha<\beta,\\
\varepsilon^{\frac{d-1}{12m}},&d\geq2,\,\alpha=1,2,...,d,\,\beta=d+1,...,\frac{d(d+1)}{2},\\
\varepsilon^{\frac{d}{12m}},&d\geq3,\,\alpha,\beta=d+1,...,\frac{d(d+1)}{2},\,\alpha<\beta.
\end{cases}
\end{align}
Then combining \eqref{KKAA1233333}, \eqref{FATL001}--\eqref{FATL002} and \eqref{MAH01}, we obtain that
\begin{align*}
a_{12}=O(1)|\ln\varepsilon|,\quad d=2,
\end{align*}
and
\begin{align*}
&a_{\alpha\beta}-a_{\alpha\beta}^{\ast}\notag\\
=&O(1)
\begin{cases}
\varepsilon^{\min\{\frac{1}{6},\frac{d-2}{12m}\}},&d\geq3,\alpha,\beta=1,2,...,d,\,\alpha<\beta,\\
\varepsilon^{\min\{\frac{1}{6},\frac{d-2}{12m}\}},&d\geq2,\,\alpha=1,2,...,d,\,\beta=d+1,...,\frac{d(d+1)}{2},\\
\varepsilon^{\min\{\frac{1}{6},\frac{d}{12m}\}},&d\geq3,\,\alpha,\beta=d+1,...,\frac{d(d+1)}{2},\,\alpha<\beta.
\end{cases}
\end{align*}

\end{proof}

\subsection{The proofs of Theorems \ref{Lthm066} and \ref{thma002}}
To begin with, we state a Lemma with its proof seen in \cite{BJL2017}, which will be used to prove Theorems \ref{Lthm066} and \ref{thma002} in the following.
\begin{lemma}\label{FBC6}
There exists a positive universal constant $C$, independent of $\varepsilon$, such that
\begin{align}\label{BJ010}
\sum^{\frac{d(d+1)}{2}}_{i,j=1}a_{ij}\xi_{i}\xi_{j}\geq\frac{1}{C},\quad\;\,\forall\;\xi\in\mathbb{R}^{\frac{d(d+1)}{2}},\;|\xi|=1.
\end{align}
\end{lemma}

In this section, $a\simeq b$ denotes $\frac{a}{b}=1+o(1)$, where $b\neq0$ and $\lim\limits_{\varepsilon\rightarrow0}o(1)=0$. This means that $b$ is approximately equal to $a$. Before proving our main results, we first introduce some notations. Denote
\begin{align*}
X=\big(C^{1},C^{2},...,C^{\frac{d(d+1)}{2}}\big)^{T},\quad Y=\big(Q_{1}[\varphi],Q_{2}[\varphi],...,Q_{\frac{d(d+1)}{2}}[\varphi]\big)^{T},
\end{align*}
and
\begin{gather*}
\mathbb{A}=\begin{pmatrix} a_{11}&\cdots&a_{1d} \\\\ \vdots&\ddots&\vdots\\\\a_{d1}&\cdots&a_{dd}\end{pmatrix}  ,\;\,
\mathbb{B}=\begin{pmatrix}a_{1\,d+1}&\cdots&a_{1\,\frac{d(d+1)}{2}} \\\\ \vdots&\ddots&\vdots\\\\ a_{d\,d+1}&\cdots&a_{d\,\frac{d(d+1)}{2}}\end{pmatrix},
\end{gather*}
\begin{gather}
\mathbb{C}=\begin{pmatrix}a_{d+1\,1}&\cdots&a_{d+1\,d} \\\\ \vdots&\ddots&\vdots\\\\ a_{\frac{d(d+1)}{2}\,1}&\cdots&a_{\frac{d(d+1)}{2}\,\frac{d(d+1)}{2}}\end{pmatrix},\notag\\
\mathbb{D}=\begin{pmatrix} a_{d+1\,d+1}&\cdots&a_{d+1\,\frac{d(d+1)}{2}} \\\\ \vdots&\ddots&\vdots\\\\ a_{\frac{d(d+1)}{2}\,d+1}&\cdots&a_{\frac{d(d+1)}{2}\,\frac{d(d+1)}{2}}\end{pmatrix}.\label{ALHNCZ001}
\end{gather}
Let
\begin{align}\label{QLAN}
\mathbb{F}=\begin{pmatrix} \mathbb{A}&\mathbb{B} \\  \mathbb{C}&\mathbb{D}
\end{pmatrix}.
\end{align}
Then \eqref{Le3.078} can be written as
\begin{align}\label{MWNA001}
\mathbb{F}X=Y.
\end{align}

\begin{proof}[The proofs of Theorems \ref{Lthm066} and \ref{thma002}]

{\bf Step 1.} Let $m\geq d+1$. Applying the Cramer's rule to \eqref{MWNA001}, it follows from Lemmas \ref{KM323} and \ref{lemmabc} that
\begin{align*}
C^{\alpha}\simeq&\frac{\prod\limits_{i\neq\alpha}^{\frac{d(d+1)}{2}}a_{ii}Q_{\alpha}[\varphi]}{\prod\limits_{i=1}^{\frac{d(d+1)}{2}}a_{ii}}=\frac{Q_{\alpha}[\varphi]}{a_{\alpha\alpha}},
\end{align*}
which yields that

$(\rm{i})$ if condition ({\bf{E1}}) holds, then
\begin{align}\label{rep001}
|C^{\alpha}|\lesssim
\begin{cases}
\frac{\eta\kappa_{2}^{\frac{d-1}{m}}}{\kappa_{1}^{\frac{d+k-1}{m}}}\frac{\rho_{k}(d,m;\varepsilon)}{\rho_{0}(d,m;\varepsilon)},&\alpha=1,2,...,d,\,m\geq d+k-1,\\
\frac{\kappa_{2}^{\frac{d-1}{m}}|Q_{\alpha}^{\ast}[\varphi]|}{\mathcal{L}_{d}^{\alpha}}\frac{1}{\rho_{0}(d,m;\varepsilon)},&\alpha=1,2,...,d,\,d+1\leq m<d+k-1,\\
\frac{\kappa_{2}^{\frac{d+1}{m}}|Q_{\alpha}^{\ast}[\varphi]|}{\mathcal{L}_{d}^{\alpha}}\frac{1}{\rho_{2}(d,m;\varepsilon)},&\alpha=d+1,...,\frac{d(d+1)}{2},\,m\geq d+1,
\end{cases}
\end{align}
and
\begin{align}\label{rep002}
|C^{\alpha}|\gtrsim
\begin{cases}
\frac{\eta\kappa_{1}^{\frac{d-1}{m}}}{\kappa_{2}^{\frac{d+k-1}{m}}}\frac{\rho_{k}(d,m;\varepsilon)}{\rho_{0}(d,m;\varepsilon)},&\alpha=1,2,...,d,\,m\geq d+k-1,\\
\frac{\kappa_{1}^{\frac{d-1}{m}}|Q_{\alpha}^{\ast}[\varphi]|}{\mathcal{L}_{d}^{\alpha}}\frac{1}{\rho_{0}(d,m;\varepsilon)},&\alpha=1,2,...,d,\,d+1\leq m<d+k-1,\\
\frac{\kappa_{1}^{\frac{d+1}{m}}|Q_{\alpha}^{\ast}[\varphi]|}{\mathcal{L}_{d}^{\alpha}}\frac{1}{\rho_{2}(d,m;\varepsilon)},&\alpha=d+1,...,\frac{d(d+1)}{2},\,m\geq d+1;
\end{cases}
\end{align}

$(\rm{ii})$ if condition ({\bf{E2}}) holds, then
\begin{align}\label{RTP001}
|C^{\alpha}|\lesssim&
\begin{cases}
\frac{\kappa_{2}^{\frac{d-1}{m}}|Q^{\ast}_{\alpha}[\varphi]|}{\mathcal{L}_{d}^{\alpha}}\frac{1}{\rho_{0}(d,m;\varepsilon)},&\alpha=1,2,...,d,\,m\geq d+1,\\
\frac{\eta\kappa_{2}^{\frac{d+1}{m}}}{\kappa_{1}^{\frac{d+k}{m}}}\frac{\rho_{k+1}(d,m;\varepsilon)}{\rho_{2}(d,m;\varepsilon)},&\alpha=d+1,\,m\geq d+k,\\
\frac{\kappa_{2}^{\frac{d+1}{m}}|Q^{\ast}_{d+1}[\varphi]|}{\mathcal{L}_{d}^{d+1}}\frac{1}{\rho_{2}(d,m;\varepsilon)},&\alpha=d+1,\,d+1\leq m<d+k,\\
\frac{\kappa_{2}^{\frac{d+1}{m}}|Q^{\ast}_{\alpha}[\varphi]|}{\mathcal{L}_{d}^{\alpha}}\frac{1}{\rho_{2}(d,m;\varepsilon)},&\alpha=d+2,...,\frac{d(d+1)}{2},\,d\geq3,\,m\geq d+1,
\end{cases}
\end{align}
and
\begin{align}\label{RTP002}
|C^{\alpha}|\gtrsim&
\begin{cases}
\frac{\kappa_{1}^{\frac{d-1}{m}}|Q^{\ast}_{\alpha}[\varphi]|}{\mathcal{L}_{d}^{\alpha}}\frac{1}{\rho_{0}(d,m;\varepsilon)},&\alpha=1,2,...,d,\,m\geq d+1,\\
\frac{\eta\kappa_{1}^{\frac{d+1}{m}}}{\kappa_{2}^{\frac{d+k}{m}}}\frac{\rho_{k+1}(d,m;\varepsilon)}{\rho_{2}(d,m;\varepsilon)},&\alpha=d+1,\,m\geq d+k,\\
\frac{\kappa_{1}^{\frac{d+1}{m}}|Q^{\ast}_{d+1}[\varphi]|}{\mathcal{L}_{d}^{d+1}}\frac{1}{\rho_{2}(d,m;\varepsilon)},&\alpha=d+1,\,d+1\leq m<d+k,\\
\frac{\kappa_{1}^{\frac{d+1}{m}}|Q^{\ast}_{\alpha}[\varphi]|}{\mathcal{L}_{d}^{\alpha}}\frac{1}{\rho_{2}(d,m;\varepsilon)},&\alpha=d+2,...,\frac{d(d+1)}{2},\,d\geq3,\,m\geq d+1;
\end{cases}
\end{align}

$(\rm{iii})$ if condition ({\bf{E3}}) holds, then
\begin{align}\label{RTP003}
|C^{\alpha}|\lesssim&
\begin{cases}
\frac{\kappa_{2}^{\frac{d-1}{m}}|Q^{\ast}_{\alpha}[\varphi]|}{\mathcal{L}_{d}^{\alpha}}\frac{1}{\rho_{0}(d,m;\varepsilon)},&\alpha=1,2,...,d,\,m\geq d+1,\\
\frac{\kappa_{2}^{\frac{d+1}{m}}|Q^{\ast}_{\alpha}[\varphi]|}{\mathcal{L}_{d}^{\alpha}}\frac{1}{\rho_{2}(d,m;\varepsilon)},&\alpha=d+1,...,\frac{d(d+1)}{2},\,m\geq d+1,
\end{cases}
\end{align}
and
\begin{align}\label{RTP005}
|C^{\alpha}|\gtrsim&
\begin{cases}
\frac{\kappa_{1}^{\frac{d-1}{m}}|Q^{\ast}_{\alpha}[\varphi]|}{\mathcal{L}_{d}^{\alpha}}\frac{1}{\rho_{0}(d,m;\varepsilon)},&\alpha=1,2,...,d,\,m\geq d+1,\\
\frac{\kappa_{1}^{\frac{d+1}{m}}|Q^{\ast}_{\alpha}[\varphi]|}{\mathcal{L}_{d}^{\alpha}}\frac{1}{\rho_{2}(d,m;\varepsilon)},&\alpha=d+1,...,\frac{d(d+1)}{2},\,m\geq d+1.
\end{cases}
\end{align}

{\bf Step 2.} Let $d-1\leq m<d+1$. Denote
\begin{gather}\label{QKT001}
\mathbb{F}_{1}^{\alpha}[\varphi]:=
\begin{pmatrix} Q_{\alpha}[\varphi]&a_{\alpha\,d+1}&\cdots&a_{\alpha\,\frac{d(d+1)}{2}} \\ Q_{d+1}[\varphi]&a_{d+1\,d+1}&\cdots&a_{d+1\,\frac{d(d+1)}{2}} \\ \vdots&\vdots&\ddots&\vdots\\Q_{\frac{d(d+1)}{2}}[\varphi]&a_{\frac{d(d+1)}{2}\,d+1}&\cdots&a_{\frac{d(d+1)}{2}\,\frac{d(d+1)}{2}}
\end{pmatrix},\quad\alpha=1,2,...,d,
\end{gather}
and, for $\alpha=d+1,...,\frac{d(d+1)}{2}$, by substituting the column vector $\big(Q_{d+1}[\varphi],...,Q_{\frac{d(d+1)}{2}}[\varphi]\big)^{T}$ for the elements of $\alpha$-th column in the matrix $\mathbb{D}$ defined by \eqref{ALHNCZ001}, we get the new matrix $\mathbb{F}_{2}^{\alpha}[\varphi]$ as follows:
\begin{gather}\label{QKT002}
\mathbb{F}_{2}^{\alpha}[\varphi]:=
\begin{pmatrix}
a_{d+1\,d+1}&\cdots&Q_{d+1}[\varphi]&\cdots&a_{d+1\,\frac{d(d+1)}{2}} \\\\ \vdots&\ddots&\vdots&\ddots&\vdots\\\\a_{\frac{d(d+1)}{2}\,d+1}&\cdots&Q_{\frac{d(d+1)}{2}}[\varphi]&\cdots&a_{\frac{d(d+1)}{2}\frac{d(d+1)}{2}}
\end{pmatrix}.
\end{gather}
Then it follows from the Cramer's rule and Lemmas \ref{KM323}--\ref{lemmabc} that if condition ({\bf{E1}}), ({\bf{E2}}) or ({\bf{E3}}) holds for $d-1\leq m<d+1$,
\begin{align}\label{TGV001}
C^{\alpha}\simeq&
\begin{cases}
\frac{\prod\limits_{i\neq\alpha}^{d}a_{ii}\det\mathbb{F}_{1}^{\alpha}[\varphi]}{\prod\limits_{i=1}^{d}a_{ii}\det\mathbb{D}}=\frac{\det\mathbb{F}_{1}^{\alpha}[\varphi]}{\det\mathbb{D}}\frac{1}{a_{\alpha\alpha}},&\alpha=1,2,...,d,\\
\frac{\det\mathbb{F}_{2}^{\alpha}[\varphi]}{\det\mathbb{D}},&\alpha=d+1,...,\frac{d(d+1)}{2}.
\end{cases}
\end{align}
Using Lemmas \ref{KM323}--\ref{lemmabc} again, we derive
\begin{align}
\det\mathbb{F}_{1}^{\alpha}[\varphi]=&\det\mathbb{F}_{1}^{\ast\alpha}[\varphi]+O(1)\varepsilon^{\min\{\frac{d+1-m}{12m},\frac{d+k-1-m}{2(d+k-1)}\}},\quad\alpha=1,2,...,d,\label{PAK001}\\
\det\mathbb{F}_{2}^{\alpha}[\varphi]=&\det\mathbb{F}_{2}^{\ast\alpha}[\varphi]+O(1)\varepsilon^{\frac{d+1-m}{12m}},\quad\alpha=d+1,...,\frac{d(d+1)}{2},\label{PAK002}\\
\det\mathbb{D}=&\det\mathbb{D}^{\ast}+O(1)\varepsilon^{\frac{d+1-m}{12m}}.\label{PAK003}
\end{align}
Therefore, substituting \eqref{PAK001}--\eqref{PAK003} into \eqref{TGV001}, we deduce that for $\alpha=1,2,...,d,$
\begin{align}\label{re001}
\frac{\kappa_{1}^{\frac{d-1}{m}}|\det\mathbb{F}_{1}^{\ast\alpha}[\varphi]|}{\mathcal{L}_{d}^{\alpha}\det\mathbb{D}^{\ast}}\frac{1}{\rho_{0}(d,m;\varepsilon)}\lesssim|C^{\alpha}|\lesssim\frac{\kappa_{2}^{\frac{d-1}{m}}|\det\mathbb{F}_{1}^{\ast\alpha}[\varphi]|}{\mathcal{L}_{d}^{\alpha}\det\mathbb{D}^{\ast}}\frac{1}{\rho_{0}(d,m;\varepsilon)},
\end{align}
and
\begin{align}\label{re002}
C^{\alpha}\simeq\frac{\det\mathbb{F}_{2}^{\ast\alpha}[\varphi]}{\det\mathbb{D}^{\ast}},\quad\alpha=d+1,...,\frac{d(d+1)}{2}.
\end{align}

We now claim that $\det\mathbb{D}^{\ast}>0$. Taking $\xi=(\xi_{1},\xi_{2},...,\xi_{\frac{d(d+1)}{2}})\in\mathbb{R}^{\frac{d(d+1)}{2}}$ with $\xi_{1}=\cdots=\xi_{d}=0$ and $|\xi|=1$ in \eqref{BJ010}, we deduce from Lemma \ref{lemmabc} that
\begin{align*}
\sum^{\frac{d(d+1)}{2}}_{i,j=1}a_{ij}\xi_{i}\xi_{j}=\sum^{\frac{d(d+1)}{2}}_{i,j=d+1}a_{ij}^{\ast}\xi_{i}\xi_{j}+O(\varepsilon^{\min\{\frac{1}{6},\frac{d+1-m}{12m}\}})\geq\frac{1}{C},
\end{align*}
where the constant $C$ is independent of $\varepsilon$. Therefore,
$$\sum^{\frac{d(d+1)}{2}}_{i,j=d+1}a_{ij}^{\ast}\xi_{i}\xi_{j}\geq\frac{1}{C}>0,$$
which implies that the matrix $\mathbb{D}^{\ast}$ is positive definite and thus $\det\mathbb{D}^{\ast}>0$.

{\bf Step 3.} Let $m<d-1$. For $\alpha=1,2,...,\frac{d(d+1)}{2}$, by replacing the elements of $\alpha$-th column of the matrix $\mathbb{F}$ defined in \eqref{QLAN} by column vector $(Q_{1}[\varphi],Q_{2}[\varphi],...,Q_{\frac{d(d+1)}{2}}[\varphi])^{T}$, we obtain the new matrix $\mathbb{F}_{3}^{\alpha}[\varphi]$ as follows:
\begin{gather}\label{PGNC}
\mathbb{F}_{3}^{\alpha}[\varphi]=:
\begin{pmatrix}
a_{11}&\cdots&Q_{1}[\varphi]&\cdots&a_{1\,\frac{d(d+1)}{2}} \\\\ \vdots&\ddots&\vdots&\ddots&\vdots\\\\a_{\frac{d(d+1)}{2}\,1}&\cdots&Q_{\frac{d(d+1)}{2}}[\varphi]&\cdots&a_{\frac{d(d+1)}{2}\frac{d(d+1)}{2}}
\end{pmatrix}.
\end{gather}
Then a consequence of Lemmas \ref{KM323} and \ref{lemmabc} yields that
\begin{align*}
\mathbb{F}_{3}^{\alpha}[\varphi]=\mathbb{F}_{3}^{\ast\alpha}[\varphi]+O(1)\varepsilon^{\min\{\frac{d+k-1-m}{2(d+k-1)},\frac{d-1-m}{12m}\}},\quad\alpha=1,2,...,\frac{d(d+1)}{2},
\end{align*}
and
\begin{align*}
\mathbb{F}=\mathbb{F}^{\ast}+O(1)\varepsilon^{\min\{\frac{1}{6},\frac{d-1-m}{12m}\}}.
\end{align*}
Similarly as before, we see from the Cramer's rule that if condition ({\bf{E1}}), ({\bf{E2}}) or ({\bf{E3}}) holds for $m<d-1$,
\begin{align}\label{re003}
C^{\alpha}=\frac{\det\mathbb{F}_{3}^{\alpha}[\varphi]}{\det\mathbb{F}}\simeq\frac{\det\mathbb{F}_{3}^{\ast\alpha}[\varphi]}{\det\mathbb{F}^{\ast}},\quad\alpha=1,2,...,\frac{d(d+1)}{2}.
\end{align}

We now demonstrate that $\det\mathbb{F}^{\ast}>0$. Proceeding as before, it follows from Lemmas \ref{lemmabc} and \ref{FBC6} that
\begin{align*}
\sum^{\frac{d(d+1)}{2}}_{i,j=1}a_{ij}\xi_{i}\xi_{j}=\sum^{\frac{d(d+1)}{2}}_{i,j=1}a_{ij}^{\ast}\xi_{i}\xi_{j}+O(\varepsilon^{\min\{\frac{1}{6},\frac{d-1-m}{12m}\}})\geq\frac{1}{C},\quad\;\,\forall\;\xi\in\mathbb{R}^{\frac{d(d+1)}{2}},\;|\xi|=1.
\end{align*}
This reads that
\begin{align*}
\sum^{\frac{d(d+1)}{2}}_{i,j=1}a_{ij}^{\ast}\xi_{i}\xi_{j}\geq\frac{1}{C}.
\end{align*}
Thus $\mathbb{F}^{\ast}$ is a positive definite matrix, which means that $\det\mathbb{F}^{\ast}>0.$

{\bf Step 4.} In this step we aim to establish the optimal upper and lower bounds on the blow-up rate of the gradient in the shortest segment $\{x'=0'\}\cap\Omega$.

$(\rm{i})$ Let condition ({\bf{E1}}), ({\bf{E2}}) or ({\bf{E3}}) hold for $m\leq d$. In light of decomposition \eqref{Le2.015}, we deduce from Corollary \ref{thm86} and \eqref{re001}--\eqref{re003} that for $x\in\{x'=0'\}\cap\Omega$,
\begin{align*}
|\nabla u|\leq&\sum^{d}_{\alpha=1}|C^{\alpha}||\nabla u_{\alpha}|+\sum^{\frac{d(d+1)}{2}}_{\alpha=d+1}|C^{\alpha}||\nabla u_{\alpha}|+|\nabla u_{0}|\notag\\
\lesssim&
\begin{cases}
\frac{\max\limits_{1\leq\alpha\leq d}\kappa_{2}^{\frac{d-1}{m}}|\mathcal{L}_{d}^{\alpha}|^{-1}|\det\mathbb{A}_{1}^{\ast\alpha}|}{\det\mathbb{D}^{\ast}}\frac{1}{\varepsilon\rho_{0}(d,m;\varepsilon)},&d-1\leq m\leq d,\\
\frac{\max\limits_{1\leq\alpha\leq d}|\det\mathbb{F}_{3}^{\ast\alpha}[\varphi]|}{\det\mathbb{F}^{\ast}}\frac{1}{\varepsilon},&m<d-1,
\end{cases}
\end{align*}
and
\begin{align*}
|\nabla u|\geq&\left|\sum^{d}_{\alpha=1}C^{\alpha}\nabla u_{\alpha}\right|-\sum^{\frac{d(d+1)}{2}}_{\alpha=d+1}|C^{\alpha}||\nabla u_{\alpha}|-|\nabla u_{0}|\notag\\
\geq&\left|\sum^{d}_{\alpha=1}C^{\alpha}\partial_{x_{d}} u_{\alpha}^{\alpha_{0}}\right|-C\notag\\
\gtrsim&
\begin{cases}
\frac{\kappa_{1}^{\frac{d-1}{m}}|\det\mathbb{A}_{1}^{\ast\alpha_{0}}|}{|\mathcal{L}_{d}^{\alpha_{0}}|\det\mathbb{D}^{\ast}}\frac{1}{\varepsilon\rho_{0}(d,m;\varepsilon)},&d-1\leq m\leq d,\\
\frac{|\det\mathbb{F}_{3}^{\ast\alpha_{0}}[\varphi]|}{\det\mathbb{F}^{\ast}}\frac{1}{\varepsilon},&m<d-1,
\end{cases}
\end{align*}
where we utilized the fact that for $x\in\{x'=0'\}\cap\Omega$,
$$|\partial_{x_{d}}u_{\alpha_{0}}^{\alpha_{0}}|\geq|\partial_{x_{d}}\bar{u}_{\alpha_{0}}^{\alpha_{0}}|-|\partial_{x_{d}}(u_{\alpha_{0}}^{\alpha_{0}}-\bar{u}_{\alpha_{0}}^{\alpha_{0}})|\geq\frac{1}{C\varepsilon},$$
and
$$|\partial_{x_{d}}u_{\alpha}^{\alpha_{0}}|\leq|\partial_{x_{d}}\mathcal{F}_{\alpha}^{\alpha_{0}}|+|\partial_{x_{d}}(u_{\alpha}^{\alpha_{0}}-\bar{u}_{\alpha}^{\alpha_{0}})|\leq C,\quad \alpha=1,2,...,d,\,\alpha\neq\alpha_{0}.$$

$(\rm{ii})$ Let condition ({\bf{E1}}) hold for $m\geq d+k$. Observe that for $x\in\{x'=0'\}\cap\Omega$, we deduce from Corollary \ref{thm86} that $|\nabla u_{0}|\leq C$, and $|\nabla u_{\alpha}|\leq C$, $\alpha=d+1,...,\frac{d(d+1)}{2}$. On one hand, using \eqref{rep001}, we have
\begin{align*}
|\nabla u|\leq&\sum^{d}_{\alpha=1}|C^{\alpha}||\nabla u_{\alpha}|+\sum^{\frac{d(d+1)}{2}}_{\alpha=d+1}|C^{\alpha}||\nabla u_{\alpha}|+|\nabla u_{0}|\notag\\
\lesssim&\frac{\eta\kappa_{2}^{\frac{d-1}{m}}}{\kappa_{1}^{\frac{d+k-1}{m}}}\frac{\rho_{k}(d,m;\varepsilon)}{\varepsilon\rho_{0}(d,m;\varepsilon)}.
\end{align*}

On the other hand, it follows from \eqref{rep002} that
\begin{align*}
|\nabla u|\geq&\left|\sum^{d}_{\alpha=1}C^{\alpha}\nabla u_{\alpha}\right|-\sum^{\frac{d(d+1)}{2}}_{\alpha=d+1}|C^{\alpha}||\nabla u_{\alpha}|-|\nabla u_{0}|\notag\\
\geq&\left|\sum^{d}_{\alpha=1}C^{\alpha}\partial_{x_{d}} u_{\alpha}^{d}\right|-C\gtrsim\frac{\eta\kappa_{1}^{\frac{d-1}{m}}}{\kappa_{2}^{\frac{d+k-1}{m}}}\frac{\rho_{k}(d,m;\varepsilon)}{\rho_{0}(d,m;\varepsilon)},
\end{align*}
where we used the fact that for $x\in\{x'=0'\}\cap\Omega$,
$$|\partial_{x_{d}}u_{d}^{d}|\geq|\partial_{x_{d}}\bar{u}_{d}^{d}|-|\partial_{x_{d}}(u_{d}^{d}-\bar{u}_{d}^{d})|\geq\frac{1}{C\varepsilon},$$
and
$$|\partial_{x_{d}}u_{\alpha}^{d}|\leq|\partial_{x_{d}}\mathcal{F}_{\alpha}^{d}|+|\partial_{x_{d}}(u_{\alpha}^{d}-\bar{u}_{\alpha}^{d})|\leq C,\quad \alpha=1,...,d-1.$$

{\bf Step 5.} This step is devoted to the establishments of the optimal upper and lower bounds on the blow-up rate of the gradient on the cylinder surface $\{|x'|=\sqrt[m]{\varepsilon}\}\cap\Omega$.

$(\rm{i})$ Let condition ({\bf{E1}}), ({\bf{E2}}) or ({\bf{E3}}) hold for $d<m<d+k,\,k>1$. In view of \eqref{rep001}, \eqref{RTP001}, \eqref{RTP003} and \eqref{re001}--\eqref{re002}, it follows from decomposition \eqref{Le2.015} and Corollary \ref{thm86} again that for $x\in\{x'=(\sqrt[m]{\varepsilon},0,...,0)'\}\cap\Omega$,
\begin{align*}
|\nabla u|\leq&\sum^{\frac{d(d+1)}{2}}_{\alpha=d+1}|C^{\alpha}||\nabla u_{\alpha}|+\left(\sum^{d}_{\alpha=1}|C^{\alpha}||\nabla u_{\alpha}|+|\nabla u_{0}|\right)\notag\\
\lesssim&
\begin{cases}
\max\limits_{d+1\leq\alpha\leq\frac{d(d+1)}{2}}|\mathcal{L}_{d}^{\alpha}|^{-1}|Q_{\alpha}^{\ast}[\varphi]|\frac{\kappa_{2}^{\frac{d+1}{m}}}{1+\kappa_{1}}\frac{1}{\varepsilon^{1-1/m}\rho_{2}(d,m;\varepsilon)},&d+1\leq m<d+k,\\
\frac{\max\limits_{d+1\leq\alpha\leq\frac{d(d+1)}{2}}\det\mathbb{F}_{2}^{\ast\alpha}[\varphi]}{(1+\kappa_{1})\det\mathbb{D}^{\ast}}\frac{1}{\varepsilon^{1-1/m}},&d<m<d+1.
\end{cases}
\end{align*}

Due to the fact that for $x\in\{x'=(\sqrt[m]{\varepsilon},0,...,0)'\}\cap\Omega$,
\begin{align*}
|\partial_{x_{d}}u_{d+1}^{d}|\geq|\partial_{x_{d}}\bar{u}_{d+1}^{d}|-|\partial_{x_{d}}(u_{d+1}^{d}-\bar{u}_{d+1}^{d})|\geq\frac{1}{C\varepsilon^{1-1/m}},
\end{align*}
and
\begin{align*}
|\partial_{x_{d}}u_{\alpha}^{d}|\leq|\partial_{x_{d}}\mathcal{F}_{\alpha}^{d}|+|\partial_{x_{d}}(u_{\alpha}^{d}-\bar{u}_{\alpha}^{d})|\leq C,\quad \alpha=d+2,...,\frac{d(d+1)}{2},\,d\geq3,
\end{align*}
we obtain
\begin{align}\label{AMU001}
\left|\sum^{\frac{d(d+1)}{2}}_{\alpha=d+1}C^{\alpha}\nabla u_{\alpha}\right|\geq&\left|\sum^{\frac{d(d+1)}{2}}_{\alpha=d+1}C^{\alpha}\partial_{x_{d}}u_{\alpha}^{d}\right|\notag\\
\geq&\left|C^{d+1}\partial_{x_{d}}u_{d+1}^{d}\right|-
\begin{cases}
0,&d=2,\\
\Big|\sum\limits^{\frac{d(d+1)}{2}}_{\alpha=d+2}C^{\alpha}\partial_{x_{d}}u_{\alpha}^{d}\Big|,&d\geq3
\end{cases}\notag\\
\gtrsim&
\begin{cases}
\frac{\kappa_{2}^{\frac{d+1}{m}}}{1+\kappa_{2}}\frac{|Q_{d+1}^{\ast}[\varphi]|}{|\mathcal{L}_{d}^{d+1}|}\frac{1}{\varepsilon^{1-1/m}\rho_{2}(d,m;\varepsilon)},&d+1\leq m<d+k,\\
\frac{\det\mathbb{F}_{2}^{\ast d+1}[\varphi]}{(1+\kappa_{2})\det\mathbb{D}^{\ast}}\frac{1}{\varepsilon^{1-1/m}},&d<m<d+1,
\end{cases}
\end{align}
while, for $x\in\{x'=(\sqrt[m]{\varepsilon},0,...,0)'\}\cap\Omega$,
\begin{align}\label{AMU002}
|\nabla u_{0}|\leq\frac{C|x'|^{k}}{\varepsilon+\kappa_{1}|x'|^{m}}\lesssim
\begin{cases}
\frac{1}{\varepsilon^{1-k/m}},&m>k,\\
1,&m\leq k,
\end{cases}
\end{align}
and
\begin{align}\label{AMU003}
\sum^{d}_{\alpha=1}|C^{\alpha}||\nabla u_{\alpha}|\lesssim&
\begin{cases}
\frac{\rho_{k}(d,m;\varepsilon)}{\varepsilon\rho_{0}(d,m;\varepsilon)},&\text{if condition}\;({\bf{E1}})\;\text{holds},\\
\frac{1}{\varepsilon\rho_{0}(d,m;\varepsilon)},&\text{if condition ({\bf{E2}}) or ({\bf{E3}}) holds}.
\end{cases}
\end{align}
Therefore, from \eqref{AMU001}--\eqref{AMU003}, we have
\begin{align*}
|\nabla u|\geq&\left|\sum^{\frac{d(d+1)}{2}}_{\alpha=d+1}C^{\alpha}\nabla u_{\alpha}\right|-\sum^{d}_{\alpha=1}|C^{\alpha}||\nabla u_{\alpha}|-|\nabla u_{0}|\notag\\
\gtrsim&
\begin{cases}
\frac{\kappa_{2}^{\frac{d+1}{m}}}{1+\kappa_{2}}\frac{|Q_{d+1}^{\ast}[\varphi]|}{|\mathcal{L}_{d}^{d+1}|}\frac{1}{\varepsilon^{1-1/m}\rho_{2}(d,m;\varepsilon)},&d+1\leq m<d+k,\\
\frac{\det\mathbb{F}_{2}^{\ast d+1}[\varphi]}{(1+\kappa_{2})\det\mathbb{D}^{\ast}}\frac{1}{\varepsilon^{1-1/m}},&d<m<d+1.
\end{cases}
\end{align*}

$(\rm{ii})$ Let condition ({\bf{E2}}) hold for $m>d+k,\,k=1$. Similarly as before, we derive that for $x\in\{x'=(\sqrt[m]{\varepsilon},0,...,0)'\}\cap\Omega$,
\begin{align*}
|\nabla u|\leq&\sum^{\frac{d(d+1)}{2}}_{\alpha=d+1}|C^{\alpha}||\nabla u_{\alpha}|+\left(\sum^{d}_{\alpha=1}|C^{\alpha}||\nabla u_{\alpha}|+|\nabla u_{0}|\right)\notag\\
\lesssim&\frac{\eta\kappa_{2}^{\frac{d+1}{m}}}{(1+\kappa_{1})\kappa_{1}^{\frac{d+k}{m}}}\frac{\rho_{k+1}(d,m;\varepsilon)}{\varepsilon^{1-1/m}\rho_{2}(d,m;\varepsilon)},
\end{align*}
and
\begin{align*}
\left|\sum^{\frac{d(d+1)}{2}}_{\alpha=d+1}C^{\alpha}\nabla u_{\alpha}\right|\geq&\left|C^{d+1}\partial_{x_{d}}u_{d+1}^{d}\right|-
\begin{cases}
0,&d=2,\\
\Big|\sum\limits^{\frac{d(d+1)}{2}}_{\alpha=d+2}C^{\alpha}\partial_{x_{d}}u_{\alpha}^{d}\Big|,&d\geq3
\end{cases}\notag\\
\gtrsim&\frac{\eta\kappa_{1}^{\frac{d+1}{m}}}{(1+\kappa_{2})\kappa_{2}^{\frac{d+k}{m}}}\frac{\rho_{k+1}(d,m;\varepsilon)}{\varepsilon^{1-1/m}\rho_{2}(d,m;\varepsilon)},
\end{align*}
which, in combination with \eqref{AMU002}--\eqref{AMU003}, yields that
\begin{align*}
|\nabla u|\geq&\left|\sum^{\frac{d(d+1)}{2}}_{\alpha=d+1}C^{\alpha}\nabla u_{\alpha}\right|-\sum^{d}_{\alpha=1}|C^{\alpha}||\nabla u_{\alpha}|-|\nabla u_{0}|\notag\\
\gtrsim&\frac{\eta\kappa_{1}^{\frac{d+1}{m}}}{(1+\kappa_{2})\kappa_{2}^{\frac{d+k}{m}}}\frac{\rho_{k+1}(d,m;\varepsilon)}{\rho_{2}(d,m;\varepsilon)}.
\end{align*}

$(\rm{iii})$ Let condition ({\bf{E3}}) hold for $m>d+k,\,k\geq1,\,k\neq2$ or $m=d+k,\,k=1$. From \eqref{ATCG001}, we see that for $x\in\{x'=(\sqrt[m]{\varepsilon},0,...,0)'\}\cap\Omega$,
\begin{align*}
|\nabla u_{0}|\lesssim\frac{|\varphi(x',h(x'))|}{\varepsilon+\kappa_{1}|x'|^{m}}\lesssim\frac{\eta|x_{1}|^{k}}{\varepsilon+\kappa_{1}|x'|^{m}}\lesssim\frac{\eta}{1+\kappa_{1}}\frac{1}{\varepsilon^{1-k/m}},\\
|\nabla u_{0}|\gtrsim\frac{|\varphi^{1}(x',h(x'))|}{\varepsilon+\kappa_{2}|x'|^{m}}=\frac{\eta|x_{1}|^{k}}{\varepsilon+\kappa_{2}|x'|^{m}}\gtrsim\frac{\eta}{1+\kappa_{2}}\frac{1}{\varepsilon^{1-\frac{k}{m}}}.
\end{align*}
This, together with \eqref{RTP003}, reads that
\begin{align*}
|\nabla u|\leq&|\nabla u_{0}|+\sum^{\frac{d(d+1)}{2}}_{\alpha=1}|C^{\alpha}||\nabla u_{\alpha}|\lesssim\frac{\eta}{1+\kappa_{1}}\frac{1}{\varepsilon^{1-k/m}},
\end{align*}
and
\begin{align*}
|\nabla u|\geq&|\nabla u_{0}|-\sum^{\frac{d(d+1)}{2}}_{\alpha=1}|C^{\alpha}||\nabla u_{\alpha}|\gtrsim\frac{\eta}{1+\kappa_{2}}\frac{1}{\varepsilon^{1-k/m}}.
\end{align*}

$(\rm{iv})$ Now we proceed to consider the remaining two cases: condition ({\bf{E2}}) holds in the case of $m>d+k,\,k\geq1,\,k\neq2$ or $m=d+k,\,k=1$; condition ({\bf{E3}}) holds in the case of $m=d+k,\,k>2$. Similarly as above, it follows from \eqref{RTP001}, \eqref{RTP003} and Corollary \ref{thm86} that for $x\in\{x'=(\sqrt[m]{\varepsilon},0,...,0)'\}\cap\Omega$, if condition ({\bf{E2}}) holds for $m>d+k,\,k\geq1,\,k\neq2$ or $m=d+k,\,k=1$,
\begin{align*}
|\nabla u|\leq&(|C^{d+1}||\nabla u_{d+1}|+|\nabla u_{0}|)+\sum^{\frac{d(d+1)}{2}}_{\alpha=1,\alpha\neq d+1}|C^{\alpha}||\nabla u_{\alpha}|\notag\\
\lesssim&
\frac{\eta(\kappa_{2}^{\frac{d+1}{m}}\kappa_{1}^{-\frac{d+k}{m}}+1)}{1+\kappa_{1}}\frac{1}{\varepsilon^{1-k/m}},
\end{align*}
and, if condition ({\bf{E3}}) holds for $m=d+k,\,k>2$,
\begin{align*}
|\nabla u|\leq&\left(\sum^{\frac{d(d+1)}{2}}_{\alpha=d+1}|C^{\alpha}||\nabla u_{\alpha}|+|\nabla u_{0}|\right)+\sum^{d}_{\alpha=1}|C^{\alpha}||\nabla u_{\alpha}|\notag\\
\lesssim&\frac{\max\limits_{d+1\leq\alpha\leq\frac{d(d+1)}{2}}\kappa_{2}^{\frac{d+1}{m}}|\mathcal{L}_{d}^{\alpha}|^{-1}|Q^{\ast}_{\alpha}[\varphi]|+\eta}{1+\kappa_{1}}\frac{1}{\varepsilon^{1-k/m}}.
\end{align*}

On the other hand, utilizing \eqref{RTP002} and \eqref{RTP005}, we deduce that
\begin{align*}
&\left|\sum^{\frac{d(d+1)}{2}}_{\alpha=d+1}C^{\alpha}\nabla u_{\alpha}+\nabla u_{0}\right|\geq\left|\sum^{\frac{d(d+1)}{2}}_{\alpha=d+1}C^{\alpha}\partial_{x_{d}} u_{\alpha}^{d}+\partial_{x_{d}}u_{0}^{d}\right|\notag\\
\geq&\left|C^{d+1}\partial_{x_{d}}u_{d+1}^{d}+\partial_{x_{d}}u_{0}^{d}\right|-
\begin{cases}
0,&d=2,\\
\Big|\sum\limits^{\frac{d(d+1)}{2}}_{\alpha=d+2}C^{\alpha}\partial_{x_{d}}u_{\alpha}^{d}\Big|,&d\geq3
\end{cases}\notag\\
\gtrsim&\left|C^{d+1}\partial_{x_{d}}\bar{u}_{d+1}^{d}+\partial_{x_{d}}\bar{u}_{0}^{d}\right|\notag\\
\gtrsim&\frac{1}{\varepsilon^{1-k/m}}
\begin{cases}
\frac{\eta(\kappa_{1}^{\frac{d+1}{m}}\kappa_{2}^{-\frac{d+k}{m}}+1)}{1+\kappa_{2}},&\text{if condition}\;({\bf{E2}})\;\text{holds for}\;m=d+k,k=1\\
&\text{or}\;m>d+k,k\geq1,k\neq2,\\
\frac{\kappa_{1}^{\frac{d+1}{m}}|Q^{\ast}_{d+1}[\varphi]|}{(1+\kappa_{2})\mathcal{L}_{d}^{d+1}},&\text{if condition}\;({\bf{E3}})\;\text{holds for}\;m=d+k,k>2,
\end{cases}
\end{align*}
which, in combination with \eqref{AMU003}, yields that
\begin{align*}
|\nabla u|\geq&\left|\sum^{\frac{d(d+1)}{2}}_{\alpha=d+1}C^{\alpha}\nabla u_{\alpha}+\nabla u_{0}\right|-\sum^{d}_{\alpha=1}|C^{\alpha}||\nabla u_{\alpha}|\notag\\
\gtrsim&\frac{1}{\varepsilon^{1-k/m}}
\begin{cases}
\frac{\eta(\kappa_{1}^{\frac{d+1}{m}}\kappa_{2}^{-\frac{d+k}{m}}+1)}{1+\kappa_{2}},&\text{if condition}\;({\bf{E2}})\;\text{holds for}\;m>d+k,\,k>2\\
&\text{or}\;m=d+k,\,k=1,\\
\frac{\kappa_{1}^{\frac{d+1}{m}}|Q^{\ast}_{d+1}[\varphi]|}{(1+\kappa_{2})\mathcal{L}_{d}^{d+1}},&\text{if condition}\;({\bf{E3}})\;\text{holds for}\;m=d+k,\,k>2,
\end{cases}
\end{align*}
where in the case of condition ({\bf{E3}}) we used the fact that $\varphi^{d}=0$ on $\Gamma^{-}_{R}$.

\end{proof}

\section{Asymptotic expansions of the stress concentration with strictly convex inclusions}\label{SEC005}
For future application, in this section we aim to establish the asymptotic expansions of the stress concentration when one strictly convex inclusion approaches the matrix boundary closely. Suppose that $h_{1}$ and $h$ satisfy
\begin{enumerate}
{\it\item[(\bf{Q1})]
$h_{1}(x')>h(x'),\;\mbox{if}\;\,x'\in B'_{2R}\setminus\{0'\},$
\item[(\bf{Q2})]
$h_{1}(0')=h(0')=\nabla_{x'}h_{1}(0')=\nabla_{x'}h(0')=0$,
\item[(\bf{Q3})]
$\nabla_{x'}^{2}(h_{1}(0')-h(0'))\geq\tau_{0} I$,
\item[(\bf{Q4})]
$\|h_{1}\|_{C^{3,1}(B'_{2R})}+\|h\|_{C^{3,1}(B'_{2R})}\leq \tau,$}
\end{enumerate}
where $\tau_{0}$ and $\tau$ are two positive constants independent of $\varepsilon$ and $I$ denotes the $(d-1)\times(d-1)$ identity matrix. Our main result in this section is as follows:
\begin{theorem}\label{THMM}
Assume that $D_{1}\subset D\subseteq\mathbb{R}^{d}\,(d\geq2)$ are defined as above, conditions $\mathrm{(}${\bf{Q1}}$\mathrm{)}$--$\mathrm{(}${\bf{Q4}}$\mathrm{)}$ hold, and $\varphi\in C^{2}(\partial D;\mathbb{R}^{d})$. Let $u\in H^{1}(D;\mathbb{R}^{d})\cap C^{1}(\overline{\Omega};\mathbb{R}^{d})$ be the solution of \eqref{La.002}. Then for a sufficiently small $\varepsilon>0$, $x\in\Omega_{R}$,

$(\rm{i})$ for $d=2$, if $\nabla_{x'}\varphi(0)=0$, then
\begin{align*}
\nabla u=&\sum^{2}_{\alpha=1}\frac{\det\mathbb{F}_{1}^{\ast\alpha}[\varphi]}{\det\mathbb{D}^{\ast}}\frac{\sqrt{\tau_{1}}\sqrt{\varepsilon}}{\sqrt{2}\pi\mathcal{L}_{2}^{\alpha}}\frac{1+O(\varepsilon^{\frac{1}{24}})}{1+\mathcal{G}_{2}^{\ast\alpha}\sqrt{\varepsilon}}\nabla\bar{u}_{\alpha}\notag\\
&+\frac{\det\mathbb{F}_{2}^{\ast3}[\varphi]}{\det\mathbb{D}^{\ast}}(1+O(\varepsilon^{\frac{1}{24}}))\nabla\bar{u}_{3}+O(1)\|\varphi\|_{C^{2}(\partial D)};
\end{align*}

$(\rm{ii})$ for $d=3$, then
\begin{align*}
\nabla u=&\sum_{\alpha=1}^{3}\frac{\det\mathbb{F}_{1}^{\ast\alpha}[\varphi]}{\det\mathbb{D}^{\ast}}\frac{\sqrt{\tau_{1}\tau_{2}}}{2\pi\mathcal{L}_{3}^{\alpha}}\frac{1+O(|\ln\varepsilon|^{-1})}{|\ln\varepsilon|+\mathcal{G}_{3}^{\ast\alpha}}\nabla\bar{u}_{\alpha}\notag\\
&+\sum^{6}_{\alpha=4}\frac{\det\mathbb{F}_{2}^{\ast\alpha}[\varphi]}{\det\mathbb{D}^{\ast}}(1+O(\varepsilon^{\frac{d-1}{24}}))\nabla\bar{u}_{\alpha}+O(1)\|\varphi\|_{C^{2}(\partial D)};
\end{align*}

$(\rm{iii})$ for $d\geq4$, then
\begin{align*}
\nabla u=\sum^{\frac{d(d+1)}{2}}_{\alpha=1}\frac{\det\mathbb{F}_{3}^{\ast\alpha}[\varphi]}{\det\mathbb{F}^{\ast}}(1+O(\varepsilon^{\min\{\frac{1}{6},\frac{d-3}{24}\}}))\nabla\bar{u}_{\alpha}+O(1)\|\varphi\|_{C^{2}(\partial D)},
\end{align*}
where $\tau_{1}$ and $\tau_{2}$ are the eigenvalues of $\nabla^{2}_{x'}(h_{1}-h)(0')$, $\mathcal{L}_{d}^{\alpha}$, $\alpha=1,2,...,d$ are defined by \eqref{AZ}--\eqref{AZ110},  the blow-up factor matrices $\mathbb{D}$, $\mathbb{F}_{1}^{\alpha}[\varphi],\,\alpha=1,2,...,d$, $\mathbb{F}_{2}^{\alpha}[\varphi],\,\alpha=d+1,...,\frac{d(d+1)}{2}$, $\mathbb{F}_{3}^{\alpha}[\varphi],\,\alpha=1,2,...,\frac{d(d+1)}{2}$, are defined by \eqref{ALHNCZ001}, \eqref{QKT001}--\eqref{QKT002} and \eqref{PGNC}, respectively, the geometry constants $\mathcal{G}_{d}^{\ast\alpha}$, $d=2,3$ are defined by \eqref{IMT001} below.

\end{theorem}
\begin{remark}
In contrast to the previous work \cite{BJL2017}, we improve the gradient estimates there to be the precise asymptotic expressions in Theorem \ref{THMM} here. We additionally point out that the geometry constants $\mathcal{G}_{d}^{\ast\alpha}$, $d=2,3$ depend not on the distance parameter $\varepsilon$ and the length $R$ of the narrow region $\Omega_{R}$, which is critical to numerical computations and simulations in future investigations. Although our asymptotic results can be generalized to more general $m$-convex inclusions, we restrict ourselves to the setup above for the convenience of computation and presentation.
\end{remark}

To begin with, a direct application of Lemma \ref{KM323} yields the following corollary.
\begin{corollary}\label{COO001}
Assume as in Theorem \ref{THMM}. Then for a sufficiently small $\varepsilon>0$,

$(\rm{i})$ for $\alpha=1,2,...,d$,
\begin{align}\label{LGAJ001}
Q_{\alpha}[\varphi]=Q_{\alpha}^{\ast}[\varphi]+O(1)&
\begin{cases}
\varepsilon^{\frac{1}{6}},&d=2,\\
\varepsilon^{\frac{d-2}{2d}},&d\geq3;
\end{cases}
\end{align}

$(\rm{ii})$ for $\alpha=d+1,...,\frac{d(d+1)}{2}$,
\begin{align}\label{LGAJ002}
Q_{\alpha}[\varphi]=Q_{\alpha}^{\ast}[\varphi]+O(1)&
\begin{cases}
\varepsilon^{\frac{1}{4}},&d=2,\\
\varepsilon^{\frac{d-1}{2(d+1)}},&d\geq3,
\end{cases}
\end{align}
where $Q_{\alpha}^{\ast}[\varphi]$, $\alpha=1,2,...,\frac{d(d+1)}{2}$ are defined in \eqref{FNCL001}.
\end{corollary}

\begin{lemma}\label{PLNDZ}
Assume as in Theorem \ref{THMM}. Then for a sufficiently small $\varepsilon>0$, if $\alpha=1,2,...,d$,

$(\rm{i})$ for $d=2$,
\begin{align*}
a_{\alpha\alpha}=&\frac{\sqrt{2}\pi\mathcal{L}_{2}^{\alpha}}{\sqrt{\tau_{1}}}\frac{1}{\sqrt{\varepsilon}}+\mathcal{K}_{2}^{\ast\alpha}+O(1)\varepsilon^{\frac{1}{24}};
\end{align*}

$(\rm{ii})$ for $d=3$,
\begin{align*}
a_{\alpha\alpha}=&\frac{2\pi\mathcal{L}_{3}^{\alpha}}{\sqrt{\tau_{1}\tau_{2}}}|\ln\varepsilon|+\mathcal{K}_{3}^{\ast\alpha}+O(1)\varepsilon^{\frac{1}{12}},
\end{align*}
where $\mathcal{K}_{d}^{\ast\alpha}$, $d=2,3$ are defined by \eqref{MGAT001}, $\tau_{1}$ and $\tau_{2}$ are the eigenvalues of $\nabla^{2}_{x'}(h_{1}-h)(0')$.
\end{lemma}
\begin{remark}
We here would like to point out that Li, Li and Yang \cite{LLY2019} were the first to give the precise computation of the energy for the perfect conductivity problem in the presence of two strictly convex inclusions.
\end{remark}
\begin{proof}
Analogous to \eqref{ZPH001}, we obtain that for $\alpha=1,2,...,d$,
\begin{align}\label{LRTA06}
a_{\alpha\alpha}=&\mathcal{L}_{d}^{\alpha}\left(\int_{\varepsilon^{\frac{1}{24}}<|x'|<R}\frac{dx'}{h_{1}(x')-h(x')}+\int_{|x'|<\varepsilon^{\frac{1}{24}}}\frac{dx'}{\varepsilon+h_{1}(x')-h(x')}\right)\nonumber\\
&+M_{d}^{\ast\alpha}+O(1)\varepsilon^{\min\{\frac{1}{6},\frac{d-1}{24}\}},
\end{align}
where $M_{d}^{\ast\alpha}$ is defined by \eqref{LTFA001}. According to the assumed conditions ({\bf{Q1}})--({\bf{Q4}}), after a rotation of the coordinates if necessary, we have
\begin{align*}
h_{1}(x')-h(x')=\sum^{d-1}_{i=1}\frac{\tau_{i}}{2}x_{i}^{2}+\sum_{|\alpha|=3}C_{\alpha}x'^{\alpha}+O(|x'|^{4}),
\end{align*}
where $\mathrm{diag}(\tau_{1},...,\tau_{d-1})=\nabla_{x'}^{2}(h_{1}-h)(0')$, $C_{\alpha}$ are some constants, $\alpha$ is an $(d-1)$-dimensional multi-index. In fact, $\tau_{1},...,\tau_{d-1}$ are the relative principal curvatures of $\partial D_{1}$ and $\partial D$ at the origin. Observe that
\begin{align}\label{LATCZ001001}
&\int_{|x'|<\varepsilon^{\frac{1}{24}}}\frac{1}{\varepsilon+h_{1}-h}-\int_{|x'|<\varepsilon^{\bar{\gamma}}}\frac{1}{\varepsilon+\sum^{d-1}_{i=1}\frac{\tau_{i}}{2}x_{i}^{2}}\notag\\
&=\int_{|x'|<\varepsilon^{\frac{1}{24}}}\frac{1}{\varepsilon+\sum^{d-1}_{i=1}\frac{\tau_{i}}{2}x_{i}^{2}}\left[\bigg(1+\frac{\sum_{|\alpha|=3}C_{\alpha}x'^{\alpha}}{\varepsilon+\sum^{d-1}_{i=1}\frac{\lambda_{i}}{2}x_{i}^{2}}+O(|x'|^{2})\bigg)^{-1}-1\right]\notag\\
&=-\int_{|x'|<\varepsilon^{\frac{1}{24}}}\frac{\sum_{|\alpha|=3}C_{\alpha}x'^{\alpha}}{(\varepsilon+\sum^{d-1}_{i=1}\frac{\tau_{i}}{2}x_{i}^{2})^{2}}+\int_{|x'|<\varepsilon^{\frac{1}{24}}}O(1)=O(\varepsilon^{\frac{d-1}{24}}),
\end{align}
where in the last line the Taylor expansion was utilized in virtue of the smallness of $R$ and we used the fact that $\sum_{|\alpha|=3}C_{\alpha}x'^{\alpha}$ is odd and the integrating domain is symmetric. Analogously, we get
\begin{align}\label{LATCZ002001}
&\int_{\varepsilon^{\frac{1}{24}}<|x'|<R}\frac{1}{h_{1}-h}-\int_{\varepsilon^{\frac{1}{24}}<|x'|<R}\frac{1}{\sum^{d-1}_{i=1}\frac{\tau_{i}}{2}x_{i}^{2}}=C^{\ast}+O(\varepsilon^{\frac{d-1}{24}}),
\end{align}
where the constant $C^{\ast}$ depends on $d,R,\tau_{1},\tau_{2}$, but not on $\varepsilon$. Observe that
\begin{align*}
&\int_{\varepsilon^{\frac{1}{24}}<|x'|<R}\frac{1}{\sum^{d-1}_{i=1}\frac{\tau_{i}}{2}x_{i}^{2}}+\int_{|x'|<\varepsilon^{\frac{1}{24}}}\frac{1}{\varepsilon+\sum^{d-1}_{i=1}\frac{\tau_{i}}{2}x_{i}^{2}}\notag\\
&=\int_{|x'|<R}\frac{1}{\varepsilon+\sum^{d-1}_{i=1}\frac{\tau_{i}}{2}x_{i}^{2}}+\int_{\varepsilon^{\frac{1}{24}}<|x'|<R}\frac{\varepsilon}{\sum^{d-1}_{i=1}\frac{\tau_{i}}{2}x_{i}^{2}(\varepsilon+\sum^{d-1}_{i=1}\frac{\tau_{i}}{2}x_{i}^{2})}\notag\\
&=\int_{|x'|<R}\frac{1}{\varepsilon+\sum^{d-1}_{i=1}\frac{\tau_{i}}{2}x_{i}^{2}}+O(1)\varepsilon^{\frac{d+19}{24}}.
\end{align*}
This, in combination with \eqref{LATCZ001001}--\eqref{LATCZ002001}, leads to that
\begin{align}\label{GARN001001}
&\int_{\varepsilon^{\frac{1}{24}}<|x'|<R}\frac{dx'}{h_{1}(x')-h(x')}+\int_{|x'|<\varepsilon^{\frac{1}{24}}}\frac{dx'}{\varepsilon+h_{1}(x')-h(x')}\notag\\
&=\int_{|x'|<R}\frac{1}{\varepsilon+\sum^{d-1}_{i=1}\frac{\tau_{i}}{2}x_{i}^{2}}+C^{\ast}+O(\varepsilon^{\frac{d-1}{24}}).
\end{align}
Then, for $d=2$, we have
\begin{align}\label{GARN002001}
\int_{|x_{1}|<R}\frac{1}{\varepsilon+\frac{\tau_{1}}{2}x_{1}^{2}}=&\int_{-\infty}^{+\infty}\frac{1}{\varepsilon+\frac{\tau_{1}}{2}x_{1}^{2}}-\int_{|x_{1}|>R}\frac{1}{\frac{\tau_{1}}{2}x_{1}^{2}}+\int_{|x_{1}|>R}\frac{\varepsilon}{\frac{\tau_{1}}{2}x_{1}^{2}(\varepsilon+\frac{\tau_{1}}{2}x_{1}^{2})}\notag\\
=&\frac{\sqrt{2}\pi}{\sqrt{\tau_{1}}}\frac{1}{\sqrt{\varepsilon}}-\frac{4}{\tau_{1}R}+O(1)\varepsilon,
\end{align}
while, for $d=3$,
\begin{align}\label{GARN003001}
\int_{|x'|<R}\frac{1}{\varepsilon+\sum^{2}_{i=1}\frac{\tau_{i}}{2}x_{i}^{2}}=&\frac{8}{\sqrt{\tau_{1}\tau_{2}}}\int^{\frac{\pi}{2}}_{0}\int_{0}^{R(\theta)}\frac{t}{\varepsilon+t^{2}}\,dtd\theta\notag\\
=&\frac{4}{\sqrt{\tau_{1}\tau_{2}}}\int^{\frac{\pi}{2}}_{0}\bigg(|\ln\varepsilon|+\ln(R(\theta)^{2})+\ln\Big(1+\frac{\varepsilon}{R(\theta)^{2}}\Big)\bigg)\,d\theta\notag\\
=&\frac{2\pi}{\sqrt{\tau_{1}\tau_{2}}}|\ln\varepsilon|+\frac{8}{\sqrt{\tau_{1}\tau_{2}}}\int^{\frac{\pi}{2}}_{0}\ln R(\theta)d\theta+O(\varepsilon),
\end{align}
where $R(\theta):=\frac{R}{\sqrt{2}}(\tau_{1}^{-1}\cos^{2}\theta+\tau_{2}^{-1}\sin^{2}\theta)^{-1/2}$ and in the last line we used the Taylor expansion of $\ln(1+x)$ for $|x|<1$. Consequently, in view of \eqref{LRTA06}, it follows from \eqref{GARN001001}--\eqref{GARN003001} that for $\alpha=1,2,...,d$,
\begin{align}\label{DHMZ001}
a_{\alpha\alpha}=&
\begin{cases}
\frac{\sqrt{2}\pi\mathcal{L}_{2}^{\alpha}}{\sqrt{\tau_{1}}}\frac{1}{\sqrt{\varepsilon}}+\mathcal{K}_{2}^{\ast\alpha}+O(1)\varepsilon^{\frac{1}{24}},&d=2,\\
\frac{2\pi\mathcal{L}_{3}^{\alpha}}{\sqrt{\tau_{1}\tau_{2}}}|\ln\varepsilon|+\mathcal{K}_{3}^{\ast\alpha}+O(1)\varepsilon^{\frac{1}{12}},&d=3,
\end{cases}
\end{align}
where
\begin{align}\label{MGAT001}
\mathcal{K}_{d}^{\ast\alpha}=&
\begin{cases}
\mathcal{L}_{2}^{\alpha}\left(C^{\ast}-\frac{4}{\tau_{1}R}\right)+\mathcal{M}_{2}^{\ast\alpha},&d=2,\\
\mathcal{L}_{3}^{\alpha}\left(C^{\ast}+\frac{8}{\sqrt{\tau_{1}\tau_{2}}}\int^{\frac{\pi}{2}}_{0}\ln R(\theta)d\theta\right)+\mathcal{M}_{3}^{\ast\alpha},&d=3.
\end{cases}
\end{align}

We now demonstrate that the geometry constants $\mathcal{K}_{d}^{\ast\alpha}$, $d=2,3$ are independent of the length parameter $R$ of the thin gap $\Omega_{R}$. Otherwise, assume that there exist $\varepsilon$-independent constants $\mathcal{K}_{d}^{\ast\alpha}(R_{1})$ and $\mathcal{K}_{d}^{\ast\alpha}(R_{2})$, $R_{1}\neq R_{2}$, such that \eqref{DHMZ001} holds. Then
\begin{align*}
\mathcal{K}_{d}^{\ast\alpha}(R_{1})-\mathcal{K}_{d}^{\ast\alpha}(R_{2})=O(1)\varepsilon^{\frac{d-1}{24}},
\end{align*}
which yields that $\mathcal{K}_{d}^{\ast\alpha}(R_{1})=\mathcal{K}_{d}^{\ast\alpha}(R_{2})$.

\end{proof}

\begin{proof}[The proof of Theorem \ref{THMM}.]
Denote
\begin{align}\label{IMT001}
\mathcal{G}_{d}^{\ast\alpha}=&
\begin{cases}
\frac{\sqrt{\tau_{1}}\mathcal{K}_{2}^{\ast\alpha}}{\sqrt{2}\pi\mathcal{L}_{2}^{\alpha}},&d=2,\\
\frac{\sqrt{\tau_{1}\tau_{2}}\mathcal{K}_{3}^{\ast\alpha}}{2\pi\mathcal{L}_{3}^{\alpha}},&d=3,
\end{cases}\quad\quad\alpha=1,2,...,d.
\end{align}
For $\alpha=1,2,...,d$, it follows from Lemma \ref{PLNDZ} that
\begin{align}\label{ALGK001}
\frac{1}{a_{\alpha\alpha}}=&
\begin{cases}
\frac{\sqrt{\varepsilon}}{\frac{\sqrt{2}\pi\mathcal{L}_{2}^{\alpha}}{\sqrt{\tau_{1}}}}\frac{1}{1-\frac{\frac{\sqrt{2}\pi\mathcal{L}_{2}^{\alpha}}{\sqrt{\tau_{1}}}-\sqrt{\varepsilon}a_{\alpha\alpha}}{\frac{\sqrt{2}\pi\mathcal{L}_{2}^{\alpha}}{\sqrt{\tau_{1}}}}}=\frac{\sqrt{\tau_{1}}}{\sqrt{2}\pi\mathcal{L}_{2}^{\alpha}}\frac{\sqrt{\varepsilon}(1+O(\varepsilon^{\frac{13}{24}}))}{1+\mathcal{G}_{2}^{\ast\alpha}\sqrt{\varepsilon}},&d=2,\\
\frac{|\ln\varepsilon|^{-1}}{\frac{2\pi\mathcal{L}_{3}^{\alpha}}{\sqrt{\tau_{1}\tau_{2}}}}\frac{1}{1-\frac{\frac{2\pi\mathcal{L}_{3}^{\alpha}}{\sqrt{\tau_{1}\tau_{2}}}-|\ln\varepsilon|^{-1}a_{\alpha\alpha}}{\frac{2\pi\mathcal{L}_{3}^{\alpha}}{\sqrt{\tau_{1}\tau_{2}}}}}=\frac{\sqrt{\tau_{1}\tau_{2}}}{2\pi\mathcal{L}_{3}^{\alpha}}\frac{1+O(\varepsilon^{\frac{1}{12}}|\ln\varepsilon|^{-1})}{|\ln\varepsilon|+\mathcal{G}_{3}^{\ast\alpha}},&d=3.
\end{cases}
\end{align}
In light of \eqref{LMC}--\eqref{LVZQ0011}, \eqref{PAK001}--\eqref{PAK003} and \eqref{LGAJ001}--\eqref{LGAJ002}, we deduce that if $d=2,3$,
\begin{align}
\frac{\det\mathbb{F}_{1}^{\alpha}[\varphi]}{\det\mathbb{D}}=&\frac{\det\mathbb{F}_{1}^{\ast\alpha}[\varphi]}{\det\mathbb{D}^{\ast}}\frac{1}{1-\frac{\det\mathbb{D}^{\ast}-\det\mathbb{D}}{\det\mathbb{D}^{\ast}}}+\frac{\det\mathbb{F}_{1}^{\alpha}[\varphi]-\det\mathbb{F}_{1}^{\ast\alpha}[\varphi]}{\det\mathbb{D}}\notag\\
=&\frac{\det\mathbb{F}_{1}^{\ast\alpha}[\varphi]}{\det\mathbb{D}^{\ast}}(1+O(\varepsilon^{\frac{d-1}{24}})),\quad \alpha=1,2,...,d,\label{DRT001}\\
\frac{\det\mathbb{F}_{2}^{\alpha}[\varphi]}{\det\mathbb{D}}=&\frac{\det\mathbb{F}_{2}^{\ast\alpha}[\varphi]}{\det\mathbb{D}^{\ast}}\frac{1}{1-\frac{\det\mathbb{D}^{\ast}-\det\mathbb{D}}{\det\mathbb{D}^{\ast}}}+\frac{\det\mathbb{F}_{2}^{\alpha}[\varphi]-\det\mathbb{F}_{2}^{\ast\alpha}[\varphi]}{\det\mathbb{D}}\notag\\
=&\frac{\det\mathbb{F}_{2}^{\ast\alpha}[\varphi]}{\det\mathbb{D}^{\ast}}(1+O(\varepsilon^{\frac{d-1}{24}})),\quad\alpha=d+1,...,\frac{d(d+1)}{2},\label{DRT003}
\end{align}
and, if $d\geq4$,
\begin{align}\label{DRT002}
\frac{\det\mathbb{F}_{3}^{\alpha}[\varphi]}{\det\mathbb{F}}=&\frac{\det\mathbb{F}_{3}^{\ast\alpha}[\varphi]}{\det\mathbb{F}^{\ast}}\frac{1}{1-\frac{\det\mathbb{F}^{\ast}-\det\mathbb{F}}{\det\mathbb{F}^{\ast}}}+\frac{\det\mathbb{F}_{3}^{\alpha}[\varphi]-\det\mathbb{F}_{3}^{\ast\alpha}[\varphi]}{\det\mathbb{F}}\notag\\
=&\frac{\det\mathbb{F}_{3}^{\ast\alpha}[\varphi]}{\det\mathbb{F}^{\ast}}(1+O(\varepsilon^{\min\{\frac{1}{6},\frac{d-3}{24}\}})),\quad \alpha=1,2,...,\frac{d(d+1)}{2},
\end{align}
where the blow-up factor matrices $\mathbb{D}$, $\mathbb{F}_{1}^{\alpha}[\varphi],\,\alpha=1,2,...,d$, $\mathbb{F}_{2}^{\alpha}[\varphi],\,\alpha=d+1,...,\frac{d(d+1)}{2}$, $\mathbb{F}_{3}^{\alpha}[\varphi],\,\alpha=1,2,...,\frac{d(d+1)}{2}$, are defined by \eqref{ALHNCZ001}, \eqref{QKT001}--\eqref{QKT002} and \eqref{PGNC}, respectively.

Let
\begin{align*}
\rho_{d}(\varepsilon):=
\begin{cases}
\sqrt{\varepsilon}|\ln\varepsilon|^{-2},&d=2,\\
|\ln\varepsilon|^{-1},&d=3.
\end{cases}
\end{align*}
Then combining \eqref{ALGK001}--\eqref{DRT002}, it follows from Cramer's rule that

$(\rm{i})$ if $d=2,3$, for $\alpha=1,2,...,d$,
\begin{align}\label{LGNA001001}
C^{\alpha}=&
\frac{\prod\limits_{i\neq\alpha}^{d}a_{ii}\det\mathbb{F}_{1}^{\alpha}[\varphi]}{\prod\limits_{i=1}^{d}a_{ii}\det\mathbb{D}}(1+O(\rho_{d}(\varepsilon)))\notag\\
=&
\begin{cases}
\frac{\det\mathbb{F}_{1}^{\ast\alpha}[\varphi]}{\det\mathbb{D}^{\ast}}\frac{\sqrt{\tau_{1}}\sqrt{\varepsilon}}{\sqrt{2}\pi\mathcal{L}_{2}^{\alpha}}\frac{1+O(\varepsilon^{\frac{1}{24}})}{1+\mathcal{G}_{2}^{\ast\alpha}\sqrt{\varepsilon}},&d=2,\\
\frac{\det\mathbb{F}_{1}^{\ast\alpha}[\varphi]}{\det\mathbb{D}^{\ast}}\frac{\sqrt{\tau_{1}\tau_{2}}}{2\pi\mathcal{L}_{3}^{\alpha}}\frac{1+O(|\ln\varepsilon|^{-1})}{|\ln\varepsilon|+\mathcal{G}_{3}^{\ast\alpha}},&d=3,
\end{cases}
\end{align}
and, for $\alpha=d+1,...,\frac{d(d+1)}{2}$,
\begin{align}\label{LGNA002001}
C^{\alpha}=&
\frac{\det\mathbb{F}_{2}^{\alpha}[\varphi]}{\det\mathbb{D}}=\frac{\det\mathbb{F}_{2}^{\ast\alpha}[\varphi]}{\det\mathbb{D}^{\ast}}(1+O(\varepsilon^{\frac{d-1}{24}}));
\end{align}

$(\rm{ii})$ if $d\geq4$, for $\alpha=1,2,...,\frac{d(d+1)}{2}$,
\begin{align}\label{HUTO}
C^{\alpha}=\frac{\det\mathbb{F}_{3}^{\alpha}[\varphi]}{\det\mathbb{F}}=\frac{\det\mathbb{F}_{3}^{\ast\alpha}[\varphi]}{\det\mathbb{F}^{\ast}}(1+O(\varepsilon^{\min\{\frac{1}{6},\frac{d-3}{24}\}})).
\end{align}

Consequently, substituting \eqref{ATCG001} and \eqref{LGNA001001}--\eqref{HUTO} into \eqref{Le2.015}, we obtain that Theorem \ref{THMM} holds.

\end{proof}

\section{Appendix:\,The proof of Theorem \ref{thm8698}}\label{SL005}
In the following, the iterate technique developed in \cite{LLBY2014} will be used to prove Theorem \ref{thm8698}. For the sake of discussion and presentation, we choose $\psi=0$ on $\partial D_{1}$ in \eqref{P2.008}. To begin with, we decompose the solution of \eqref{P2.008} as follows:
$$v=\sum^{d}_{j=1}v_{j},\quad v_{j}=(v_{j}^{1},v_{j}^{2},...,v_{j}^{d})^{T},$$
where $v_{j}$, $j=1,2,...,\frac{d(d+1)}{2}$, satisfy that $v_{j}^{i}=0$, $j\neq i$, and $v_{j}$ is a solution to the following problem
\begin{align}\label{P2.010}
\begin{cases}
  \mathcal{L}_{\lambda,\mu}v_{j}:=\nabla\cdot(\mathbb{C}^0e(v_{j}))=0,\quad&
\hbox{in}\  \Omega,  \\
v_{j}=0,\ &\hbox{on}\ \partial{D}_{1},\\
v_{j}=(0,...,0,\phi^{j}, 0,...,0)^{T},&\hbox{on} \ \partial{D}.
\end{cases}
\end{align}
Then
\begin{align*}
\nabla v=\sum^{d}_{j=1}\nabla v_{j}.
\end{align*}

For $x\in\Omega_{2R}$, denote
\begin{align*}
\tilde{v}_{j}:=&\phi^{j}(x',h(x'))(1-\bar{v})e_{j}-\frac{\lambda+\mu}{\lambda+2\mu}f(\bar{v})\phi^{j}(x',h(x'))\partial_{x_{j}}\delta\,e_{d},\quad j=1,...,d-1,
\end{align*}
and
\begin{align*}
\tilde{v}_{d}:=&\phi^{d}(x',h(x'))(1-\bar{v})e_{d}-\frac{\lambda+\mu}{\mu}f(\bar{v})\phi^{d}(x',h(x'))\sum^{d-1}_{i=1}\partial_{x_{i}}\delta\,e_{i}.
\end{align*}
Then
\begin{align*}
\tilde{v}=\sum^{d}_{j=1}\tilde{v}_{j},
\end{align*}
where $\tilde{v}$ is defined by \eqref{AHM001} with $\psi=0$ on $\partial D_{1}$. It follows from a direct computation that for $j=1,2,...,d$, $x\in\Omega_{2R}$,
\begin{align}
|\mathcal{L}_{\lambda,\mu}\tilde{v}_{j}|\leq&\frac{C|\phi^{j}(x',h(x'))|}{\delta^{2/m}}+\frac{C|\nabla_{x'}\phi^{j}(x',h(x'))|}{\delta}+C\|\phi^{j}\|_{C^{2}(\partial D)}.\label{QQ.103}
\end{align}

For $x\in\Omega_{2R}$, write
\begin{equation*}
w_j:=v_j-\tilde{v}_j,\quad j=1,2,...,d.
\end{equation*}
Therefore, $w_{j}$ satisfies
\begin{align}\label{LHM001}
\begin{cases}
\mathcal{L}_{\lambda,\mu}w_{j}=-\mathcal{L}_{\lambda,\mu}\tilde{v}_{j},&
\hbox{in}\  \Omega_{2R},  \\
w_{j}=0, \quad&\hbox{on} \ \Gamma^{\pm}_{2R}.
\end{cases}
\end{align}

We next divide into two subparts to prove Theorem \ref{thm8698}. For the convenience of presentation, in the following we utilize $\|\phi^{j}\|_{C^{1}}$ and $\|\phi^{j}\|_{C^{2}}$ to denote $\|\phi^{j}\|_{C^{1}(\partial D)}$ and $\|\phi^{j}\|_{C^{2}(\partial D)}$, respectively.

\noindent{\bf Part 1.}
Let $v_j\in H^1(\Omega;\mathbb{R}^{d})$ be the weak solution of \eqref{P2.010}, $j=1,2,...,d$. Then
\begin{align}\label{QTNY001}
\int_{\Omega_{2R}}|\nabla w_j|^2dx\leq C\|\phi^{j}\|_{C^{2}}^{2},\quad j=1,2,...,d.
\end{align}

Firstly, we extend $\phi\in C^{2}(\partial D;\mathbb{R}^{d})$ to $\phi\in C^{2}(\overline{\Omega};\mathbb{R}^{d})$ satisfying that for $j=1,2,...,d$, $\|\phi^{j}\|_{C^{2}(\overline{\Omega\setminus\Omega_{R}})}\leq C\|\phi^{j}\|_{C^{2}(\partial D)}$. Pick a smooth cutoff function $\rho\in C^{2}(\overline{\Omega})$ such that $0\leq\rho\leq1$, $|\nabla\rho|\leq C$ in $\overline{\Omega}$, and
\begin{align}\label{Le2.021}
\rho=1\;\,\mathrm{in}\;\,\Omega_{\frac{3}{2}R},\quad\rho=0\;\,\mathrm{in}\;\,\overline{\Omega}\setminus\Omega_{2R}.
\end{align}
For $x\in\Omega$, denote
\begin{align*}
\hat{v}_{j}(x)=\left(0,...,0,[\rho(x)\phi^{j}(x',h(x'))+(1-\rho(x))\phi^{j}(x)](1-\bar{v}(x)),0,...,0\right)^{T}.
\end{align*}
Especially,
\begin{align*}
\hat{v}_{j}(x)=(0,...,0,\phi^{j}(x',h(x'))(1-\bar{v}(x)),0,...,0)^{T},\quad\;\,\mathrm{in}\;\,\Omega_{R}.
\end{align*}
From \eqref{BATL001} and \eqref{Le2.021}, we obtain
\begin{align}\label{JAT001}
\|\hat{v}_{j}\|_{C^{2}(\Omega\setminus\Omega_{R})}\leq C\|\phi^{j}\|_{C^{2}},\quad j=1,2,...,d.
\end{align}
For $j=1,2,...,d$, write $\hat{w}_{j}:=v_{j}-\hat{v}_{j}$ in $\Omega$. Consequently, $\hat{w}_{j}$ satisfies
\begin{align}\label{ESGH01}
\begin{cases}
\mathcal{L}_{\lambda,\mu}\hat{w}_{j}=-\mathcal{L}_{\lambda,\mu}\hat{v}_{j},&
\hbox{in}\  \Omega,  \\
\hat{w}_{j}=0, \quad&\hbox{on} \ \partial\Omega.
\end{cases}
\end{align}
By a direct calculation, we deduce that for $i=1,...,d-1$, $x\in\Omega_{2R}$,
\begin{align}
|\partial_{x_{i}}\hat{v}_{j}|\leq&C|\phi^{j}(x',h(x'))|\delta^{-1/m}+\|\phi^{j}\|_{C^{1}},\label{QQ.10101}\\
|\partial_{x_{d}}\hat{v}_{j}|=&|\phi^{j}(x',h(x'))|\delta^{-1},\quad\partial_{x_{d}x_{d}}\hat{v}_{j}=0.\label{QQ.101}
\end{align}
Multiplying equation (\ref{ESGH01}) by $\hat{w}_{j}$ and integrating by parts, we have
\begin{align}\label{AK001}
\int_{\Omega}\left(\mathbb{C}^0e(\hat{w}_{j}),e(\hat{w}_{j})\right)dx=\int_{\Omega}\hat{w}_{j}\left(\mathcal{L}_{\lambda,\mu}\hat{v}_{j}\right)dx.
\end{align}
From the Poincar\'{e} inequality, we have
\begin{align}\label{AK002}
\|\hat{w}_{j}\|_{L^2(\Omega\setminus\Omega_R)}\leq\,C\|\nabla \hat{w}_{j}\|_{L^2(\Omega\setminus\Omega_R)}.
\end{align}
The first Korn's inequality, in combination with \eqref{Le2.012} and \eqref{AK001}, gives that
\begin{align}\label{KWQP001}
\int_{\Omega}|\nabla\hat{w}_{j}|^2dx\leq &\,2\int_{\Omega}|e(\hat{w}_{j})|^2dx\nonumber\\
\leq&\,C\left|\int_{\Omega_R}\hat{w}_{j}(\mathcal{L}_{\lambda,\mu}\hat{v}_{j})dx\right|+C\left|\int_{\Omega\setminus\Omega_R}\hat{w}_{j}(\mathcal{L}_{\lambda,\mu}\hat{v}_{j})dx\right|.
\end{align}

On one hand, from \eqref{JAT001} and \eqref{AK002} we deduce that
\begin{align}\label{KWQP002}
\left|\int_{\Omega\setminus\Omega_R}\hat{w}_{j}(\mathcal{L}_{\lambda,\mu}\hat{v}_{j})dx\right|\leq&\|\phi^{j}\|_{C^2}\int_{\Omega\setminus\Omega_R}|\hat{w}_{j}|dx\leq C\|\phi^{j}\|_{C^2}\|\nabla \hat{w}_{j}\|_{L^{2}(\Omega\setminus\Omega_{R})}.
\end{align}

On the other hand, observe that the Sobolev trace embedding theorem allows us to write
\begin{align}\label{trace}
\int\limits_{\scriptstyle |x'|={R},\atop\scriptstyle
h(x')<x_{d}<{\varepsilon}+h_1(x')\hfill}|\hat{w}_{j}|\,dx\leq\ C \left(\int_{\Omega\setminus\Omega_{R}}|\nabla \hat{w}_{j}|^2dx\right)^{\frac{1}{2}},
\end{align}
where $C$ is a positive constant independent of $\varepsilon$. In view of (\ref{QQ.10101}), we deduce
\begin{align}\label{JANC}
\int_{\Omega_{R}}|\nabla_{x'}\hat{v}_{j}|^2dx\leq& C\int_{|x'|<R}\delta\left(\frac{|\phi^{j}(x',h(x'))|^{2}}{\delta^{2/m}}+\|\phi^{j}\|_{C^{1}}^{2}\right)dx'\leq C\|\phi^{j}\|_{C^{1}}^{2}.
\end{align}
Then combining \eqref{QQ.101} and \eqref{trace}--\eqref{JANC}, we obtain
\begin{align*}
&\left|\int_{\Omega_R}\hat{w}_{j}(\mathcal{L}_{\lambda,\mu}\hat{v}_{j})dx\right|
\leq\,C\sum_{k+l<2d}\left|\int_{\Omega_{R}}\hat{w}_{j}\partial_{x_{k}x_{l}}\hat{v}_{j}dx\right|\nonumber\\
\leq&C\int_{\Omega_{R}}|\nabla \hat{w}_{j}\|\nabla_{x'}\hat{v}_{j}|dx+\int\limits_{\scriptstyle |x'|={R},\atop\scriptstyle
h(x')<x_{d}<\varepsilon+h_1(x')\hfill}C|\nabla_{x'}\hat{v}_{i}\|\hat{w}_{j}|dx\nonumber \\
\leq&C\|\nabla \hat{w}_{j}\|_{L^{2}(\Omega_{R})}\|\nabla_{x'}\hat{v}_{j}\|_{L^{2}(\Omega_{R})}+C\|\phi^{j}\|_{C^1}\|\nabla \hat{w}_{j}\|_{L^{2}(\Omega\setminus\Omega_{R})}\nonumber\\
\leq&C\|\phi^{j}\|_{C^{1}}\|\nabla \hat{w}_{j}\|_{L^{2}(\Omega)}.
\end{align*}
This, together with \eqref{KWQP001}--\eqref{KWQP002}, yields that
\begin{align*}
\|\nabla \hat{w}_{j}\|_{L^{2}(\Omega)}\leq  C\|\phi^{j}\|_{C^{2}}.
\end{align*}
Due to the fact that
\begin{align*}
w_{j}=\hat{w}_{j}+
\begin{cases}
\frac{\lambda+\mu}{\lambda+2\mu}f(\bar{v})\phi^{j}(x',h(x'))\partial_{x_{j}}\delta\,e_{d},&j=1,...,d-1,\\
\frac{\lambda+\mu}{\mu}f(\bar{v})\psi^{d}(x',h(x'))\sum\limits^{d-1}_{i=1}\partial_{x_{i}}\delta\,e_{i},&j=d,
\end{cases}\quad \mathrm{in}\;\Omega_{2R},
\end{align*}
it follows that \eqref{QTNY001} holds.

\noindent{\bf Part 2.}
Proof of
\begin{align}\label{step2}
 \int_{\Omega_\delta(z')}|\nabla w_{j}|^2dx\leq& C\delta^{d}(|\phi^{j}(z',h(z'))|^{2}\delta^{2-4/m}+|\nabla_{x'}\phi^{j}(z',h(z'))|^{2})\notag\\
&+C\delta^{d+2}\|\phi^{j}\|_{C^{2}(\partial D)}^{2},\quad j=1,2,...,d.
\end{align}
For $0<t<s<R$ and $|z'|\leq R$, let $\eta\in C^{2}(\Omega_{2R})$ be a smooth cutoff function satisfying that $\eta(x')=1$ if $|x'-z'|<t$, $\eta(x')=0$ if $|x'-z'|>s$, $0\leq\eta(x')\leq1$ if $t\leq|x'-z'|\leq s$, and $|\nabla_{x'}\eta(x')|\leq\frac{2}{s-t}$. Multiplying equation \eqref{LHM001} by $w_{j}\eta^{2}$, it follows from integration by parts that
\begin{align}\label{LFN001}
\int_{\Omega_{s}(z')}\left(\mathbb{C}^0e(w_{j}),e(w_{j}\eta^{2})\right)dx=\int_{\Omega_{s}(z')}w_{j}\eta^{2}\left(\mathcal{L}_{\lambda,\mu}\tilde{v}_{j}\right)dx.
\end{align}
On one hand, from \eqref{ACDT001}, \eqref{Le2.012} and the first Korn's inequality, we deduce
\begin{align}\label{MAH0199}
\int_{\Omega_{s}(z')}\left(\mathbb{C}^0e(w_{j}),e(w_{j}\eta^{2})\right)dx\geq\frac{1}{C}\int_{\Omega_{s}(z')}|\eta\nabla w_{j}|^{2}-C\int_{\Omega_{s}(z')}|\nabla\eta|^{2}|w_{j}|^{2}.
\end{align}
On the other hand, applying H\"{o}lder inequality and Cauchy inequality to the right hand side of \eqref{LFN001}, we deduce
\begin{align*}
\left|\int_{\Omega_{s}(z')}w_{j}\eta^{2}\left(\mathcal{L}_{\lambda,\mu}\tilde{v}_{j}\right)dx\right|\leq \frac{C}{(s-t)^{2}}\int_{\Omega_{s}(z')}|w_{j}|^{2}dx+C(s-t)^{2}\int_{\Omega_{s}(z')}|\mathcal{L}_{\lambda,\mu}\tilde{v}_{j}|^{2}dx,
\end{align*}
which, in combination with \eqref{LFN001}--\eqref{MAH0199}, gives the following iteration formula:
\begin{align*}
\int_{\Omega_{t}(z')}|\nabla w_{j}|^{2}dx\leq\frac{C}{(s-t)^{2}}\int_{\Omega_{s}(z')}|w_{j}|^{2}dx+C(s-t)^{2}\int_{\Omega_{s}(z')}|\mathcal{L}_{\lambda,\mu}\tilde{v}_{j}|^{2}dx.
\end{align*}

For $|z'|\leq R$, $\delta<s\leq\vartheta(\kappa_{1},\kappa_{3})\delta^{1/m}$, $\vartheta(\kappa_{1},\kappa_{3})=\frac{1}{2^{m+1}\kappa_{3}\max\{1,\kappa_{1}^{1/m-1}\}}$, using conditions ({\bf{H1}}) and ({\bf{H2}}), we obtain that for $(x',x_{d})\in\Omega_{s}(z')$,
\begin{align}\label{KHW01}
|\delta(x')-\delta(z')|\leq&|h_{1}(x')-h_{1}(z')|+|h(x')-h(z')|\notag\\
\leq&(|\nabla_{x'}h_{1}(x'_{\theta_{1}})|+|\nabla_{x'}h(x'_{\theta_{2}})|)|x'-z'|\notag\\
\leq&\kappa_{3}|x'-z'|(|x'_{\theta_{1}}|^{m-1}+|x'_{\theta_{2}}|^{m-1})\notag\\
\leq&2^{m-1}\kappa_{3}s(s^{m-1}+|z'|^{m-1})\notag\\
\leq&\frac{\delta(z')}{2},
\end{align}
which implies that
\begin{align}\label{QWN001}
\frac{1}{2}\delta(z')\leq\delta(x')\leq\frac{3}{2}\delta(z'),\quad\mathrm{in}\;\Omega_{s}(z').
\end{align}
In view of the fact that $w_{j}=0$ on $\Gamma^{-}_{R}$, it follows from \eqref{QWN001} that
\begin{align}\label{AQ3.038}
\int_{\Omega_{s}(z')}|w_{j}|^{2}\leq C\delta^{2}\int_{\Omega_{s}(z')}|\nabla w_{j}|^{2}.
\end{align}
From (\ref{QQ.103}) and \eqref{QWN001}, we have
\begin{align}\label{AQ3.037}
\int_{\Omega_{s}(z')}|\mathcal{L}_{\lambda,\mu}\tilde{v}_{j}|^{2}\leq& C\delta^{-1}s^{d-1}(|\phi^{j}(z',h(z'))|^{2}\delta^{2-4/m}+|\nabla_{x'}\phi^{j}(z',h(z'))|^{2})\notag\\
&+C\delta^{-1}s^{d+1}\|\phi^{j}\|_{C^{2}}^{2}.
\end{align}

Write
$$F(t):=\int_{\Omega_{t}(z')}|\nabla w_{j}|^{2}.$$
Then combining \eqref{AQ3.038} and \eqref{AQ3.037}, we deduce that
\begin{align*}
F(t)\leq& \left(\frac{c\delta}{s-t}\right)^2F(s)+C(s-t)^2s^{d+1}\delta^{-1}\|\phi^{j}\|_{C^{2}(\partial D)}^{2}\notag\\
&+C(s-t)^2\delta^{-1}s^{d-1}(|\phi^{j}(z',h(z'))|^{2}\delta^{2-4/m}+|\nabla_{x'}\phi^{j}(z',h(z'))|^{2}),
\end{align*}
which, together with $s=t_{i+1}$, $t=t_{i}$, $t_{i}=\delta+2ci\delta,\;i=0,1,2,...,\left[\frac{\vartheta(\kappa_{1},\kappa_{3})}{4c\delta^{1-1/m}}\right]+1$, leads to that
\begin{align}\label{ALZ001}
F(t_{i})\leq&\frac{1}{4}F(t_{i+1})+C(i+1)^{d+1}\delta^{d+2}\|\phi^{j}\|_{C^{2}}^{2}\notag\\
&+C(i+1)^{d-1}\delta^{d}(|\phi^{j}(z',h(z'))|^{2}\delta^{2-4/m}+|\nabla_{x'}(\phi^{j}(z',h(z')))|^{2}).
\end{align}
Hence, applying $\left[\frac{\vartheta(\kappa_{1},\kappa_{3})}{4c\delta^{1-1/m}}\right]+1$ iterations to \eqref{ALZ001}, it follows from \eqref{QTNY001} that for a sufficiently small $\varepsilon>0$,
\begin{align*}
F(t_{0})\leq C\delta^{d}(|\phi^{j}(z',h(z'))|^{2}\delta^{2-4/m}+|\nabla_{x'}\phi^{j}(z',h(z'))|^{2}+\delta^{2}\|\phi^{j}\|_{C^{2}}^{2}).
\end{align*}

\noindent{\bf Part 3.}
Claim that for $j=1,2,...,d$, $z\in\Omega_{R}$,
\begin{align}\label{AQ3.052}
|\nabla w_{j}(z)|\leq&C(|\phi^{j}(z',h(z'))|\delta^{1-2/m}+|\nabla_{x'}(\phi^{j}(z',h(z')))|)+C\delta\|\phi^{j}\|_{C^{2}}.
\end{align}

Making use of a change of variables in the thin gap $\Omega_{\delta}(z')$ as follows:
\begin{align*}
\begin{cases}
x'-z'=\delta y',\\
x_{d}=\delta y_{d},
\end{cases}
\end{align*}
we rescales $\Omega_{\delta}(z')$ into $Q_{1}$, where, for $0<r\leq 1$,
\begin{align*}
Q_{r}=\left\{y\in\mathbb{R}^{d}\,\Big|\,\frac{1}{\delta}h(\delta y'+z')<y_{d}<\frac{\varepsilon}{\delta}+\frac{1}{\delta}h_{1}(\delta y'+z'),\;|y'|<r\right\}.
\end{align*}
Denote the top and bottom boundaries of $Q_{r}$, respectively, by
\begin{align*}
\widehat{\Gamma}^{+}_{r}=&\left\{y\in\mathbb{R}^{d}\,\Big|\,y_{d}=\frac{\varepsilon}{\delta}+\frac{1}{\delta}h_{1}(\delta y'+z'),\;|y'|<r\right\},
\end{align*}
and
\begin{align*}
\widehat{\Gamma}^{-}_{r}=&\left\{y\in\mathbb{R}^{d}\,\Big|\,y_{d}=\frac{1}{\delta}h(\delta y'+z'),\;|y'|<r\right\}.
\end{align*}
Analogously as in \eqref{KHW01}, we obtain that for $x\in\Omega_{\delta}(z')$,
\begin{align*}
|\delta(x')-\delta(z')|
\leq&2^{m-1}\kappa_{3}\delta(\delta^{m-1}+|z'|^{m-1})\notag\\
\leq&2^{m}\kappa_{3}\max\{1,\kappa_{1}^{1/m-1}\}\delta^{2-1/m},
\end{align*}
which yields that
\begin{align*}
\left|\frac{\delta(x')}{\delta(z')}-1\right|\leq2^{m}\max\{1,\kappa_{1}^{1/m-1}\}\kappa_{2}^{1-1/m}\kappa_{3}R^{m-1}.
\end{align*}
This, in combination with the fact that $R$ is a small positive constant, leads to that $Q_{1}$ is of nearly unit size as far as applications of Sobolev embedding theorems and classical $L^{p}$ estimates for elliptic systems are concerned.

Let
\begin{align*}
W_{j}(y',y_d):=w_{j}(\delta y'+z',\delta y_d),\quad \widetilde{V}_{j}(y',y_d):=\tilde{v}_{j}(\delta y'+z',\delta y_d).
\end{align*}
Therefore, $W_{j}(y)$ verifies
\begin{align*}
\begin{cases}
\mathcal{L}_{\lambda,\mu}W_{j}=-\mathcal{L}_{\lambda,\mu}\widetilde{V}_{j},&
\hbox{in}\  Q_{1},  \\
W_{j}=0, \quad&\hbox{on} \ \widehat{\Gamma}^{\pm}_{1}.
\end{cases}
\end{align*}
Since $W_{j}=0$ on $\widehat{\Gamma}^{\pm}_{1}$, then we see from Poincar\'{e} inequality that
$$\|W_{j}\|_{H^1(Q_1)}\leq C\|\nabla W_{j}\|_{L^2(Q_1)}.$$
This, together with the Sobolev embedding theorem and classical $W^{2, p}$ estimates for elliptic systems, yields that for some $p>d$,
\begin{align*}
\|\nabla W_{j}\|_{L^{\infty}(Q_{1/2})}\leq C\|W_{j}\|_{W^{2,p}(Q_{1/2})}\leq C(\|\nabla W_{j}\|_{L^{2}(Q_{1})}+\|\mathcal{L}_{\lambda,\mu}\widetilde{V}_{j}\|_{L^{\infty}(Q_{1})}).
\end{align*}
Consequently, rescaling back to $w_{j}$ and $\tilde{v}_{j}$, we have
\begin{align}\label{ADQ601}
\|\nabla w_{j}\|_{L^{\infty}(\Omega_{\delta/2}(z'))}\leq\frac{C}{\delta}\left(\delta^{1-d/2}\|\nabla w_{j}\|_{L^{2}(\Omega_{\delta}(z'))}+\delta^{2}\|\mathcal{L}_{\lambda,\mu}\tilde{v}_{j}\|_{L^{\infty}(\Omega_{\delta}(z'))}\right).
\end{align}
From \eqref{QQ.103} and \eqref{step2}, it follows that for $z'\in B'_{R}$,
\begin{align*}
&\delta^{-\frac{d}{2}}\|\nabla w_{j}\|_{L^{2}(\Omega_{\delta}(z'))}+\delta\|\mathcal{L}_{\lambda,\mu}\tilde{v}_{j}\|_{L^{\infty}(\Omega_{\delta}(z'))}\notag\\
\leq& C(|\phi^{j}(z',h(z'))|\delta^{1-2/m}+|\nabla_{x'}\phi^{j}(z',h(z'))|)+C\delta\|\phi^{j}\|_{C^{2}}.
\end{align*}
This, in combination with \eqref{ADQ601}, yields that \eqref{AQ3.052} holds. The proof is complete.


\noindent{\bf{\large Acknowledgements.}} C. Miao was supported by the National Key Research and Development Program of China (No. 2020YFA0712900) and NSFC Grant 11831004. Z. Zhao was partially supported by CPSF (2021M700358).

\bibliographystyle{plain}

\end{document}